\documentclass[11pt, reqno]{amsart}
\usepackage{amssymb, amsthm, amsmath, amsfonts,mathtools}
\usepackage{array, epsfig}
\usepackage{bbm}

\usepackage{hyperref}
\usepackage[numbers,square]{natbib}
\usepackage{color}
\allowdisplaybreaks

\numberwithin{equation}{section}

  \evensidemargin
\oddsidemargin
\newcommand{\stirling}[2]{\genfrac{[}{]}{0pt}{}{#1}{#2}}
\newcommand{\stirlingb}[2]{\genfrac{\{}{\}}{0pt}{}{#1}{#2}}

\newcommand{\Res}{\mathop{\mathrm{Res}}\nolimits}

\newcommand{\ii}{{\rm{i}}}

\newcommand{\mA}{{\tilde a}}
\newcommand{\mB}{{\tilde b}}
\newcommand{\mtA}{{\tilde {\mathfrak{a}}}}
\newcommand{\mtB}{{\tilde {\mathfrak{b}}}}

\newcommand{\lA}{a}
\newcommand{\lB}{b}
\newcommand{\rA}{\mathfrak{a}}
\newcommand{\rB}{\mathfrak{b}}

\newcommand{\bB}{\mathbb{B}}

\newcommand{\bS}{\mathbb{S}}
\newcommand{\bI}{\mathbb{I}}
\newcommand{\bJ}{\mathbb{J}}

\def\dint{\textup{d}}
\newcommand{\Mod}[1]{\ (\mathrm{mod}\ #1)}

\newcommand{\E}{\mathbb E}

\newcommand{\R}{\mathbb{R}}
\newcommand{\N}{\mathbb{N}}

\newcommand{\C}{\mathbb{C}}
\newcommand{\Z}{\mathbb{Z}}

\renewcommand{\P}{\mathbb{P}}

\renewcommand{\Re}{\operatorname{Re}}

\newcommand{\Vol}{\mathop{\mathrm{Vol}}\nolimits}

\newcommand{\sgn}{\mathop{\mathrm{sgn}}\nolimits}
\newcommand{\conv}{\mathop{\mathrm{conv}}\nolimits}

\newcommand{\eps}{\varepsilon}

\newcommand{\bsl}{\backslash}
\newcommand{\ind}{\mathbbm{1}}

\newcommand{\dd}{{\rm d}}
\newcommand{\eee}{{\rm e}}

\theoremstyle{plain}
\newtheorem{theorem}{Theorem}[section]
\newtheorem{lemma}[theorem]{Lemma}

\newtheorem{proposition}[theorem]{Proposition}

\theoremstyle{definition}

\newtheorem{example}[theorem]{Example}

\theoremstyle{remark}
\newtheorem{remark}[theorem]{Remark}

\begin{document}

\author{Zakhar Kabluchko}
\address{Zakhar Kabluchko: Institut f\"ur Mathematische Stochastik,
Westf\"alische Wilhelms-Universit\"at M\"unster,
Orl\'eans-Ring 10,
48149 M\"unster, Germany}
\email{zakhar.kabluchko@uni-muenster.de}

\title[Angles of Random Simplices]{Angles of Random Simplices and Face Numbers of Random Polytopes}



\begin{abstract}
Pick $d+1$ points uniformly at random on the unit sphere in $\mathbb R^d$. What is the expected value of  the angle sum of the simplex spanned by these points? Choose $n$ points uniformly at random in the $d$-dimensional ball. What is the expected number of faces of their convex hull?  We answer these and some other, seemingly unrelated, questions of stochastic geometry.  To this end, we compute expected internal angles of random simplices whose vertices are independent random points sampled from one of the following  $d$-dimensional distributions:
(i) the beta distribution with the density proportional to $(1-\|x\|^2)^{\beta}$, where $x$ is belongs to the unit ball in $\mathbb R^d$;
(ii) the beta' distribution with the density proportional to $(1+\|x\|^2)^{-\beta}$ on the whole of $\mathbb R^d$.
These results imply explicit formulae for the expected face numbers of the following random polytopes:  (a) the typical Poisson-Voronoi cell; (b)  the zero cell of the Poisson hyperplane tessellation; (c) beta and beta' polytopes defined as convex hulls of i.i.d.\ samples from the corresponding distributions.
\end{abstract}

\maketitle
\section{Main results}
\subsection{Introduction}\label{subsec:intro}
The aim of the present paper is to compute  expectations of several quantities appearing in stochastic geometry. In particular, we solve the following problems:
\vspace*{0.5mm}
\begin{quote}
\textsc{Problem A.} 
Compute the expected number of $k$-dimensional faces of the convex hull of $n$ points sampled uniformly at random in the unit ball (or on the unit sphere) in $\R^d$. 
\end{quote}
\vspace*{0.5mm}
\begin{quote}
\textsc{Problem B.} 
Compute the expected internal angles of a simplex spanned by $d+1$ points sampled uniformly at random in the unit ball (or on the unit sphere) in $\R^{d}$.
\end{quote}
\vspace*{0.5mm}
\begin{quote}
\textsc{Problem C.} Compute the expected number of $k$-dimensional faces of the typical cell of the Poisson-Voronoi tessellation in $\R^d$.
\end{quote}
\vspace*{0.5mm}
\begin{quote}
\textsc{Problem D.} Compute the expected number of $k$-dimensional faces of the zero cell of the Poisson hyperplane tessellation in $\R^d$.
\end{quote}
\vspace*{0.5mm}



Let us start with a discussion of Problem A. Pick $n$ points $X_1,\ldots,X_n$ uniformly at random in the unit ball in $\R^d$, where $n\geq d+1$.
Let $P=[X_1,\ldots,X_n]$ denote their convex hull.  Problem~A asks to determine the expected $f$-vector of the polytope $P$, i.e.\ the vector whose entries are the expectations $\E f_0(P), \ldots, \E f_{d-1}(P)$,  where $f_0(P)$ is the number of vertices of $P$, $f_1(P)$ is the number of edges of $P$, and, more generally, $f_k(P)$ is the number of $k$-dimensional faces of $P$.
If the uniform distribution of $X_1,\ldots,X_n$ is replaced by the $d$-dimensional standard Gaussian one, a formula for the complete expected $f$-vector  in terms of internal and external angles of the regular simplex has been obtained by Affentranger and Schneider~\cite{AS92} (whose result has to be combined with the observation of Baryshnikov and Vitale~\cite{baryshnikov_vitale}).
For distributions other than the Gaussian one, satisfactory results are available only  in a few special cases.
For example, Buchta and M\"uller~\cite{buchta_mueller} derived an explicit formula for the expected number of facets (that is, faces of dimension $k=d-1$) of $P$. In the case when the points are chosen uniformly on the sphere, a similar formula was obtained by Buchta, M\"uller and Tichy~\cite{buchta_mueller_tichy}, and more general distributions were treated by Affentranger~\cite{affentranger}. Since all faces of $P$ are simplices with probability $1$, relation $d f_{d-1}(P) = 2 f_{d-2}(P)$ also yields the value of $\E f_{d-2}(P)$.  The expected volume of $P$ is known explicitly~\cite{affentranger}, see also~\cite[Corollary~3.9]{beta_polytopes_temesvari}, and a trick called Efron's identity~\cite{efron} yields a formula for $\E f_0(P)$; see, e.g., \cite[Proposition~3.4]{beta_polytopes_temesvari}. A special case of this approach with $n= d+2$ is related to the classical Sylvester four-point problem which asks to compute the probability that four points chosen uniformly at random in a planar region have a convex hull which is a triangle. If the region is a disk, the answer is $35/(12\pi^2)$. Kingman~\cite{kingman_secants} considered a generalization of Sylvester's problem in which $n=d+2$ points are chosen uniformly at random in a $d$-dimensional unit ball and computed explicitly the probability that their convex hull is a simplex. It is an exercise to check that both, the expected volume of the simplex spanned by the first $d+1$ points and the expected number of vertices of $P$  in the special case when $n=d+2$, can be expressed through this probability.  We shall not attempt to review the vast subject of random polytopes (and, more generally, geometric probability) that started with Sylvester's question and refer to~\cite{SW08} for a monograph treatment, \cite{buchta_polyeder,schneider_polytopes,hug_rev,calka_classical_problems} for overview articles and to~\cite{buchta_zufallspolygone,buchta_reitzner,ReitznerCombinatorialStructure} for some highlights.

Previous works on random polytopes (including the papers cited above) used tools from stochastic and integral geometry~\cite{SW08}, most notably the Blaschke-Petkantschin formulae~\cite[Section~7.2]{SW08}, \cite{miles}.  In the present paper, we shall  completely solve Problems~A, B, C, D and some other problems of stochastic geometry by combining the integral-geometric approach of the paper~\cite{beta_polytopes} with the combinatorial analysis of certain new special functions related to the Schl\"afli function and having properties somewhat similar to the properties of the Stirling numbers of both kinds. The Schl\"afli function expresses the volumes of regular spherical simplices and appears as a limiting case of the functions studied here.

First of all, it is necessary to extend Problem~A to a more general class of random polytopes. We shall be interested in the so-called beta and beta' polytopes defined as the convex hulls of random samples from the following two families of distributions: the beta distributions with the density proportional to
$
(1-\|x\|^2)^{\beta}
$
on the unit ball in $\R^d$, and the beta' distributions whose density is proportional to $(1+\|x\|^2)^{-\beta}$ on the whole of $\R^d$.
Both classes of distributions were introduced by Miles~\cite[Section~12]{miles}. They are characterized by a remarkable property, called the canonical decomposition, which  was discovered by Ruben and Miles~\cite{ruben_miles}. In~\cite{beta_polytopes}, this decomposition was combined with other tools from stochastic geometry to express the complete expected $f$-vectors of beta and beta' polytopes in terms of two sorts of quantities: the expected external angle sums of random simplices sampled from the beta and beta' distributions, and the expected internal angle sums of the same simplices. Precise definitions of these quantities,  denoted by $\bI_{n,k}(\alpha)$ and $\bJ_{n,k}(\beta)$ in the beta case and $\tilde \bI_{n,k}(\alpha)$ and $\tilde \bJ_{n,k}(\beta)$ in the beta' case, will be given below. Moreover, it was shown in~\cite{beta_polytopes} that several problems of stochastic geometry can be solved in terms of  $\bI_{n,k}(\alpha)$ and $\bJ_{n,k}(\beta)$ or $\tilde \bI_{n,k}(\alpha)$ and $\tilde \bJ_{n,k}(\beta)$, most notably the determination of the expected $f$-vectors of the typical Poisson-Voronoi cell and the zero cell of the Poisson hyperplane tessellation, as well as the constants appearing in the work of Reitzner~\cite{ReitznerCombinatorialStructure} on random polytopes approximating smooth convex bodies.

While the external angle sums $\bI_{n,k}(\alpha)$ and $\tilde \bI_{n,k}(\alpha)$ are easy to evaluate (which has already been done in~\cite{beta_polytopes}), no formula for the internal angle sums $\bJ_{n,k}(\beta)$ and $\tilde \bJ_{n,k}(\beta)$ has been provided in~\cite{beta_polytopes}. In~\cite{kabluchko_algorithm}, we described a recursive scheme which allows to compute $\bJ_{n,k}(\beta)$ and $\tilde \bJ_{n,k}(\beta)$ in finite time for every given values of $n,k,\beta$. The aim of the present paper is to solve this recursive scheme, thus providing an explicit formula for the internal angle sums and, consequently,  completely solving Problems~A, B, C, D and some other problems of stochastic geometry.

We always aim at obtaining reasonably simple formulae for the quantities of interest.
Let us agree that the formulae are allowed to contain elementary functions, as well as multiple sums and multiple integrals thereof (including the Gamma function), but the multiplicity of sums and integrals is not allowed to depend on the parameters of the problem such as $d,n,k$. All formulae stated below can  easily be evaluated by  computer algebra systems.


\subsection{Internal angles}
Let us start with the necessary definitions. A $d$-dimensional \textit{simplex} in $\R^d$ is defined as the convex hull $[x_1,\ldots,x_n]$ of $n=d+1$ points $x_1,\ldots,x_n\in \R^{d}$ that are not contained in a common affine hyperplane.
The \textit{internal angle} of the simplex $T := [x_1,\ldots,x_n]$ at its face $F := [x_{i_1},\ldots,x_{i_k}]$, where $1\leq i_1<\ldots< i_k\leq n$, can be defined as
$$
\beta(F, T)
:=
\lim_{r\downarrow 0} \frac{\Vol_{d}(\bB_r^d(z)\cap T) }{\Vol_{d}(\bB_r^d(z))},
$$
where $\Vol_d$ denotes the Lebesgue measure in $\R^d$, $\bB_r^d(z)$ is the $d$-dimensional ball with radius $r>0$ centered at $z$, and $z$ is any point in $F$ not belonging to a face of smaller dimension. Note that we have chosen the units of measurement for angles in such a way that the full-space angle equals $1$. For each $k\in \{1,\ldots,n\}$ we let $\sigma_k(T)$ denote the sum of internal angles of $T$ at all its $k$-vertex faces, that is
$$
\sigma_k(T) = \sum_{\substack{1\leq i_1<\ldots< i_k\leq n\\F := [X_{i_1},\ldots,X_{i_k}]}} \beta(F, T).
$$
The sum of the measures of angles in any plane triangle is constant, but in dimensions $d\geq 3$ the angle sums $\sigma_k(T)$ are not constant except for the trivial values $\sigma_{d+1}(T) = 1$ and $\sigma_{d}(T) = \frac 12 (d+1)$. The range of all possible values of $\sigma_k(T)$, for every fixed $k$ and $d$,  has been determined by Perles and Shephard~\cite[pp.~208--209]{perles_shephard}.  It is therefore natural to ask what are the \textit{average} values of these quantities, which leads to Problem~B stated above. 
The aim of the present paper is to solve a more general version of this problem in which the vertices of the random simplex  have a beta or a beta' distribution.

\subsection{Beta simplices and their internal angles}\label{subsec:beta_simplices}
A random vector in $\R^d$ is said to have a $d$-dimensional \textit{beta distribution} with parameter $\beta>-1$ if its Lebesgue density is given by
\begin{equation}\label{eq:def_f_beta}
f_{d,\beta}(x)=c_{d,\beta} \left( 1-\left\| x \right\|^2 \right)^\beta\ind_{\{\|x\| <  1\}},\qquad x\in\R^d,\qquad
c_{d,\beta}= \frac{ \Gamma\left( \frac{d}{2} + \beta + 1 \right) }{ \pi^{ \frac{d}{2} } \Gamma\left( \beta+1 \right) },
\end{equation}
where $\|x\| = (x_1^2+\ldots+x_d^2)^{1/2}$ denotes the Euclidean norm of the vector $x= (x_1,\ldots,x_d)\in\R^d$. Since the special case of the normalizing constant $c_{d,\beta}$ with $d=1$ will frequently appear below, we introduce the shorthand
\begin{equation}\label{eq:c_beta}
c_{\beta} := c_{1,\beta} = \frac{ \Gamma\left(\beta + \frac{3}{2} \right) }{  \sqrt \pi\, \Gamma (\beta+1)}.
\end{equation}

\begin{example}
The uniform distribution on the unit ball $\bB^{d}:=\{x\in\R^d\colon \|x\| \leq 1\}$ coincides with the beta distribution with parameter $\beta = 0$. The uniform distribution on the unit sphere $\bS^{d-1} := \{x\in \R^d\colon \|x\| = 1\}$ is the weak limit of the beta distribution as $\beta \downarrow -1$, see~\cite[p.~102]{beta_polytopes_temesvari}. The standard Gaussian distribution on $\R^d$ is the limit of the appropriately rescaled beta distribution as $\beta \to +\infty$; see~\cite[Lemma~1.1]{beta_polytopes}. All formulae of the present paper apply to the uniform distribution on $\bS^{d-1}$ by taking $\beta= -1$ and to the standard Gaussian distribution by letting $\beta\to+\infty$.
\end{example}

Let $X_1,\ldots,X_{n}$ be independent random points in $\R^{n-1}$ sampled from the beta distribution $f_{n-1,\beta}$, where $\beta\geq -1$. For $\beta=-1$, this just means that $X_1,\ldots,X_n$ have uniform distribution on the unit sphere.  The convex hull $[X_1,\ldots,X_n]$ is referred to as the $(n-1)$-dimensional \textit{beta simplex}. Miles~\cite{miles} completely characterized the distribution of the volume of the beta simplex by computing its moments.
Our aim is to compute explicitly the expected internal angles of these random simplices, denoted by
\begin{equation}
J_{n,k}(\beta) := \E \beta ([X_1,\ldots,X_{k}], [X_1,\ldots,X_{n}]),
\end{equation}
for all integer $n\geq 2$ and $k\in \{1,\ldots,n\}$. By exchangeability, the expected sum of internal angles at all $k$-vertex faces of the beta simplex $[X_1,\ldots,X_n]$ is then
\begin{equation}
\bJ_{n,k}(\beta) :=  \binom {n}{k}  J_{n,k}(\beta),
\end{equation}
Clearly, $\bJ_{n,n}(\beta) = J_{n,n}(\beta) = 1$ and $\bJ_{n,n-1}(\beta) = n J_{n,n-1}(\beta) = \frac n2$ for all $n\geq 2$. It is natural to put $\bJ_{1,1}(\beta) = J_{1,1}(\beta) = 1$.  As another trivial example, we have $\bJ_{3,1} (\beta) = 1/2$  because the sum of angles in any plane triangle equals half the full angle.
In~\cite{kabluchko_angles}, we computed $\bJ_{n,k}(\beta)$ for $n=4$ and $n=5$ (which corresponds to simplices in dimensions $d=3$ and $d=4$). Some non-trivial values obtained there include
$$
\bJ_{4,1}(-1) = \frac 18,
\quad
\bJ_{5,1}(-1) = \frac{539}{288\pi^2}-\frac 16,
\quad
\bJ_{4,1}(0) = \frac {401}{2560},
\quad
\bJ_{5,1}(0) = \frac{1692197}{846720 \pi^2}-\frac 16.
$$
As explained at the end of~\cite{kabluchko_angles}, the method used there cannot be extended to higher dimensions. In~\cite{kabluchko_algorithm}, we derived recursive relations for $\bJ_{n,k}(\beta)$ which  yielded a complicated formula for these quantities. Although this formula allows us to compute $\bJ_{n,k}(\beta)$ in finitely many steps, its complexity rapidly increases with $n, k$ and $\beta$,  and it cannot be considered satisfactory.
In the present paper, we prove the following explicit formula for the quantities $\bJ_{n,k}(\beta)$, where $\beta\geq -1$.
\begin{theorem}\label{theo:bJ_formula_integral}
Let $n\geq 3$ be integer and $k\in \{1,\ldots,n\}$.  For all $\alpha \geq n-3$ we have
\begin{equation}\label{eq:J_nk_integral}
\bJ_{n,k}\left(\frac{\alpha - n + 1}{2}\right)
=
\binom nk \int_{-\pi/2}^{+\pi/2} c_{\frac{\alpha n}2} (\cos x)^{\alpha n + 1}
\left(\frac 12  + \ii \int_0^x  c_{\frac{\alpha-1}{2}} (\cos y)^{-\alpha-1}\dd y \right)^{n-k} \dd x.
\end{equation}
\end{theorem}
By making  the change of variables  $t=\tanh x$, $s=\tanh y$ in~\eqref{eq:J_nk_integral} and then substituting $t=\sin u$, $s=\sin v$ in the resulting formula, it is possible to obtain the following equivalent versions of~\eqref{eq:J_nk_integral}:
\begin{align}
\bJ_{n,k}\left(\frac{\alpha - n + 1}{2}\right)
&=
\binom nk \int_{-1}^{+1} c_{\frac{\alpha n}2} (1-t^2)^{\frac{\alpha n}2}
\left(\frac 12  + \ii \int_0^t  c_{\frac{\alpha-1}{2}} (1-s^2)^{-\frac \alpha2 -1}\dd s \right)^{n-k} \dd t \label{eq:bJ_nk_equiv1}
\\
&=
\binom nk \int_{-\infty}^{+\infty} c_{\frac{\alpha n}2} (\cosh u)^{-\alpha n - 2}
\left(\frac 12  + \ii \int_0^u  c_{\frac{\alpha-1}{2}} (\cosh v)^{\alpha}\dd v \right)^{n-k} \dd u. \label{eq:bJ_nk_equiv2}
\end{align}
If $\alpha$ is integer, the inner and then the outer integrals in these formulae can be evaluated exactly by standard integration techniques.  This gives explicit expressions for $\bJ_{n,k}(\beta)$ is $\beta:= (\alpha-n+1)/2\geq -1$ is integer or half-integer; see~\cite{kabluchko_algorithm} for the tables of $\bJ_{n,k}(\beta)$ for some values of $\beta$ including $\beta = 0$ and $\beta=-1$. In some cases, it is possible to evaluate the integral more efficiently  with the help of residue calculus.

\begin{theorem}\label{theo:bJ_formula_residue}
Let $n\geq 3$, $k\in \{1,\ldots,n\}$ and $\alpha \geq n-3$ be integer. If either (i) $\alpha$ is even and $n-k$ is odd, or (ii) both $\alpha$ and $n$ are odd, then
\begin{equation}\label{eq:J_nk_residue}
\bJ_{n,k}\left(\frac{\alpha - n + 1}{2}\right)
=
\binom nk
c_{\frac{\alpha n}2} \left(c_{\frac{\alpha-1}{2}}\right)^{n-k}
\pi  \Res\limits_{x=0} \left[\frac{\left(\int_{0}^x (\sin y)^{\alpha} \dd y\right)^{n-k}}{(\sin x)^{\alpha n +2}}\right].
\end{equation}
\end{theorem}
Here, $\Res_{x=x_0}f(x)$ denotes the residue of the meromorphic function $f(x)$ at $x_0\in\C$, that is the coefficient of $(x-x_0)^{-1}$ in the Laurent expansion of $f(x)$ around $x_0$.
There are some cases not covered by this theorem. The best result we have in the case when (iii) $\alpha$ is odd and $n$ is even are Equations~\eqref{eq:J_nk_integral}, \eqref{eq:bJ_nk_equiv1}, \eqref{eq:bJ_nk_equiv2}. In the case when (iv) both $\alpha$ and $n-k$ are even, we can compute $\bJ_{n,k}(\frac{\alpha - n + 1}{2})$ by combining case (i) of Theorem~\ref{theo:bJ_formula_residue} with the Poincar\'e relations. These are linear relations between angle sums of any deterministic simplex which, in the special case of the beta simplex, imply that
$$
\sum_{k=m}^n (-1)^k \binom{k}{m} \bJ_{n,k}(\beta) = (-1)^n\bJ_{n,m}(\beta),
\qquad \bJ_{n,0}(\beta) := 0,
$$
for all $m\in\{0,\ldots,n\}$ and $\beta\geq -1$. As will be explained in Section~\ref{sec:dehn_sommerville}, these relations allow to express the odd-positioned entries of the vector  $(\bJ_{n,1}(\beta), \ldots, \bJ_{n,n}(\beta))$ as linear combinations of the even-positioned ones, and vice versa.  It is possible to transform these considerations into the following ``ugly'' formula.
\begin{proposition}\label{prop:ugly_residue}
Let $n\geq 3$, $k\in \{1,\ldots,n\}$ and $\alpha \geq n-3$ be integer.
If  both $\alpha$ and $n-k$ are even, then
\begin{equation}\label{eq:J_nk_residue_ugly}
\bJ_{n,k}\left(\frac{\alpha - n + 1}{2}\right)
=
(-1)^{\frac{n-k}2}
\frac{n!}{k!} \pi c_{\frac {\alpha n}{2}} \cdot  [u^{n-k}x^{-1}] \left( \frac{ \sin \left(u c_{\frac{\alpha-1}{2}}\int_{0}^x (\sin y)^{\alpha} \dd y\right)}{\tan\left(\frac u2\right)(\sin x)^{\alpha n +2}}\right),
\end{equation}
where $[u^{n-k}x^{-1}]g(u,x)$ denotes the coefficient of $u^{n-k}x^{-1}$ in the series expansion of the function $g$ around $(0,0)$.
\end{proposition}

In Theorems~3.8 and 3.9 of~\cite{kabluchko_algorithm} we gave explicit formulae for $\bJ_{n,1}(-\frac 12)$ and $\bJ_{n,1}(\frac 12)$ with arbitrary $n\in\N$ in terms of products of Gamma functions. Stated in different terms, the latter value can be found in the work of Hug~\cite[Corollary~7.1 and p.~209]{hug_rev}. The proofs in these works used only stochastic geometry. Using Theorem~\ref{theo:bJ_formula_residue}, we can generalize these formulae to $\bJ_{n,1}(m- \frac 12)$ with arbitrary $m\in \N_0$. The formulae are as explicit as possible but involve rational functions whose complexity increases rapidly with $m$.

\begin{proposition}\label{prop:J_n_1_half_integer}
There is a sequence of rational functions $R_0(x), R_1(x), \ldots$ with rational coefficients such that for every $n\geq 3$ and $m\in\N_0$, we have
$$
\bJ_{n,1} \left(m - \frac 12\right) =
c_{\frac{(n+2m-2) n}2} \left(c_{\frac{n+2m-3}{2}}\right)^{n-1}
\frac{\pi n \, R_m(n)}{(n+2m-1)^{n-1}}.
$$
The firs terms are $R_0(n) = 1$, $R_1(n) = \frac{n^2+n+2}{2(n+3)}$, $R_2(n)=\frac{n^5+15 n^4+81 n^3+225 n^2+326 n+216}{8 (n+5)^2 (n+7)},\ldots$.
\end{proposition}

The arithmetic properties of the quantities $\bJ_{n,k}(\beta)$ for integer and half-integer values of  $\beta$ are summarized in the following theorem. It turns out that these are either rational numbers, or polynomials in $\pi^{-2}$ with rational coefficients which in some cases reduce to a single rational multiple of a power of $\pi^{-2}$.
\begin{theorem}
\label{theo:arithm_J}
Let $\beta\geq -1$ be integer or half-integer. Let also $n\in\N$ and $k\in \{1,\ldots,n\}$.
\begin{itemize}
\item[(a)] If $2\beta + n$ is even, then $\bJ_{n,k}(\beta)$ is a rational number.
\item[(b)] If both $2\beta + n$  and $n-k$ are odd, then $\bJ_{n,k}(\beta)$ is a number of the form $q\pi^{-(n-k-1)}$ with some rational $q$.
\item[(c)] If $2\beta + n$ is odd and $n-k$ is even, then $\bJ_{n,k}(\beta)$ can be expressed as  $q_0 + q_2 \pi^{-2} + q_4 \pi^{-4} + \ldots + q_{n-k} \pi^{-(n-k)}$,
     where the numbers $q_{2j}$ are rational.
\end{itemize}
\end{theorem}
\begin{proof}
Parts (a) and (c) were established in~\cite{kabluchko_algorithm}.  Part (b), which was only  conjectured there,  follows from  Theorem~\ref{theo:bJ_formula_residue} because the residue appearing in~\eqref{eq:J_nk_residue} is a rational number.
\end{proof}

\subsection{Beta' simplices and their internal angles}\label{subsec:beta_simplices_prime}
A random vector in $\R^d$ has \textit{beta' distribution} with parameter $\beta>d/2$ if its Lebesgue density is given by
\begin{equation}\label{eq:def_f_beta_prime}
\tilde{f}_{d,\beta}(x)=\tilde{c}_{d,\beta} \left( 1+\left\| x \right\|^2 \right)^{-\beta},\qquad
x\in\R^d,\qquad
\tilde{c}_{d,\beta}= \frac{ \Gamma\left( \beta \right) }{\pi^{ \frac{d}{2} } \Gamma\left( \beta - \frac{d}{2} \right)}.
\end{equation}
We shall use parallel notation for beta and beta' distributions, with the quantities related to beta' distributions being always marked by a tilde. For the special case of the normalizing constant with $d=1$ we introduce the shorthand
\begin{equation}\label{eq:c_beta_tilde}
\tilde c_{\beta}
:=
\tilde{c}_{1,\beta}
=
\frac{ \Gamma(\beta)}{ \sqrt \pi\, \Gamma\left( \beta - \frac{1}{2}\right)}.
\end{equation}
Let $\tilde X_1,\ldots,\tilde X_{n}$ be independent random points in $\R^{n-1}$ sampled from the beta' distribution $\tilde f_{n-1,\beta}$, where $\beta > (n-1)/2$. Their convex hull $[\tilde X_1,\ldots,\tilde X_n]$ is called the $(n-1)$-dimensional \textit{beta' simplex}. Its expected internal angles are denoted by
\begin{equation}
\tilde J_{n,k}(\beta) := \E \beta ([\tilde X_1,\ldots,\tilde X_{k}], [\tilde X_1,\ldots,\tilde X_{n}],
\end{equation}
for all integer $n\geq 2$ and $k\in \{1,\ldots,n\}$. The expected sum of internal angles at all $k$-vertex faces of the beta' simplex is
\begin{equation}
\tilde \bJ_{n,k}(\beta) :=  \binom {n}{k}  \tilde J_{n,k}(\beta).
\end{equation}
By convention,  $\tilde J_{n,n}(\beta) = \tilde \bJ_{n,n}(\beta) = 1$ for all $n\in\N$.
For the quantities $\tilde \bJ_{n,k}(\beta)$, we shall prove the following explicit formula.
\begin{theorem}\label{theo:bJ_tilde_formula_integral}
Let $\alpha>0$. Then, for all $n\in\N$ and $k\in \{1,\ldots,n\}$ such that $\alpha n>1$ we have that
\begin{equation}\label{eq:J_nk_tilde_integral}
\tilde \bJ_{n,k}\left(\frac{\alpha + n - 1}{2}\right)
=
\binom nk \int_{-\pi/2}^{+\pi/2} \tilde c_{\frac{\alpha n}2} (\cos x)^{\alpha n - 2} \left(\frac 12  + \ii \int_0^x \tilde c_{\frac{\alpha+1}{2}}(\cos y)^{-\alpha}\dd y \right)^{n-k} \dd x.
\end{equation}
\end{theorem}
Making in~\eqref{eq:J_nk_tilde_integral} the substitutions $t=\tanh x$, $s=\tanh y$  and then $t=\sin u$, $s=\sin v$, we arrive at the following equivalent versions of~\eqref{eq:J_nk_tilde_integral}:
\begin{align}
\tilde \bJ_{n,k}\left(\frac{\alpha + n - 1}{2}\right)
&=
\binom nk \int_{-1}^{+1} \tilde c_{\frac{\alpha n}2} (1-t^2)^{\frac{\alpha n - 3}2} \left(\frac 12  + \ii \int_0^t \tilde c_{\frac{\alpha+1}{2}}(1-s^2)^{-\frac{\alpha+1}2}\dd s \right)^{n-k} \dd t
\label{eq:bJ_nk_tilde_equiv1}
\\
&=
\binom nk \int_{-\infty}^{+\infty} \tilde c_{\frac{\alpha n}2} (\cosh u)^{- (\alpha n - 1)} \left(\frac 12  + \ii \int_0^u \tilde c_{\frac{\alpha+1}{2}}(\cosh v)^{\alpha-1}\dd v \right)^{n-k} \dd u. \label{eq:bJ_nk_tilde_equiv2}
\end{align}
Again, it is easy to evaluate the integrals exactly if $\alpha$ is integer, which gives explicit expressions for $\tilde \bJ_{n,k}(\beta)$ if $\beta:= (\alpha+n-1)/2 > (n-1)/2$ is integer or half-integer. In some cases, it is possible to evaluate the integral by means of  the residue calculus as follows.
\begin{theorem}\label{theo:bJ_tilde_formula_residue}
Let  $\alpha\in\N$.  Then,  for all $n\in\N$ and $k\in \{1,\ldots,n\}$  such that $\alpha k$ is even, we have
\begin{equation}\label{eq:J_nk_tilde_res}
\tilde \bJ_{n,k}\left(\frac{\alpha + n-1}{2}\right)
=
\binom nk
\tilde c_{\frac{\alpha n}2} \left(\tilde c_{\frac{\alpha+1}{2}}\right)^{n-k}
\pi  \Res\limits_{x=0} \left[\frac{\left(\int_{0}^x (\sin y)^{\alpha-1} \dd y\right)^{n-k}}{(\sin x)^{\alpha n - 1}}\right].
\end{equation}
\end{theorem}

It is interesting to note that the validity domain of this formula is larger than in the beta case. Indeed, it gives the complete vector of expected internal angles if $\alpha$ is even.  In the case when $\alpha$ is odd, this theorem identifies $\tilde \bJ_{n,k}(\frac{\alpha + n-1}{2})$ only for even $k$. The values with odd $k$ can be obtained using the Poincar\'e relations
$$
\sum_{k=m}^n (-1)^k \binom{k}{m} \tilde \bJ_{n,k}(\beta) = (-1)^n \tilde \bJ_{n,m}(\beta),
\qquad
\tilde \bJ_{n,0}(\beta) := 0,
$$
for all $n\geq 2$,  $m\in \{0,\ldots,n\}$ and $\beta>\frac{n-1}{2}$. The best result we were able to derive in this way is as follows.

\begin{proposition}\label{prop:ugly_residue_tilde}
Let $\alpha\in\N$. Let also $n\in\N$ and $k\in \{1,\ldots,n\}$ be such that $\alpha k$ is odd. If $n\neq 1$ is odd, then
\begin{equation}\label{eq:J_nk_residue_ugly_tilde}
\tilde \bJ_{n,k}\left(\frac{\alpha +  n - 1}{2}\right)
=
(-1)^{\frac{n-k}2}
\frac{n!}{k!} \pi \tilde c_{\frac {\alpha n}{2}} \cdot  [u^{n-k}x^{-1}] \left( \frac{ \sin \left(u \tilde c_{\frac{\alpha+1}{2}}\int_{0}^x (\sin y)^{\alpha-1} \dd y\right)}{\tan\left(\frac u2\right)(\sin x)^{\alpha n -1}}\right).
\end{equation}
If $n$ is even, then
\begin{equation}\label{eq:J_nk_residue_ugly_tilde1}
\tilde \bJ_{n,k}\left(\frac{\alpha +  n - 1}{2}\right)
=
(-1)^{\frac{n-k-1}2}
\frac{n!}{k!} \pi \tilde c_{\frac {\alpha n}{2}} \cdot  [u^{n-k}x^{-1}] \left( \frac{ \cos \left(u \tilde c_{\frac{\alpha+1}{2}}\int_{0}^x (\sin y)^{\alpha-1} \dd y\right)}{\cot\left(\frac u2\right)(\sin x)^{\alpha n - 1}}\right).
\end{equation}
\end{proposition}

The arithmetic structure of $\tilde \bJ_{n,k}(\beta)$ for integer and half-integer $\beta$ is described in the following theorem.
\begin{theorem}
\label{theo:arithm_J_tilde}
Let $n\in\N$ and  $k\in \{1,\ldots,n\}$. Let also $\beta> \frac{n-1}{2}$ be integer or half-integer.
\begin{itemize}
\item[(a)] If $2\beta - n$ is odd, then $\tilde \bJ_{n,k}(\beta)$ is a rational number.
\item[(b)] If $2\beta - n$ is even and $k$ is even, then  $\tilde \bJ_{n,k}(\beta)$ has the form $q\pi^{-(n-k)}$ (if $n-k$ is even) or $q\pi^{-(n-k-1)}$ (if $n-k$ is odd) with some rational $q$.
\item[(c)] If $2\beta - n$ is even, then $\tilde \bJ_{n,k}(\beta)$ can be expressed as
$q_0 + q_2 \pi^{-2} + q_4 \pi^{-4} + \ldots + q_{n-k-1} \pi^{-(n-k-1)}$ (if $n-k$ is odd) or
$q_0 + q_2 \pi^{-2} + q_4 \pi^{-4} + \ldots + q_{n-k} \pi^{-(n-k)}$ (if $n-k$ is even),
where the coefficients $q_{2j}$ are rational numbers.
\end{itemize}
\end{theorem}
\begin{proof}
Parts (a) and (c) were established in~\cite{kabluchko_algorithm}.  Part (b), which was only  conjectured there,  follows from Equation~\eqref{eq:J_nk_tilde_res} of Theorem~\ref{theo:bJ_tilde_formula_residue} since the residue appearing there is a rational number.
\end{proof}

In general, it does not seem possible to simplify the residue in~\eqref{eq:J_nk_tilde_res}. However, if $\alpha$ and $k$ are fixed, we can show that it is essentially a polynomial in $n$.
\begin{proposition}\label{prop:J_n_tilde_polynomial}
Fix $\alpha\in\N$ and $k\in\N$ such that $\alpha k$ is even. Then, there is a polynomial $P_{\alpha,k}(n)$ with rational coefficients such that for all integer $n\geq k$,
$$
\tilde \bJ_{n,k}\left(\frac{\alpha + n-1}{2}\right)
=
\binom nk
\tilde c_{\frac{\alpha n}2} \left(\tilde c_{\frac{\alpha+1}{2}}\right)^{n-k}
\frac{\pi\, P_{\alpha,k}(n)}{\alpha^{n-k}} .
$$
For example, we have
\begin{align*}
&P_{1,2}(n) = 1, &&P_{1,4}(n)= \frac{n-1} 6,&& P_{1,6}(n) = \frac{5 n^2-8 n+3}{360},&&  P_{1,8}(n) = \frac{35 n^3-63 n^2+37 n-9}{45360},\\
&P_{2,1}(n) = 1, && P_{2,2}(n) = \frac n4,&& P_{2,3}(n) = \frac{n (n+1)}{32},&&
P_{2,4}(n) = \frac{n (n^2+3 n+2)}{384}.
\end{align*}
\end{proposition}

\section{Applications to random polytopes}\label{sec:appl}

In this section, we give applications of the above results to some problems of stochastic geometry including the determination of the expected $f$-vectors of two random polytopes: the zero cell of the Poisson hyperplane tessellation and the typical cell of the Poisson-Voronoi tessellation.  Recall that for a convex $d$-dimensional polytope $P$, we let $f_{\ell}(P)$ denote the number of $\ell$-dimensional faces of $P$, for all $\ell\in \{0,\ldots,d\}$.  The \textit{$f$-vector} of $P$ is then the vector $(f_\ell(P))_{\ell = 0}^{d-1}$. If $P$ is random, we are interested in its expectation, called the \textit{expected $f$-vector}. We start with the so-called Poisson polytopes which can be used to treat both models.

\subsection{Poisson polytopes}
For $\alpha>0$ let $\Pi_{d,\alpha}$ be a Poisson point process on $\R^d\backslash\{0\}$ with the following power-law intensity function:
$$
x \mapsto \|x\|^{-d-\alpha},\qquad x\in \R^d \bsl \{0\}.
$$
Let $\conv \Pi_{d,\alpha}$ denote the convex hull of the atoms of $\Pi_{d,\alpha}$. We call $\conv \Pi_{d,\alpha}$ the \textit{Poisson polytope}.  Even though the number of atoms of $\Pi_{d,\alpha}$ is a.s.\ infinite because the atoms cluster at the origin, it can be shown that $\conv \Pi_{d,\alpha}$ is a.s.\ a polytope; see~\cite[Corollary~4.2]{convex_hull_sphere}. Moreover, this polytope is simplicial (that is, all of its faces are simplices) and contains the origin in its interior, with probability $1$.

In~\cite[Theorem~1.21]{beta_polytopes}, the expected $f$-vector of the Poisson polytope $\conv\Pi_{d,\alpha}$ has been expressed through the quantities $\tilde \bJ_{m,\ell}(\gamma)$. Namely, for every $d\in\N$ and every $k\in \{1,\ldots,d\}$, the expected number of $(k-1)$-dimensional faces of $\conv \Pi_{d,\alpha}$ is given by
\begin{equation} \label{eq:E_f_k_beta_prime_to_poisson}
\E f_{k-1}(\conv \Pi_{d,\alpha})
=
2 \sum_{\substack{m\in \{k,\ldots,d\}\\ m\equiv d \Mod{2}}}
\tilde \bI_{\infty,m}(\alpha) \tilde \bJ_{m,k}\left(\frac{\alpha + m-1}{2}\right),
\end{equation}
where\footnote{To explain our notation, let us mention that it was shown in the proof of Theorem~1.21 in~\cite{beta_polytopes} that $\tilde \bI_{\infty,m}(\alpha)= \lim_{n\to\infty} \tilde \bI_{n,m}(\alpha)$, for some  quantities $\tilde \bI_{n,m}(\alpha)$ to be defined below.}
\begin{equation}\label{eq:poisson_voronoi_f_k_addition}
\tilde \bI_{\infty,m}(\alpha)
:=
\frac{\tilde c_{\frac{\alpha m+1}{2}}}{\tilde c_{\frac{\alpha+1}{2}}^m} \cdot \frac{\alpha^{m-1}}{m}
=
\frac{\Gamma({\frac{m\alpha  + 1} 2})}{\Gamma(\frac {m\alpha}2)} \left(\frac {\Gamma(\frac \alpha 2)} {\Gamma(\frac{\alpha+1} 2)}\right)^{m}
\frac{(\sqrt{\pi}\alpha)^{m-1}}{m}.
\end{equation}
Using our results on the quantities $\tilde \bJ_{m,\ell}(\gamma)$ we are able to prove the following explicit formula for the expected $f$-vector of $\conv \Pi_{d,\alpha}$.
\begin{theorem}\label{theo:poisson_polyhedra}
Let $d\in\N$ and $\alpha>0$. Then, for every $k\in \{1,\ldots,d\}$ we have
\begin{multline*}
\E f_{k-1}(\conv \Pi_{d,\alpha})
\\
=
\binom{d}{k}
\cdot
\left(\frac{\alpha}{\tilde c_{\frac{\alpha+1}{2}}}\right)^d
\cdot
\frac 1 \pi
\int_{-\infty}^{+\infty} (\cosh u)^{- (\alpha d + 1)} \left(\frac 12  + \ii \int_0^u \tilde c_{\frac{\alpha+1}{2}}(\cosh v)^{\alpha-1}\dd v \right)^{d-k} \dd u.
\end{multline*}
If, additionally,  $\alpha\in\N$ and $k\in \{1,\ldots,d\}$ are such that $\alpha k$ is even, then
\begin{equation}\label{eq:f_vect_poisson_poly_residue}
\E f_{k-1}(\conv \Pi_{d,\alpha})
= \alpha^{d}  \binom{d}{k} \left(\frac{\sqrt \pi\, \Gamma(\frac \alpha2)}{\Gamma(\frac{\alpha+1}{2})}\right)^k
\Res\limits_{x=0} \left[\frac{\left(\int_{0}^x (\sin y)^{\alpha-1} \dd y\right)^{d-k}}{(\sin x)^{\alpha d+1}}\right].
\end{equation}
\end{theorem}
Equation~\eqref{eq:f_vect_poisson_poly_residue} identifies the complete expected $f$-vector $(\E f_{\ell}(\conv \Pi_{d,\alpha}))_{\ell=0}^{d-1}$  if $\alpha$ is even. If $\alpha$ is odd, it gives only the entries with odd $\ell$. However, since the polytope $\conv \Pi_{d,\alpha}$ is simplicial a.s., the remaining entries are uniquely determined by the Dehn-Sommerville relations; see Section~\ref{sec:dehn_sommerville}. In the same way as in Proposition~\ref{prop:ugly_residue_tilde}, it is possible to prove the following statement complementing~\eqref{eq:f_vect_poisson_poly_residue}.

\begin{proposition}\label{prop:poisson_polytope_ugly}
Let $\alpha\in\N$. Let also $d\in\N$ and $k\in \{1,\ldots,d\}$ be such that $\alpha k$ is odd. If $d$ is odd, then
\begin{equation}\label{eq:poisson_polytope_ugly}
\E f_{k-1}(\conv \Pi_{d,\alpha})
=
\frac{(-1)^{\frac{d-k}2}
\alpha^d d!}{\tilde c_{\frac {\alpha+1}{2}}^{d} k!}   \cdot  [u^{d-k}x^{-1}] \left( \frac{ \sin \left(u \tilde c_{\frac{\alpha+1}{2}}\int_{0}^x (\sin y)^{\alpha-1}\dd y\right)}{\tan\left(\frac u2\right)(\sin x)^{\alpha d +1}}\right).
\end{equation}
If $d$ is even, then
\begin{equation}\label{eq:poisson_polytope_ugly1}
\E f_{k-1}(\conv \Pi_{d,\alpha})
=
\frac{(-1)^{\frac{d-k-1}2}
\alpha^d d!}{\tilde c_{\frac {\alpha+1}{2}}^{d} k!} \cdot  [u^{d-k}x^{-1}] \left( \frac{ \cos \left(u \tilde c_{\frac{\alpha+1}{2}}\int_{0}^x (\sin y)^{\alpha-1}\dd y\right)}{\cot\left(\frac u2\right)(\sin x)^{\alpha d +1}}\right).
\end{equation}
\end{proposition}

\vspace*{2mm}
Let us now consider three special cases of Poisson polytopes.

\subsection{Poisson zero cell: \texorpdfstring{$\alpha=1$}{alpha=1}}\label{subsec:alpha_1}
It is known that on the space of all affine hyperplanes in $\R^d$ there is an infinite measure invariant under rotations and translations of $\R^d$. Moreover, this measure is unique up to multiplication by constants; see~\cite[Section~13.2]{SW08}.  Consider a Poisson process on the space of affine hyperplanes having this measure as intensity. The hyperplanes belonging to this point process dissect $\R^d$ into countably many polytopes, called the cells of the (homogeneous and isotropic) Poisson hyperplane tessellation. The a.s.\ unique cell containing the origin is called the \textit{Poisson zero polytope} and is denoted here by $\mathcal Z_d$; see~\cite[Sections~10.3, 10.4]{SW08}. The expected $f$-vector of $\mathcal Z_d$ has been identified only recently in~\cite{kabluchko_poisson_zero}. There  we have shown that for all $\ell\in \{1,\ldots, d\}$ such that $d-\ell$ is even,
\begin{equation}\label{eq:f_vect_poisson_zero1}
\E f_{\ell}(\mathcal Z_d) = \frac{\pi^{d-\ell}}{(d-\ell)!} \cdot  [x^{d-\ell}]
\prod_{\substack{j\in \{1,\ldots,d-1\} \\ j \not\equiv d \Mod{2}}} (1+ j^2 x^2).
\end{equation}
Theorem~\ref{theo:poisson_polyhedra} can be used to derive the following alternative formula.
\begin{theorem}\label{theo:poisson_zero_f_vect}
For all $d\in\N$ and $\ell\in \{0,\ldots, d\}$ such that $d-\ell$ is even, we have
\begin{equation}\label{eq:f_vect_poisson_zero2}
\E f_{\ell}(\mathcal Z_d) = \pi^{d-\ell} \binom dl [x^{d-\ell}] \left(\frac{x}{\sin x}\right)^{d+1}.
\end{equation}
\end{theorem}
\begin{proof}
The claim is trivial for $\ell=d$, so let $\ell \in \{0,\ldots,d-1\}$.
It is known~\cite{beta_polytopes} that $\conv \Pi_{d,1}$ can be identified (up to scaling) with the convex dual of the zero cell $\mathcal Z_d$ of the homogeneous and isotropic Poisson hyperplane tessellation.  In particular,
\begin{equation}\label{eq:E_f_ell_Z_d_Poisson_Process}
\E f_{\ell}(\mathcal Z_d)= \E f_{d-\ell-1}(\conv \Pi_{d,1})
\end{equation}
for all $\ell\in \{0,\ldots,d-1\}$. Taking $\alpha = 1$ in the second claim of Theorem~\ref{theo:poisson_polyhedra} we obtain after straightforward transformations that for all even $k\in \{1,\ldots, d\}$,
$$
\E f_{k-1}(\conv \Pi_{d,1}) = \pi^k \binom dk [x^k] \left(\frac{x}{\sin x}\right)^{d+1}.
$$
The claim follows by taking $k=d-\ell$ and applying~\eqref{eq:E_f_ell_Z_d_Poisson_Process}.
\end{proof}

\begin{remark}
Comparing~\eqref{eq:f_vect_poisson_zero1} and~\eqref{eq:f_vect_poisson_zero2}, we obtain the following curious combinatorial identity which is valid for all $d\in\N$ and all even $m\in \{0,\ldots,d\}$:
$$
\frac{d!}{(d-m)!} [x^m]  \left(\frac{x}{\sin x}\right)^{d+1} =
[x^m] \cdot \prod_{\substack{j\in \{1,\ldots,d-1\} \\ j \not\equiv d \Mod{2}}} (1+ j^2 x^2).
$$
\end{remark}

\begin{remark}
The values of $\E f_{\ell}(\mathcal Z_d)$ with even $d-\ell$ were called ``nice'' in~\cite{kabluchko_poisson_zero} because they are rational multiples of powers of $\pi^2$. The values with odd $d-\ell$ were called ``ugly'' and are polynomials in $\pi^2$ with rational coefficients.  For the ugly values, Proposition~\ref{prop:poisson_polytope_ugly} and~\eqref{eq:E_f_ell_Z_d_Poisson_Process} yield the formulae
\begin{align*}
\E f_{\ell}(\mathcal Z_d)
&=
\frac{(-1)^{\frac{\ell}2} \pi^{d} d!}{(d-\ell)!} \cdot  [u^{\ell} x^{-1}] \left( \frac{\cot\left(\frac u2\right) \sin \left( \frac{u x}{\pi} \right)}{(\sin x)^{d +1}}\right), &&\text{if $d$ is odd and $\ell$ is even},\\
\E f_{\ell}(\mathcal Z_d)
&=
\frac{(-1)^{\frac{\ell-1}2} \pi^{d} d!}{(d-\ell)!} \cdot  [u^{\ell}x^{-1}] \left(\frac{\tan\left(\frac u2\right)\cos\left(\frac{u x}{\pi} \right)}{(\sin x)^{d +1}}\right), &&\text{if $d$ is even and $\ell$ is odd}.
\end{align*}
\end{remark}

Note finally that all results of this section can be restated in terms of $\tilde \bJ_{n,k}(\frac n2)$ since, as we have shown in~\cite[Equation~(3.9)]{kabluchko_poisson_zero} and~\cite[Proposition~1.2]{kabluchko_poisson_zero},
$$
\tilde \bJ_{n,k}\left(\frac n2\right)
=
\frac{n (\E f_{n-k}(\mathcal Z_n) - \E f_{n-k-2}(\mathcal Z_{n-2}))}{2\pi^n \tilde c_{\frac{n+1}{2}}}
=
\frac{n(n-1)^2 \E f_{n-k}(\mathcal Z_{n-2})}{2\pi^{n-2} k(k-1) \tilde c_{\frac{n+1}{2}}},
$$
for all $n\in \{3,4,\ldots\}$, $k\in \{1,\ldots,n-2\}$, where the second equality additionally requires $k\neq 1$.

\subsection{The simplest case: \texorpdfstring{$\alpha=2$}{alpha=2}}\label{subsec:alpha_2}
As it turns out, this case can be treated by purely combinatorial tools.
\begin{theorem}\label{theo:bJ_alpha_2}
For all $n\in \N$ and $k\in \{1,\ldots,n\}$ we have\footnote{In the cases $n=1$ and $n=2$ we use the natural definitions $\binom{-1}{1} = -1$ and $\binom {0}{1} = \binom {0}{2}=0$.}
$$
\binom {2n}{n} \tilde \bJ_{n,k}\left(\frac{n+1}2\right) = \binom nk \binom{n+k}{k} - \binom {n-2} k \binom{n-2+k}{k}.
$$
\end{theorem}
The proof will be given in Section~\ref{sec:simplest_case_proof}. As a corollary, we can compute the expected $f$-vector of $\conv \Pi_{d,2}$.
\begin{theorem}
For all $d\in\N$ and $k\in \{1,\ldots,d\}$ we have
\begin{equation}\label{eq:exp_f_vect_Pi_d_2}
\E f_{k-1}(\conv \Pi_{d,2}) = \binom dk \binom {d+k}{k}.
\end{equation}
\end{theorem}

\begin{proof}
Recalling~\eqref{eq:poisson_voronoi_f_k_addition} and using Legendre's duplication formula, it is easy to check that for all $m\in\N$,
$$
\tilde \bI_{\infty,m}(2)
=
\frac{\Gamma({\frac{2m  + 1} 2})}{\Gamma(m)} \left(\frac {\Gamma(1)} {\Gamma(\frac 32)}\right)^{m}
\frac{(2\sqrt{\pi})^{m-1}}{m} = \frac 12  \binom {2m}{m}.
$$
By~\eqref{eq:E_f_k_beta_prime_to_poisson}, we obtain
\begin{equation*}
\E f_{k-1}(\conv \Pi_{d,2})
=
\sum_{\substack{m\in \{k,\ldots,d\}\\ m\equiv d \Mod{2}}}
\binom {2m}{m}  \tilde \bJ_{m,k}\left(\frac{m+1}{2}\right).
\end{equation*}
To complete the proof, apply Theorem~\ref{theo:bJ_alpha_2} and evaluate the telescope sum.
\end{proof}

On the other hand, Theorem~\ref{theo:poisson_polyhedra} with $\alpha = 2$ yields
$$
\E f_{k-1}(\conv \Pi_{d,2})
= 2^{d+k} \binom dk
\Res\limits_{x=0} \left[\frac{(1-\cos x)^{d-k}}{(\sin x)^{2d+1}}\right]
=
\binom dk
\Res\limits_{y=0} \left[(\sin y)^{-2k-1}(\cos y)^{-2d-1}\right],
$$
where the last step follows from the formulae $1-\cos x = 2 \sin^2 \frac x2$ and $\sin x = 2 \sin \frac x2 \cos \frac x2$ after the substitution $y:= x/2$.
Comparing both results, we obtain the curious identity
$$
\Res\limits_{y=0} \left[(\sin y)^{-2k-1}(\cos y)^{-2d-1}\right] = \binom {d+k}{k}
$$
which can be shown to be valid in the range $d\in\Z$, $k\in\N_0$.

The numbers on the right-hand side of~\eqref{eq:exp_f_vect_Pi_d_2} appear as Entry A063007 in~\cite{sloane}. Interestingly, the expected $f$-vector of $\conv \Pi_{d,2}$ coincides with the $f$-vector of the dual of the associahedron of type $B_n$ (cyclohedron)~\cite{fomin_reading}. We have no explanation for this coincidence, but one trivial reason can be easily ruled out: It is not true that the combinatorial type of $\conv \Pi_{d,2}$ is that of the cyclohedron (or any other deterministic polytope), with probability $1$. Instead we claim that for every $\alpha>0$ and every given simplicial $d$-dimensional polytope $P$, the probability that $\conv \Pi_{d,\alpha}$ has  the same combinatorial type as $P$ is strictly positive. To see this, embed $P$ into $\R^d$ such that the origin is in the interior of $P$. Then, for every $\eps>0$ the probability that each $\eps$-ball around each vertex of $P$ contains exactly one point of the Poisson point process $\Pi_{d,\alpha}$, while all other points of $\Pi_{d,\alpha}$ are located inside $\frac 12 P$, is positive. If $\eps>0$ is sufficiently small and this event occurs, then $\conv \Pi_{d,\alpha}$ has the same combinatorial type as $P$ (because $P$ is simplicial and hence its combinatorial type does not change under small perturbations).

\subsection{Typical Poisson-Voronoi cell: \texorpdfstring{$\alpha=d$}{alpha=d}}
Let $P_1,P_2,\ldots$ be the points of a Poisson point process on $\R^d$ with constant intensity $1$. The \textit{Voronoi cell} $\mathcal C_i$ of the point $P_i$ consists of all points $x\in\R^d$ whose distance to $P_i$ is smaller or equal than the distances to all other points $P_j$ with $j\neq i$. With probability $1$, the cells $\mathcal C_1,\mathcal C_2,\ldots$ are polytopes with disjoint interiors that form a covering of $\R^d$ which is called the \textit{Poisson-Voronoi tessellation}. We shall be interested in the \textit{typical cell} of this tessellation which can be constructed as follows~\cite[p.~450, p.~106]{SW08}.
Take some Borel set $Q\subset \R^d$ with finite but non-zero Lebesgue measure $\Vol_d(Q)$. Then, the typical Poisson-Voronoi cell is the random polytope $\mathcal V_d$ whose probability law is given by
$$
\P[\mathcal V_d\in \cdot\,] = \frac{1}{\Vol_d(Q)}\, \E \left[\sum_{i\in \N: P_i\in Q} \ind_{\{\mathcal C_i -P_i \in \cdot\,\}}\right],
$$
where $\mathcal C_i - P_i = \{x-P_i\colon x\in \mathcal C_i\}$ denotes the cell $\mathcal C_i$ shifted in such a way that  its ``center'' $P_i$ moves to the origin.
For our purposes, the following explicit construction is more convenient: the typical Poisson-Voronoi cell is the random polytope
\begin{equation}\label{eq:def_typical_voronoi}
\mathcal V_d := \{x\in \R^d\colon \|x\| \leq \|x-P_j\| \text{ for all $j\in\N$}\}.
\end{equation}
In words, $\mathcal V_d$ the Voronoi cell of the origin in the Poisson process with an additional atom at the origin. The equivalence of both constructions of $\mathcal V_d$ follows from the characterisation of the Palm distribution of the Poisson process~\cite[Theorem~3.3.5]{SW08}.

The Poisson-Voronoi tessellation is one of the basic objects in stochastic geometry~\cite{SW08,okabe_etal_book,stoyan_etal_book,moller,moller_book,calka_classical_problems} and has been intensively studied at least since the works of Meijering~\cite{meijering}, Gilbert~\cite{gilbert} and Miles~\cite{miles_synopsis}. However, such a natural functional as the expected number of $k$-faces of the typical cell has been known explicitly only in some special cases. It is a classical fact~\cite[Theorem 10.2.5]{SW08} that in dimension $d=2$, the typical cell is ``on average'' a hexagon, that is
$$
\E f_0 (\mathcal V_2) = \E f_1 (\mathcal V_2) = 6.
$$
In the $1953$ work, Meijering~\cite{meijering} determined explicitly some basic characteristics of $\mathcal V_d$ with $d=3$ including the expected $f$-vector
$$
\E  f_0 (\mathcal V_{3}) = \frac{96 \pi ^2}{35},
\quad
\E  f_1 (\mathcal V_{3}) = \frac{144 \pi ^2}{35},
\quad
\E  f_2 (\mathcal V_{3}) = 2+\frac{48 \pi ^2}{35}.
$$
In his $1970$ work, Miles~\cite{miles_synopsis} stated a formula for $\E f_0(\mathcal V_d)$ which is valid for all $d\in\N$. A full proof was given by Miles in~\cite{miles_sectional} and  can also be found in the work of M{\o}ller~\cite{moller}. In view of the fact that $\mathcal V_d$ is a simple polytope (that is, its dual is simplicial), this gives also a formula for $\E f_1(\mathcal V_d)= \frac d2\, \E f_0(\mathcal V_d)$. All the results listed above and many more references can be found in Section 10.2 of the book of Schneider and Weil~\cite{SW08} and in the notes thereafter. Asymptotic results on the expected $f$-vectors of $\mathcal V_d$ (and more general Poisson polytopes) as $d\to\infty$ have been obtained in~\cite{HoermannHugReitznerThaele}.
Very recently, we described~\cite{kabluchko_algorithm} an algorithm which allowed us to compute the expected $f$-vector of $\mathcal V_d$ in dimensions $d\leq 10$.

Now  we are able to derive an explicit formula for the expected $f$-vector of $\mathcal V_d$. It is known~\cite{HoermannHugReitznerThaele,beta_polytopes} that $\mathcal V_d$ has the same distribution as the convex dual of $\Pi_{d,\alpha}$ with $\alpha=d$, up to scaling.\footnote{Let us sketch the proof. Denote by $H_j$ the half-space containing the origin and whose bounding hyperplane passes through the point $P_j/2$ and is orthogonal to the segment $[0,P_j]$. By~\eqref{eq:def_typical_voronoi},  $\mathcal V_d$ is the intersection of the $H_j$'s, $j\in\N$. The dual polytope of $\mathcal V_d$ is the convex hull of the polars of the hyperplanes bounding the $H_j$'s, i.e.\ of the points $Q_j$ obtained by inverting $P_j/2$ w.r.t.\ the unit sphere. By the transformation property of the Poisson point processes, the points $Q_j$ form a Poisson point process whose intensity is easily seen to be of the form $\text{const} \cdot \|x\|^{-2d}$.} In particular, for the expected $f$-vectors of these random polytopes we have the duality relation
$$
\E f_{d-k}(\mathcal V_d) = \E f_{k-1}(\conv\Pi_{d,d}),
$$
for all $k\in \{1,\ldots,d\}$.
Taking $\alpha=d$ in Theorem~\ref{theo:poisson_polyhedra}, we obtain the following explicit formula for the expected $f$-vector of the typical Poisson-Voronoi cell.
\begin{theorem}\label{theo:typical_poisson_voronoi}
For all  $d\in\N$ and $k\in \{1,\ldots,d\}$ such that $dk$ is even, we have
\begin{equation}
\E f_{d-k}(\mathcal V_d)
= d^{d} \binom{d}{k} \left(\frac{ \sqrt \pi\, \Gamma(\frac d2)}{\Gamma(\frac{d+1}{2})}\right)^k
\Res\limits_{x=0} \left[\frac{\left(\int_{0}^x (\sin y)^{d-1} \dd y\right)^{d-k}}{(\sin x)^{d^2+1}}\right].
\end{equation}
\end{theorem}
This gives the complete expected $f$-vector of  $\mathcal V_d$ if the dimension $d$ is even. In the case when $d$ is odd, the theorem identifies only the entries with even codimension, while the remaining entries are uniquely determined by the (dual) Dehn-Sommerville relations; see Section~\ref{sec:dehn_sommerville}. It is also straightforward to write down an explicit though ``ugly'' formula by taking $\alpha = d$ in Proposition~\ref{prop:poisson_polytope_ugly}.
The next theorem describes the arithmetic structure of the expected $f$-vector of $\mathcal V_d$.
\begin{theorem}\label{theo:poisson_voronoi_arithm}
Let $d\in\N$ and $\ell\in \{0,\ldots,d-1\}$.
\begin{itemize}
\item[(a)] If $d$ is even, then $\E f_\ell(\mathcal V_{d})$ is a rational number.
\item[(b)]  If both $d$ and $\ell$ are odd, then $\E f_\ell(\mathcal V_{d})$ is a number of the form $q\pi^{d-\ell}$ with some rational $q$.
\item[(c)] If $d$ is odd and $\ell$ is even, then $\E f_\ell(\mathcal V_{d})$ can be expressed as  $q_{d-1} \pi^{d-1} + q_{d-3} \pi^{d-3} + \ldots +q_{d-\ell-1}\pi^{d-\ell-1}$, where the coefficients $q_{i}$ are rational.
\end{itemize}
\end{theorem}
\begin{proof}
Parts (a) and (c) were established in~\cite{kabluchko_algorithm}, whereas part (b) was conjectured there and is now an immediate consequence of Theorem~\ref{theo:typical_poisson_voronoi} because the residue appearing there is a rational number.
\end{proof}

\begin{remark}
For $j\in \{0,\ldots,d\}$, the \textit{intensity of $j$-faces} in the Poisson-Voronoi tessellation, denoted by $\gamma^{(j)}$, is the large $n$ limit of the number of $j$-dimensional faces of the Poisson-Voronoi tessellation contained in the cubical window $Q_n =[-n,n]^d$, divided by the volume of $Q_n$.  The intensity of the $j$-faces satisfies
$$
\gamma^{(j)} = \frac{\E f_j(\mathcal V_d)}{d-j+1}, \qquad j\in \{0,\ldots,d\};
$$
see Theorems~10.1.2 and~10.2.3 of~\cite{SW08}. Theorem~\ref{theo:typical_poisson_voronoi} and Proposition~\ref{prop:poisson_polytope_ugly} (with $\alpha = d$) yield explicit formulae for all $\gamma^{(j)}$'s. A summary of what has previously been known on $\gamma^{(j)}$ can be found in~\cite[pp.~476--477]{SW08}.
\end{remark}

\subsection{Reitzner constants}
Let $K\subset \R^d$ be a $d$-dimensional convex body. Assume that the  boundary $\partial K$ is  $\mathcal C^2$-differentiable and the Gaussian curvature $\kappa(x)>0$ is strictly positive at every point $x\in \partial K$.  Let $U_1,U_2,\ldots$ be independent random points distributed uniformly in $K$. Denote the convex hull of $n$ such points by $K_{n,d}= [U_1,\ldots,U_n]$. As $n\to\infty$, the random polytopes $K_{n,d}$ approximate $K$ and their asymptotic properties have been intensively studied.
In the case $d=2$, R\'enyi and Sulanke~\cite[Satz~3]{renyi_sulanke1} proved that
$$
\lim_{n\to\infty}
\frac{\E f_0(K_{n,2})}{n^{1/3}}
=
\lim_{n\to\infty}
\frac{\E f_1(K_{n,2})}{n^{1/3}}
=
 \Gamma\left(\frac 53\right) \sqrt[3]{\frac 23} \frac 1{(\Vol_2(K))^{1/3}} \int_{\partial K}(\kappa(x))^{1/3} \dint x.
$$
In the case of arbitrary dimension $d\in\N$,  Reitzner~\cite[p.~181]{ReitznerCombinatorialStructure} proved that  for every $k\in\{0,1,\ldots,d-1\}$,
\begin{equation}\label{eq:ReitznerExpectation}
\lim_{n\to\infty}
\frac{\E f_k(K_{n,d})}{n^{\frac{d-1}{d+1}}}
=
C_{d,k}\frac{\Omega(K)/ (\Vol_d(K))^{\frac{d-1}{d+1}}}{\Omega(\bB^d)/(\Vol_d(\bB^d))^{\frac{d-1}{d+1}}},
\end{equation}
where $\Omega(K):=\int_{\partial K}(\kappa(x))^{\frac{1} {d+1}} \dint x$ is the so-called \textit{affine surface area} of $K$ and $C_{d,0}, \ldots, C_{d,d-1}$ are strictly positive constants\footnote{In fact, Reitzner used the slightly different notation $c_{d,k} = C_{d,k}\cdot (\Vol_d(\bB^d))^{\frac{d-1}{d+1}}/\Omega(\bB^d)$. Note that he assumes that $\Vol_d(K) = 1$, which is why additional terms involving the volume appear in~\eqref{eq:ReitznerExpectation}.} not depending on $K$. It seems that so far only the values $C_{d,d-1}$ and $C_{d,0}$ have been known~\cite[Corollary~1 on p.~366]{affentranger}, \cite[Theorem~3 on p.~760]{buchta_mueller}, \cite[Corollary~7.1]{hug_rev}; see also~\cite[\S~3.4]{kabluchko_algorithm} for a discussion. This yields also the value $C_{d,d-2}= (d/2) C_{d,d-1}$ since the polytope $K_{n,d}$ is simplicial.   In~\cite[Remark~1.9]{beta_polytopes}, the constant $C_{d,k}$ was expressed through $\bJ_{d,k+1}(1/2)$. This was used in~\cite{kabluchko_algorithm} to compute $C_{d,k}$ for all $d\leq 10$. Using Theorem~\ref{theo:bJ_formula_residue} we can now state an explicit formula for $C_{d,k}$, thus answering a question posed by Reitzner~\cite[p.~181]{ReitznerCombinatorialStructure}.

\begin{theorem}
For all $d\in\N$ and $k\in\{0,\ldots,d-1\}$ such that either (i) both $d$ and $d-k$ are even or (ii) $d$ is odd, we have
$$
C_{d,k} = \frac{d^2+1}{d!} (d+1)^{\frac{d^2-d}{d+1}}  \binom d{k+1} \Gamma\left(\frac{d^2+1}{d+1}\right) \left(\frac{\sqrt \pi \, \Gamma\left(\frac{d+1}{2}\right)}{\Gamma\left(\frac{d+2}{2}\right)} \right)^{k + \frac{2}{d+1}}  \Res\limits_{x=0} \left[\frac{\left(\int_{0}^x (\sin y)^{d} \dd y\right)^{d-k-1}}{(\sin x)^{d^2 + 2}}\right].
$$
\end{theorem}
\begin{proof}
In~\cite[Remark~1.9]{beta_polytopes}, it has been shown that for all $d\in\N$ and $k\in\{0,\ldots,d-1\}$, we have
\begin{equation}\label{eq:Limitf_kBall}
C_{d,k}
= \frac{2\pi^{\frac{d(d-1)} {2(d+1)}}}{(d+1)!}
\frac{\Gamma(1+\frac{d^2}{2})\Gamma(\frac{d^2+1}{d+1})}{\Gamma(\frac {d^2+1}{2})}
(d+1)^{\frac{d^2+1}{d+1}}
\left(\frac{\Gamma(\frac{d+1}{2})}{\Gamma\left(1+\frac {d}{2}\right)}\right)^{\frac{d^2+1}{d+1}} \bJ_{d,k+1}\left(\frac 12\right).
\end{equation}
To compute $\bJ_{d,k+1}(\frac 12)$ we apply Theorem~\ref{theo:bJ_formula_residue} with $\alpha = n= d$ and $k+1$ instead of $k$. After some cancellations,  we obtain the required formula.
\end{proof}

Similar analysis can be done for random polytopes whose vertices are uniformly distributed on the sphere. Let  $U_1^*,U_2^*,\ldots$ be independent random points  chosen uniformly at random on the unit sphere $\bS^{d-1}$, $d\geq 2$. Denote the convex hull of $n$ such points by $K_{n,d}^*:= [U_1^*,\ldots, U_n^*]$.  In~\cite[Remark~1.9]{beta_polytopes}, it has been shown that for all $k\in \{0,\ldots,d-1\}$,
\begin{equation}\label{eq:reitzner_const_sphere}
C_{d,k}^*:=\lim_{n\to\infty} \frac 1n  \E f_k(K^*_{n,d})
=
\frac{2^d\pi^{{\frac d2}-1}}{d(d-1)^2} \frac{\Gamma\left(1+\frac{d(d-2)}{2}\right)}{\Gamma\left(\frac{(d-1)^2}{2}\right)}\left(\frac{\Gamma\left({\frac{d+1}{2}}\right)}{\Gamma\left(\frac d 2\right)}\right)^{d-1}
\bJ_{d,k+1}\left(-\frac 12\right).
\end{equation}
It seems that previously only the values $C_{d,d-1}^*$, $C_{d,d-2}^* = (d/2)C_{d,d-1}^*$ and $C_{d,0}^*=1$ have been known~\cite[Corollary~1 on p.~366]{affentranger}, \cite[p.~231]{buchta_mueller_tichy}; see also~\cite[\S~3.5]{kabluchko_algorithm}  for a discussion.

\begin{theorem}
For all $d\in\N$ and $k\in\{0,\ldots,d-1\}$ such that either (i) both $d$ and $d-k$ are even or (ii) $d$ is odd, we have
$$
C_{d,k}^*
=
\frac{\sqrt \pi}{d} (d-1)^{d-1}  \binom d{k+1}
\left(\frac{\sqrt \pi\, \Gamma\left(\frac{d-1}{2}\right)}{\Gamma\left(\frac d2\right)}\right)^{k}
\Res\limits_{x=0} \left[\frac{\left(\int_{0}^x (\sin y)^{d-2} \dd y\right)^{d-k-1}}{(\sin x)^{d^2 - 2d + 2}}\right].
$$
\end{theorem}
\begin{proof}
Compute $\bJ_{d,k+1}(-\frac 12)$ by Theorem~\ref{theo:bJ_formula_residue} with $n=d$, $\alpha = d-2$ and $k+1$ instead of $k$. Inserting this value into~\eqref{eq:reitzner_const_sphere}, we obtain the claim after straightforward cancellations.
\end{proof}

\subsection{Further applications}
Beta and beta' polytopes are defined as convex hulls of points  sampled independently according to a $d$-dimensional beta or beta' distribution. The results of Sections~\ref{subsec:beta_simplices} and~\ref{subsec:beta_simplices_prime} of the present paper, combined with those of~\cite{beta_polytopes}, yield explicit formulae for the expected $f$-vectors of these polytopes, solving in particular Problem~A stated in Section~\ref{subsec:intro}. Since these formulae are most conveniently stated after introducing some notation, we postpone them to Theorems~\ref{theo:beta_poly_f_vector} and~\ref{theo:beta_poly_f_vector_tilde}. Using~\cite{beta_polytopes} it is also possible to compute expected sums of (Grassmann) angles of beta and beta' polytopes.
Let us finally mention that there are further quantities in stochastic geometry that can be studied by the methods of the present paper, for example expected $f$-vectors of typical Voronoi cells on the sphere or in the hyperbolic space, probability contents of beta and beta' polytopes, and expected face intensities in certain Laguerre tessellations.

\section{Preliminaries}\label{sec:prelim}
The remaining part of the paper is devoted to the proofs of the above results. The only exception is Section~\ref{sec:expected_face_beta_poly} where we explain the solution of Problem~A. We start by recalling some facts about the expected external angles of beta and beta' simplices.

\subsection{Expected external angles}\label{subsec:external}
Given a $d$-dimensional simplex $T := [x_1,\ldots,x_n]\subset \R^{n-1}$, the \textit{normal cone} $N(F,T)$ at its face $F := [x_{i_1},\ldots,x_{i_k}]$ is defined as
$$
N(F,T) := \{y\in \R^{n-1}: \langle y, x-z\rangle \leq 0 \text{ for all } x\in T\},
$$
where $z$ is any point in $F$ not belonging to a face of smaller dimension. The \textit{external angle} of $T$ at its face $F$ is then defined as
$$
\gamma(F,T) := \frac{\Vol_{n-1}(\bB_r^{n-1}(0)\cap N(F,T)) }{\Vol_{n-1}(\bB_r^{n-1}(0))},
$$
where $r>0$ is arbitrary and $\bB_r^{n-1}(0)$ is the ball of radius $r$ centered at $0$.

The expected internal angles $J_{n,k}(\beta)$ and $\tilde J_{n,k}(\beta)$ appeared in~\cite{beta_polytopes} together with the quantities $I_{n,k}(\alpha)$ and $\tilde I_{n,k}(\alpha)$ that are related to the expected external angles and are defined as follows:
\begin{align}
I_{n,k}(\alpha)
&=
\int_{-\pi/2}^{+\pi/2} c_{\frac{\alpha k - 1}{2}} (\cos \varphi)^{\alpha k} \left(\int_{-\pi/2}^\varphi c_{\frac{\alpha-1}{2}}(\cos \theta)^{\alpha} \,\dd \theta \right)^{n-k} \, \dd \varphi,
\quad \alpha>-1/k,
\label{eq:I_n_k}\\
\tilde I_{n,k}(\alpha)
&=
\int_{-\pi/2}^{+\pi/2} \tilde c_{\frac{\alpha k + 1}{2}} (\cos \varphi)^{\alpha k-1} \left(\int_{-\pi/2}^\varphi \tilde c_{\frac{\alpha+1}{2}}(\cos \theta)^{\alpha-1} \,\dd \theta \right)^{n-k} \, \dd \varphi,
\quad \alpha>0.
\label{eq:I_n_k_tilde}
\end{align}
They next theorem states that the expected external angles of beta and beta' simplices can be expressed through $I_{n,k}(\alpha)$ and $\tilde I_{n,k}(\alpha)$.  It is a special case of Theorems 1.6 and 1.16 in~\cite{beta_polytopes}.
\begin{theorem}\label{theo:external}
Let $X_1,\ldots,X_n$ be i.i.d.\ random points in $\R^{n-1}$ with beta distribution $f_{n-1,\beta}$, where $\beta\geq -1$. Similarly, let $\tilde X_1,\ldots,\tilde X_n$ be i.i.d.\ random points in $\R^{n-1}$ with beta' distribution $\tilde f_{n-1,\beta}$, where $\beta > \frac {n-1}{2}$.  Then, for all $k\in \{1,\ldots,n\}$, 
\begin{align*}
&\E \gamma ([X_1,\ldots,X_k], [X_1,\ldots,X_n]) = I_{n,k}(2\beta + n-1),\\
&\E \gamma ([\tilde X_1,\ldots,\tilde X_k], [\tilde X_1,\ldots,\tilde X_n]) = \tilde I_{n,k}(2\beta - n + 1).
\end{align*}
\end{theorem}
Since it will be more convenient to work with angle sums rather than with  individual angles, let us introduce the quantities
\begin{equation}\label{eq:I_tilde_I_bold}
\bI_{n,k}(\alpha) := \binom nk I_{n,k}(\alpha),
\qquad
\tilde \bI_{n,k}(\alpha) := \binom nk \tilde I_{n,k}(\alpha).
\end{equation}
Note that $\bI_{n,n}(\alpha) = \tilde \bI_{n,n}(\alpha) = 1$ (which follows either from~\eqref{eq:I_n_k_tilde} or from the observation that if all external angles at vertices are shifted to the origin, they fill the whole space and intersect only at their boundaries).

\subsection{Relations between internal and external angles}\label{sec:relations}
The expected external and internal angles satisfy non-linear relations which will play a crucial role in the following and were obtained in~\cite{kabluchko_algorithm} (although the main work had been done in~\cite{beta_polytopes}). As explained in~\cite{kabluchko_algorithm}, these relations can be derived by combining McMullen's angle-sum relations~\cite{mcmullen,mcmullen_polyhedra} with the canonical decomposition of beta and beta' distributions~\cite{ruben_miles}.
\begin{proposition}\label{prop:relations}
For all $n\in \{2,3,\ldots\}$ and $k\in \{1,\ldots,n-1\}$,  the following relations are satisfied:
\begin{align}
&\sum_{m=k}^n (-1)^m \bI_{n,m}(\alpha) \bJ_{m,k} \left(\frac{\alpha - m + 1}{2}\right) = 0
\quad \text{for  $\alpha\geq n-3$},
\label{eq:relation_I_J_1}\\
&\sum_{m=k}^n (-1)^m \tilde \bI_{n,m}(\alpha) \tilde \bJ_{m,k} \left(\frac{\alpha + m - 1}{2}\right) = 0 \quad \text{for  $\alpha > 0$}.
\label{eq:relation_I_J_1_tilde}
\end{align}
\end{proposition}
An important observation, also made in~\cite{kabluchko_algorithm}, is that these relations, together with the condition $\bJ_{n,n}(\beta) = \tilde \bJ_{n,n}(\beta)=1$, determine the $\bJ$-quantities \textit{uniquely}. Precise statements are given in the next two propositions.
\begin{proposition}\label{prop:bJ_unique}
Fix some integer $N\geq 2$ and $\alpha \geq N-3$.
Consider the following finite system of linear equations in the unknowns $\xi_{n,k}$, where $n\in\{1,\ldots,N\}$, $k\in \{1,\ldots,n\}$:
\begin{equation}\label{eq:system_alpha_arbitrary_notilde}
\begin{cases}
\sum_{m=k}^n (-1)^{m-k} \bI_{n,m}(\alpha) \xi_{m,k} =  0, &\text{ for all } n\in \{1,\ldots,N\}, \; k\in \{1,\ldots,n-1\},\\
\xi_{n,n} = 1, \text{ for all } n\in \{1,\ldots,N\}.
\end{cases}
\end{equation}
Then, the unique solution to this system is $\xi_{n,k} = \bJ_{n,k}(\frac{\alpha - n + 1}{2})$.
\end{proposition}

\begin{proposition}\label{prop:bJ_tilde_unique}
Fix some $\alpha>0$. Consider the following infinite system of linear equations in the unknowns $\tilde \xi_{n,k}$, where $n\in\N$, $k\in \{1,\ldots,n\}$:
\begin{equation}\label{eq:system_alpha_arbitrary}
\begin{cases}
\sum_{m=k}^n (-1)^{m-k}\tilde \bI_{n,m}(\alpha) \tilde \xi_{m,k} =  0, &\text{ for all } n\in\N, \; k\in \{1,\ldots,n-1\},\\
\tilde \xi_{n,n} = 1, \text{ for all } n\in \N.
\end{cases}
\end{equation}
Then, the unique solution to this system is $\tilde \xi_{n,k} = \tilde \bJ_{n,k}(\frac{\alpha + n-1}{2})$.
\end{proposition}

Both propositions can be conveniently stated in matrix form. For example, Proposition~\ref{prop:bJ_unique} states that the following lower-triangular matrices with $\pm 1$'s on the main diagonal are inverse to each other:
$$
\left((-1)^{m} \bI_{n,m}(\alpha)\right)_{1\leq n,m\leq N} \text{ and } \left((-1)^{m} \bJ_{n,m}\left(\frac{\alpha - n + 1}{2}\right)\right)_{1\leq n,m\leq N}.
$$
For both matrices, we agree to define the elements above the main diagonal to be $0$.  Thus, our task is ``just'' to invert a lower-triangular matrix. Of course, this can be done sequentially, row by row. This approach has been exploited in~\cite{kabluchko_algorithm}. Although it allows to compute any entry of the inverse matrix in finitely many steps, it does not lead to a satisfactory general formula and, moreover, fails to explain the  arithmetic properties of the entries of the inverse matrix; see~\cite{kabluchko_algorithm}.
The most difficult problem is to \textit{guess} the solutions to~\eqref{eq:system_alpha_arbitrary_notilde} and~\eqref{eq:system_alpha_arbitrary} (which is much harder than proving that the guess is indeed correct).

\subsection{How to guess the solution}\label{subsec:method}
To get some impression on how the solution may look like, we look at the case $\beta=+\infty$ which, as we shall explain, is closely related to the angles of the regular simplex.

\subsubsection{The regular simplex and the Gaussian simplex}
The \textit{regular simplex} with $n$ vertices is defined as $[e_1,\ldots,e_n]$, where $e_1,\ldots, e_n$ is the standard orthonormal basis in $\R^n$. Both external and internal angles of the regular simplex are known explicitly; see~\cite{hadwiger} and~\cite{ruben} for external angles and~\cite[Section~4]{rogers} (where the method used was attributed to H.\ E.\ Daniels) as well as~\cite[Lemma~4]{vershik_sporyshev} for internal angles.  We state these formulae in the form given in~\cite[Proposition~1.7]{kabluchko_zaporozhets_absorption}:
\begin{align}
\gamma(n,k) &:= \gamma ([e_1,\ldots,e_k], [e_1,\ldots,e_n])  = g_{n-k}\left(\frac 1k\right),\label{eq:I_J_infty1}\\
\beta(n,k) &:= \beta ([e_1,\ldots,e_k], [e_1,\ldots,e_n])  = g_{n-k}\left(-\frac 1n\right),\label{eq:I_J_infty2}
\end{align}
where $g_m:[-1/m,\infty] \to [0,1]$ is a real-analytic function given by
$$
g_m(r) =  \frac 1 {\sqrt {2\pi}} \int_{-\infty}^{\infty} \Phi^m (\sqrt r x) \eee^{-x^2/2} \dd x.
$$
Here, $\Phi(y)$ is the standard Gaussian distribution function (which can be analytically continued to the entire complex plane), and we use the convention $\sqrt{r} = i \sqrt{-r}$ for $r<0$. The function $g_m$ is closely related to the Schl\"afli function which expresses volumes of regular simplices in the spherical space; see, e.g., \cite{kabluchko_zaporozhets_absorption}.  An interesting consequence of~\eqref{eq:I_J_infty1} and~\eqref{eq:I_J_infty2} is that, on the formal level, we have the ``reciprocity law''
\begin{equation}\label{eq:reciprocity_infty}
\beta(n,k) = \gamma(-k,-n).
\end{equation}


The \textit{Gaussian simplex} is the random simplex  $[X_1,\ldots,X_n]$, where $X_1,\ldots,X_n$ are independent random points having the standard Gaussian random distribution on $\R^{n-1}$. It is known that both internal and external angles of the Gaussian simplex coincide, on average, with the corresponding expected angles of the regular simplex~\cite{kabluchko_zaporozhets_gauss_simplex,goetze_kabluchko_zaporozhets,kabluchko_zaporozhets_ongoing}.
More precisely, for all $k\in \{1,\ldots,n\}$ we have
$$
\E \gamma ([X_1,\ldots,X_k], [X_1,\ldots,X_n]) = \gamma(n,k),
\qquad
\E \beta ([X_1,\ldots,X_k], [X_1,\ldots,X_n]) = \beta(n,k).
$$
On the other hand, both the beta and the beta' distribution, when appropriately rescaled, weakly converge to the standard Gaussian distribution as $\beta\to+\infty$; see~\cite[Lemma~1.1]{beta_polytopes}. Since the rescaling does not change angles, the continuous mapping theorem implies that
\begin{align*}
I_{n,k}(+\infty) := \lim_{\beta\uparrow +\infty} I_{n,k}(\beta)= \gamma(n,k),
\qquad
J_{n,k}(+\infty) := \lim_{\beta\uparrow +\infty} J_{n,k}(\beta)= \beta(n,k),
\end{align*}
and similarly for $\tilde I_{n,k}(+\infty)$ and $\tilde J_{n,k}(+\infty)$. Summarizing, we have the ``reciprocity law''
$$
J_{n,k}(+\infty) = I_{-k,-n}(+\infty),
$$
where the right-hand side has to be understood in the sense of analytic continuation.

\subsubsection{Stirling numbers}
Relations very similar to~\eqref{eq:system_alpha_arbitrary_notilde} and~\eqref{eq:system_alpha_arbitrary} hold for Stirling numbers of the first and second kind denoted by $\stirling{n}{k}$ and $\stirlingb{n}{k}$, respectively. Namely, it is well known~\cite[p.250]{graham_knuth_patashnik_book} that
\begin{equation}\label{eq:stirling_inverse}
\sum_{m=k}^n (-1)^{n-m} \stirlingb{n}{m}\stirling{m}{k}
=
\sum_{m=k}^n (-1)^{n-m} \stirling{n}{m}\stirlingb{m}{k}
=
\delta_{nk}.
\end{equation}
On the other hand, there is a natural definition of Stirling numbers for negative $n$ and $k$, see~\cite{knuth}, such that the following reciprocity relation holds:
\begin{equation}\label{eq:stirling_reciprocity}
\stirlingb{n}{k} = \stirling{-k}{-n}.
\end{equation}
Assuming that $\bI_{n,k}(\alpha)$ is the analogue of $\stirlingb{n}{k}$, while $\bJ_{n,k}(\frac{\alpha - n + 1}{2})$ is the analogue of $\stirling{n}{k}$, Equation~\eqref{eq:stirling_reciprocity} suggests that we should have something like
\begin{equation}\label{eq:reciprocity_conj}
\bJ_{n,k}\left(\frac{\alpha - n + 1}{2}\right) \stackrel{?}{=} \bI_{-k,-n}(\alpha).
\end{equation}
In fact, things are more complicated. Attempting to guess the formula for $\bJ_{n,k}(\beta)$ we constructed an analytic continuation of $\bI_{n,k}(\alpha)$ to complex values of $n$ and $k$.  It turned out that some version of~\eqref{eq:reciprocity_conj} is indeed valid, but it requires additional assumptions on $\alpha$. To avoid unnecessary assumptions on $\alpha$, we shall refrain from constructing the analytic continuation here and use certain closely related quantity instead. The question of how the true analytic continuation is related to the substitute used here will be studied elsewhere.

\subsection{Outline of the proofs}
We start by presenting the proof of Theorem~\ref{theo:bJ_alpha_2} in Section~\ref{sec:simplest_case_proof}. On the one hand, this proof, being purely combinatorial,  is easy compared to the proof of the general Theorems~\ref{theo:bJ_formula_integral}, \ref{theo:bJ_formula_residue}, \ref{theo:bJ_tilde_formula_integral}, \ref{theo:bJ_tilde_formula_residue}. On the other hand, some of the ideas of this argument will be important for the proof in the general case.

The proof of Theorem~\ref{theo:bJ_formula_integral}  will be given in Section~\ref{sec:proof_beta}. It is based on Proposition~\ref{prop:bJ_unique} and apart from this does not use any stochastic geometry. The proof is structured as follows.
\begin{itemize}
\item [(1)] We define certain quantities called $\lB\{\nu,\kappa\}$ that are related to the expected external angle sums $\bI_{n,k}(\alpha)$.
\item[(2)]  We prove that these quantities satisfy certain recurrence relations which are somewhat reminiscent of the relations satisfied by the Stirling numbers.
\item[(3)] We define another set of quantities, called $\lA[\nu,\kappa]$, which satisfy the same recurrence relations as the quantities $\lB\{-\kappa, -\nu\}$.
\item[(4)] Using both sets of recurrence relations, we prove that certain matrices constructed out of $\lB\{\nu,\kappa\}$ and $\lA[\nu,\kappa]$ are inverse to each other.
\item[(5)] We construct a solution to the system of equations stated in Proposition~\ref{prop:bJ_unique} in terms of $\lA[\nu,\kappa]$, thus identifying $\bJ_{n,k}(\frac{\alpha - n + 1}{2})$.
\end{itemize}
The remaining results in the beta case can be deduced from Theorem~\ref{theo:bJ_formula_integral} by residue calculus and other standard methods. The proofs in the beta' case are very similar and will be sketched in Section~\ref{sec:proof_beta_tilde}.

\subsection{Notational conventions}
We use the letters $\nu, \mu,\kappa$ to denote variables that are similar to $n,m,k$ but need not be integer.
If $\lA[\nu,\kappa],\lB\{\nu,\kappa\},F(x),\ldots$ denote some quantities related to beta simplices, then  $\mA[\nu,\kappa],\mB\{\nu,\kappa\},\tilde F(x),\ldots$ denote similar quantities related to beta' simplices. Record for future use the standard integrals
\begin{align}
\int_{-\pi/2}^{+\pi/2} (\cos x)^{h-1} \dd x
=
\int_{-\infty}^{+\infty} (\cosh y)^{-h} \dd y
=
\frac{\sqrt \pi\, \Gamma \left(\frac{h}{2}\right)}{\Gamma\left(\frac{h+1}{2}\right)}
=
\frac{1}{c_{\frac{h-2}{2}}}
=
\frac{1}{\tilde c_{\frac{h+1}{2}}} ,
\quad
\Re h >0.\label{eq:int_cos_cosh}
\end{align}
The first integral can be reduced to the beta function, whereas the second one can be reduced to the first one by the substitution $y= \arcsin \tanh x$.

\section{The simplest case: Proof of Theorem~\ref{theo:bJ_alpha_2}}\label{sec:simplest_case_proof}

\subsection{Proof  of Theorem~\ref{theo:bJ_alpha_2}}\label{subsec:proof_alpha_2}
Our aim is to show that for all $n\in \N$ and $k\in \{1,\ldots,n\}$ we have
\begin{equation}\label{eq:bJ_alpha_2_proof}
\binom {2n}{n} \tilde \bJ_{n,k}\left(\frac{n+1}2\right) = \binom nk \binom{n+k}{k} - \binom {n-2} k \binom{n-2+k}{k}.
\end{equation}
Recalling Proposition~\ref{prop:bJ_tilde_unique} and taking $\alpha=2$ there, consider the following system of linear equations in the unknown quantities $\tilde \xi_{n,k}$, where $n\in \N$ and $k\in \{1,\ldots,n\}$:
\begin{equation}\label{eq:system_alpha_2}
\begin{cases}
\sum_{m=k}^n (-1)^{m-k}\tilde \bI_{n,m}(2) \tilde \xi_{m,k} =  0, &\text{ for all } n\in\N, \; k\in \{1,\ldots,n-1\},\\
\tilde \xi_{n,n} = 1, \text{ for all } n\in \N.
\end{cases}
\end{equation}
By~\eqref{eq:I_n_k_tilde}, \eqref{eq:I_tilde_I_bold} and the Legendre duplication formula,  the coefficients of these equations are given by the formula
$$
\tilde \bI_{n,k}(2) = \binom {2k}{k} \frac{n!}{2^{n+k}} \frac{1}{(n-k)!(k-1)!} \int_{-\pi/2}^{+\pi/2} (\cos x)^{2k-1}(1+\sin x)^{n-k} \dd x.
$$
Proposition~\ref{prop:bJ_tilde_unique}  states  that
$$
\tilde \xi_{n,k} = \tilde \bJ_{n,k} \left(\frac {n+1}{2}\right)
$$
is the \textit{unique} solution to this system. To prove~\eqref{eq:bJ_alpha_2_proof} it therefore suffices to check that
$$
\tilde \xi_{n,k} =  \binom {2n}{n}^{-1} \left(\binom nk \binom{n+k}{k} - \binom {n-2} k \binom{n-2+k}{k}\right)
$$
defines a solution to~\eqref{eq:system_alpha_2}.
To this end, we introduce triangular arrays $\mtA[n,k]$ and $\mtB\{n,k\}$ as follows:
\begin{align}
\mtA[n,k]
&:=  \binom nk \binom{n+k}{k} \frac{k!}{4^k} =  \frac{(n+k)!}{4^k(n-k)!k!}, \quad n\in\N_0,\,k\in \{0,\ldots,n\}, \label{eq:def_A_alpha_2}\\
\mtB\{n,k\}
&:= \frac{2^{n-k}}{(n-k)!(k-1)!} \int_{-\pi/2}^{+\pi/2} (\cos x)^{2k-1}(1+\sin x)^{n-k} \dd x \notag\\
&\phantom{:}=\frac{2^{2n-1}(n-1)!}{(n-k)!(n+k-1)!}, \quad n\in\N,\, k\in \{1,\ldots,n\}.
\label{eq:def_B_alpha_2}
\end{align}
The integral in the definition of $\mtB\{n,k\}$ can be evaluated by the substitution $x = -\frac \pi 2 + 2 y$  which reduces it to the well-known~\cite[12.42]{whittaker_watson_book}  formula
$$
\int_{0}^{\pi/2} (\cos y)^{2k-1} (\sin y)^{2n-1} \dd y = \frac 12 \frac{\Gamma(n)\Gamma(k)}{\Gamma(n+k)}.
$$
We also put $\mtA[n,k] := \mtB\{n,k\} := 0$ if $k>n$. With this notation,
$$
\tilde \bI_{n,k}(2)
=
\frac{n!}{2^{2n}} \binom {2k}{k} \mtB\{n,k\}, \qquad n\in \N,\, k\in \{1,\ldots,n\}.
$$
We claim that
$$
\tilde \xi_{n,k} = \frac{4^k}{k!} \binom {2n}{n}^{-1} (\mtA[n,k] - \mtA[n-2,k]),
\qquad n\in \N,\, k\in \{1,\ldots,n\},
$$
defines a solution to~\eqref{eq:system_alpha_2}. Using~\eqref{eq:def_A_alpha_2} it is trivial to check that $\tilde \xi_{n,n} = 1$, and it remains to show that
$$
\sum_{m=k}^n (-1)^{m-k} \left(\frac{n!}{2^{2n}} \binom {2m}{m} \mtB\{n,m\} \right)  \left(\frac{4^k}{k!} \binom {2m}{m}^{-1} (\mtA[m,k] - \mtA[m-2,k])\right) =  0
$$
for all $n\in\N$ and $k\in \{1,\ldots,n-1\}$. After some cancellations, the identity simplifies to
\begin{equation}\label{eq:s_r_def_alpha_2}
\tilde s_{n-k}(k) := \sum_{m=k}^n  (-1)^{m-k} \mtB\{n,m\} (\mtA[m,k] - \mtA[m-2,k]) =  0,
\end{equation}
for all $n\in\N$ and $k\in \{1,\ldots,n-1\}$.
The proof of this identity is based on the following recurrence relations for $\mtA[n,k]$ and $\mtB\{n,k\}$:
\begin{align}
& \mtA[n,k] - \mtA[n-2,k] = \left(n - \frac 12\right) \mtA[n-1, k-1], \qquad n\in\N, \, k\in \{0,\ldots,n\},\label{eq:rec_A_alpha_2}\\
& \mtB\{n,k+2\}- \mtB\{n,k\} = -\left(\frac 12 + k\right) \mtB\{n+1,k+1\}, \qquad n\in \N, \, k\in \{0,\ldots,n\}.\label{eq:rec_B_alpha_2}
\end{align}
Here, we put $\mtA[n,-1]:=0$ and $\mtA[0,-1] = 2$, as well as $\mtB\{n,0\} := \mtB\{n,1\} = 2^{2n-1}/n!$ for $n\in\N$, which is quite natural in view of~\eqref{eq:def_B_alpha_2}.
Verifying both relations is an easy exercise.
Let $n\geq 2$ and $k\in \{1,\ldots,n-1\}$. Using first~\eqref{eq:rec_A_alpha_2} and then~\eqref{eq:rec_B_alpha_2}, we obtain
\begin{align*}
\tilde s_{n-k}(k)
&= \sum_{m=k}^n  (-1)^{m-k} \mtB\{n,m\} (\mtA[m,k] - \mtA[m-2,k]) \\
&=  \sum_{m=k}^n   (-1)^{m-k} \mtB\{n,m\} \left(m - \frac 12\right) \mtA[m-1,k-1]\\
&= - \sum_{m=k}^n  (-1)^{m-k}  (\mtB\{n-1,m+1\} - \mtB\{n-1,m-1\})  \mtA[m-1,k-1]\\
&=  \sum_{m=k}^n  (-1)^{m-k} \mtB\{n-1,m-1\}  \mtA[m-1,k-1] - \sum_{m=k}^n   (-1)^{m-k} \mtB\{n-1,m+1\}\mtA[m-1,k-1].
\end{align*}
After shifting the summation index in the second sum, we obtain
\begin{align*}
\tilde s_{n-k}(k)
&= \sum_{m=k}^n (-1)^{m-k} \mtB\{n-1,m-1\}  \mtA[m-1,k-1] - \sum_{m=k+2}^{n+2}  (-1)^{m-k} \mtB\{n-1,m-1\}\mtA[m-3,k-1]\\
&= \sum_{m=k}^n (-1)^{m-k} \mtB\{n-1,m-1\}  \mtA[m-1,k-1] - \sum_{m=k}^{n}  (-1)^{m-k} \mtB\{n-1,m-1\}\mtA[m-3,k-1],
\end{align*}
where the last identity holds because the terms in the second sum vanish for $m\in \{k,k+1,n+1,n+2\}$. Finally, taking both sums together and shifting the summation index one more time, we arrive at
\begin{align*}
\tilde s_{n-k}(k)
&=  \sum_{m=k}^{n} (-1)^{m-k}  \mtB\{n-1,m-1\}(\mtA[m-1,k-1] - \mtA[m-3,k-1] )\\
&=  \sum_{m=k-1}^{n-1} (-1)^{m-k+1}  \mtB\{n-1,m\}(\mtA[m,k-1] - \mtA[m-2,k-1])\\
&= \tilde s_{n-k}(k-1).
\end{align*}
Iterating this $k$ times, we obtain
$$
\tilde s_{n-k}(k)  =   \tilde s_{n-k}(0), \qquad k\in\N_0, \; n\geq k.
$$
To complete the proof, it remains to check that $\tilde s_r(0)=0$ for every $r\in\N$.  But this is trivial since $\mtA[m,0]=\mtA[m-2,0] = 1$ for $m\geq 2$ and hence
$$
\tilde s_r(0) = \sum_{m=0}^{r} (-1)^m \mtB\{r,m\} (\mtA[m,0] - \mtA[m-2,0]) =\mtB \{r,0\} \mtA[0,0] - \mtB\{r,1\}\mtA[1,0]=0,
$$
where we recall that $\mtA[0,0] = \mtA[1,0] = 1$ and $\mtB\{r,0\}=\mtB\{r,1\} = 2^{2r-1}/r!$.
The proof of Theorem~\ref{theo:bJ_alpha_2}  is complete. \hfill $\Box$

Let us record for future reference the following
\begin{proposition}\label{prop:inverse_matr_alpha_2}
Let $\mtA[n,k]$ and $\mtB\{n,k\}$ be defined by~\eqref{eq:def_A_alpha_2} and~\eqref{eq:def_B_alpha_2}.
Then, for every $n\in\N$ and $k\in \{1,\ldots,n\}$ we have
\begin{align}
\sum_{m=k}^n (-1)^{m-k} \mtB\{n,m\} \left(m - \frac 12\right) \mtA\left[m - 1, k - 1\right]
=
\delta_{nk},
\end{align}
where $\delta_{nk}$ denotes Kronecker's delta function. With other words, the infinite lower-triangular matrices $((-1)^m \mtB\{n,m\})_{n,m\in \N}$ and $((-1)^k (m - \frac 12) \mtA[m-1, k-1])_{m,k\in\N}$ are inverse to each other.
\end{proposition}
\begin{proof}
In view of~\eqref{eq:rec_A_alpha_2}, the left-hand side equals $\tilde s_{n-k}(k)$, the sum defined in~\eqref{eq:s_r_def_alpha_2}. For $k<n$ we have already shown that $\tilde s_{n-k}(k) = 0$. Let us verify that $\tilde s_0(k) = 1$ for all $k\in\N$. Indeed,
$$
\tilde s_0(k) =  \mtB\{k,k\} \mtA[k,k] = \frac{2^{2k-1}(k-1)!}{(2k-1)!} \cdot \frac{(2k)!}{4^k k!} = 1,
$$
where we used~\eqref{eq:def_B_alpha_2} and~\eqref{eq:def_A_alpha_2}.
\end{proof}

\begin{remark}
The above proof would not be possible without the pre-knowledge of the final formula for $\bJ_{n,k}(\frac{n+1}{2})$ stated in~\eqref{eq:bJ_alpha_2_proof}. In fact, we proceeded as follows.  Using computer algebra, we calculated the solution $\tilde \xi_{n,k}$ to~\eqref{eq:system_alpha_2} for small $n$ and then guessed~\eqref{eq:bJ_alpha_2_proof} using the Online Encyclopedia of Integer Sequences (OEIS)~\cite{sloane}. In a similar way, it is also possible to guess $\tilde \bJ_{n,k}(\frac n2)$ (which was done in~\cite{kabluchko_poisson_zero}; see  also Section~\ref{subsec:alpha_1}). As we shall see in Remark~\ref{rem:alpha_1_2_special} below, these two cases are distinguished by a very special property making them easy compared to the general case.
\end{remark}


\section{Proofs in the beta case}\label{sec:proof_beta}

\subsection{Recurrence relations for external quantities}
Fix some  $\alpha >0$ once and for all. Most functions we shall introduce depend on $\alpha$, although this dependence is usually suppressed in our notation. Define
\begin{equation}\label{eq:def_F}
F(x) = \int_{-\pi/2}^x (\cos y)^{\alpha} \dd y, \qquad x\in \R.
\end{equation}
Later, we shall need an extension of this definition to complex $x$. There are no problems if $\alpha$ is integer, but for non-integer $\alpha$ the function $(\cos x)^{\alpha}$ has branch points located at $\frac \pi 2 + \pi n$, $n\in\Z$. In order to define $F$ as a \textit{univalued} analytic function we agree to cut the complex plane at $(-\infty, -\frac \pi 2]$ and  $[+\frac \pi2,\infty)$. In~\eqref{eq:def_F}, we integrate along any contour which connects $-\frac \pi2$ to $x$ and stays in the doubly slit plane.

Next we are going to introduce quantities which are related to the expected external angles. For $\nu,\kappa\in \C$ such that $\Re \kappa > -\frac 1 \alpha$ and $\nu-\kappa\in \N_0$ we define
\begin{equation}\label{eq:def_B}
\lB\{\nu, \kappa\}
:=
\frac{\alpha^{\nu-\kappa}}{(\nu-\kappa)!}\int_{-\pi/2}^{+\pi/2} (\cos x)^{\alpha \kappa}
(F(x))^{\nu-\kappa} \dd x.
\end{equation}
By convention, we also put $\lB \{\nu,\kappa\} = 0$ if $\nu-\kappa \in \{-1,-2,\ldots\}$.  The expected external angles of beta simplices can be expressed through the $\lB\{\nu,\kappa\}$'s since by~\eqref{eq:I_n_k} and~\eqref{eq:I_tilde_I_bold},
\begin{align}
\bI_{n,k}(\alpha)
&=
\binom nk  c_{\frac{\alpha k-1}{2}} c_{\frac{\alpha-1}{2}}^{n-k} \int_{-\pi/2}^{+\pi/2} (\cos x)^{\alpha k} (F(x))^{n-k} \dd x\notag \\
&=
\binom nk  c_{\frac{\alpha k-1}{2}} c_{\frac{\alpha-1}{2}}^{n-k} \cdot \frac{(n-k)!}{\alpha^{n-k}}\lB\{n,k\}\notag \\
&=
\alpha^{-n} n! c_{\frac{\alpha-1}{2}}^{n} \cdot \frac{c_{\frac{\alpha k-1}{2}} \alpha^{k}}{c_{\frac{\alpha-1}{2}}^k k!} \cdot \lB\{n,k\}.\label{eq:I_n_m_b_n_m}
\end{align}
The quantities $\lB\{\nu,\kappa\}$ satisfy a recurrence relation which will be crucial for what follows.
\begin{proposition}\label{prop:B_relation}
For $\nu,\kappa \in \C$ such that $\Re \kappa > \frac 1\alpha$ and $\nu-\kappa\in \N_0$, we have
\begin{equation}\label{eq:B_relation}
\lB\{\nu,\kappa+2\} + (\kappa+1)\kappa\, \lB\{\nu,\kappa\}
=
\left(\kappa - \frac 1\alpha\right) (\kappa+1)\, \lB \left\{\nu - \frac 2\alpha, \kappa - \frac 2 \alpha\right\}.
\end{equation}
\end{proposition}
\begin{proof}
Let us first give a proof assuming that $\nu - \kappa \in \{2,3,\ldots\}$.
Writing $\dd (F(x))^{\nu-\kappa-1} = (\nu-\kappa-1) (F(x))^{\nu-\kappa-2} (\cos x)^{\alpha} \dd x$ and integrating by parts, we obtain
\begin{align*}
\lB\{\nu, \kappa+2\}
&=
\frac {\alpha^{\nu-\kappa-2}} { (\nu-\kappa-1)!} \int_{-\pi/2}^{+\pi/2} (\cos x)^{\alpha \kappa + \alpha} \dd (F(x))^{\nu-\kappa-1}\\
&=
\frac {(\kappa+1)\alpha^{\nu-\kappa-1}} {(\nu-\kappa-1)!} \int_{-\pi/2}^{+\pi/2} (F(x))^{\nu-\kappa-1} (\cos x)^{\alpha \kappa + \alpha - 1} (\sin x) \dd x.
\end{align*}
The boundary term in the partial integration formula vanishes since $\alpha \kappa + \alpha >0$. Writing $\dd (F(x))^{\nu-\kappa} = (\nu-\kappa) (F(x))^{\nu-\kappa-1} (\cos x)^{\alpha} \dd x$ and again integrating  by parts, we obtain
\begin{align*}
\lB \{\nu, \kappa+2\}
&=
\frac {(\kappa + 1) \alpha^{\nu-\kappa-1}} {(\nu-\kappa)!}
\int_{-\pi/2}^{+\pi/2} (\cos x)^{\alpha \kappa - 1 } (\sin x) \dd (F(x))^{\nu-\kappa}\\
&=
- \frac {(\kappa + 1) \alpha^{\nu-\kappa-1}} {(\nu-\kappa)!}
\int_{-\pi/2}^{+\pi/2} (F(x))^{\nu-\kappa}
\left((\cos x)^{\alpha \kappa-1} \sin x\right)' \dd x\\
&=
 \frac {(\kappa + 1) \alpha^{\nu-\kappa-1}} {(\nu-\kappa)!}
\int_{-\pi/2}^{+\pi/2} (F(x))^{\nu-\kappa}
\left((\alpha \kappa - 1)(\cos x)^{\alpha \kappa - 2} \sin^2 x -(\cos x)^{\alpha \kappa} \right) \dd x.
\end{align*}
Again, the boundary term in partial integration vanishes since  $\alpha \kappa -1 >0$ by our assumptions. Using the identity $\sin^2 x = 1-\cos^2 x$  we can write
\begin{align*}
\lB \{\nu, \kappa+2\}
&=
\frac {(\kappa + 1) \alpha^{\nu-\kappa-1}} {(\nu-\kappa)!}
\int_{-\pi/2}^{+\pi/2} (F(x))^{\nu-\kappa}
(\alpha\kappa -1) (\cos x)^{\alpha \kappa-2} \dd x\\
&\phantom{=}
-
\frac {(\kappa + 1) \alpha^{\nu-\kappa-1}} {(\nu-\kappa)!}
\int_{-\pi/2}^{+\pi/2} (F(x))^{\nu-\kappa}
(\alpha \kappa)(\cos x)^{\alpha \kappa}\dd x.
\end{align*}
Recalling the definition of $\lB\{\nu,\kappa\}$, we arrive at
$$
\lB\{\nu,\kappa+2\}
=
\left(\kappa - \frac 1\alpha\right) (\kappa+1)\, \lB \left\{\nu - \frac 2\alpha, \kappa - \frac 2 \alpha\right\}
-
(\kappa+1)\kappa\, \lB\{\nu,\kappa\}
,
$$
which proves the claim under the assumption $\nu-\kappa \in \{2,3,\ldots\}$.  In the cases when $\nu=\kappa$ and $\nu= \kappa+1$, Equation~\eqref{eq:B_relation} takes the form
\begin{align}
\kappa\, \lB\{\kappa,\kappa\}
&=
\left(\kappa - \frac 1\alpha\right) \, \lB \left\{\kappa - \frac 2\alpha, \kappa - \frac 2 \alpha\right\},\label{eq:b_rec_diagonal1}\\
\kappa\, \lB\{\kappa+1,\kappa\}
&=
\left(\kappa - \frac 1\alpha\right)\, \lB \left\{\kappa - \frac 2\alpha + 1, \kappa - \frac 2 \alpha\right\},\label{eq:b_rec_diagonal2}
\end{align}
respectively, where we recall that $\lB\{\kappa,\kappa+2\}=\lB\{\kappa+1,\kappa+2\}=0$ by definition. Identity~\eqref{eq:b_rec_diagonal1} is a direct consequence of the definition of $\lB\{\kappa,\kappa\}$ since by~\eqref{eq:def_B},
\begin{equation}\label{eq:b_kappa_kappa}
\lB\{\kappa, \kappa\}
=
\int_{-\pi/2}^{+\pi/2} (\cos x)^{\alpha \kappa} \dd x
=
\frac{\sqrt \pi\, \Gamma\left(\frac{\alpha \kappa + 1}{2}\right)}{\Gamma\left(\frac{\alpha \kappa + 2}{2}\right)}.
\end{equation}
To prove~\eqref{eq:b_rec_diagonal2}, we again recall~\eqref{eq:def_B}:
\begin{align}
\lB\{\kappa+1, \kappa\}
&=
\alpha\, \int_{-\pi/2}^{+\pi/2} (\cos x)^{\alpha \kappa}
F(x)\, \dd x
=
\alpha\, \int_{-\pi/2}^{+\pi/2} (\cos x)^{\alpha \kappa}
F(0)\, \dd x\notag\\
&=
\frac \alpha 2\cdot \frac{\sqrt \pi\, \Gamma\left(\frac{\alpha + 1}{2}\right)}{\Gamma\left(\frac{\alpha + 2}{2}\right)} \cdot  \frac{\sqrt \pi\, \Gamma\left(\frac{\alpha \kappa + 1}{2}\right)}{\Gamma\left(\frac{\alpha \kappa + 2}{2}\right)}
, \label{eq:b_kappa+1_kappa}
\end{align}
where we used that $(\cos x)^{\alpha \kappa} (F(x)-F(0))$ is an odd function implying that its integral over $[-\frac\pi2, +\frac \pi 2]$ vanishes.
\end{proof}

\subsection{Recurrence relations for internal quantities}
In view of what has been said in Section~\ref{subsec:method}, it would be natural to proceed as follows. First, construct a meromorphic continuation of $\lB\{\nu, \kappa\}$, considered as a function of $\kappa$ with fixed $r:=\nu-\kappa\in\N_0$, to the whole complex plane. For example, in the cases $r=0$ and $r=1$ the meromorphic continuation was already given  in~\eqref{eq:b_kappa_kappa} and~\eqref{eq:b_kappa+1_kappa}.  Introduce the quantities
\begin{equation*}
\lA_1[\nu, \kappa] := \lB\{-\kappa, -\nu\}.
\end{equation*}
Then, one could conjecture that the expected internal angles can be expressed through $\lA_1[\nu,\kappa]$. After having tried this approach out, we convinced ourselves that it works only under certain restrictions on $\alpha$ (namely, one has to assume that $\alpha$ is integer and that certain parity assumptions hold).  Since we need a proof for arbitrary $\alpha>0$, we refrain from presenting the details of meromorphic continuation of $\lB\{\nu, \kappa\}$ here and proceed in a slightly different way.
For $\nu,\kappa\in \C$ such that $\nu-\kappa\in \N_0$ and $\Re \kappa > 0$ we define
\begin{equation}\label{eq:def_A_0}
\lA[\nu,\kappa] := \frac{\alpha^{\nu - \kappa + 1}}{(\nu - \kappa)!} \cdot \frac 1 {2\pi \ii} \int_{- \ii \infty}^{+\ii \infty} (\cos x)^{-\alpha \nu} (F(x))^{\nu-\kappa} \dd x.
\end{equation}
As it turns out, the function $a$ is more suitable for what follows than $a_1$. Parametrizing the contour of integration by $x= \ii u$ with $u\in\R$, we can equivalently write
\begin{equation}\label{eq:def_A_0_rep}
\lA[\nu,\kappa] = \frac{\alpha^{\nu - \kappa + 1}}{(\nu - \kappa)!} \cdot \frac 1 {2\pi} \int_{-\infty}^{+\infty} (\cosh u)^{-\alpha \nu} (F(\ii u))^{\nu-\kappa} \dd u.
\end{equation}
Our conditions on $\nu$ and $\kappa$ ensure that this integral converges absolutely. Indeed, using~\eqref{eq:def_F} and the L'Hospital rule it is easy to check that
\begin{equation}\label{eq:F_asympt}
F(\ii u) = \frac{\ii}{\alpha 2^{\alpha}} (\sgn u) \eee^{\alpha |u|} (1+o(1)), \qquad \text{as } u\to \pm \infty.
\end{equation}
We also put $\lA \{\nu,\kappa\} = 0$ if $\nu-\kappa \in \{-1,-2,\ldots\}$. The quantities $\lA[\nu,\kappa]$ satisfy recurrence relations which are, in some sense, dual to those satisfied by $\lB\{\nu,\kappa\}$.
\begin{proposition}\label{prop:A_0_relations}
For $\nu,\kappa\in \C$ such that $\nu-\kappa\in \N_0$ and $\Re \kappa > 0$, we have
\begin{align}
\lA[\nu-2,\kappa] +(\nu-1)\nu \, \lA[\nu,\kappa]
&=
\left(\nu+\frac 1 \alpha\right) (\nu-1) \,\lA \left[\nu + \frac 2\alpha, \kappa + \frac 2 \alpha\right]. \label{eq:A_0_relation}
\end{align}
\end{proposition}
\begin{remark}
Comparing the definitions of $\lB\{\nu,\kappa\}$ and $\lA[\nu,\kappa]$ given in~\eqref{eq:def_B} and~\eqref{eq:def_A_0}, we detect only the following three differences. Firstly, the integrand in~\eqref{eq:def_B} turns into the integrand in~\eqref{eq:def_A_0} if we replace $(\nu, \kappa)$ by $(-\kappa, -\nu)$. Secondly, the integration contours are different. Finally, there is an additional factor of $\alpha/(2\pi \ii)$ in~\eqref{eq:def_A_0}, but it does not influence the recurrence relation. Since the shape of the integration contour is completely irrelevant for the proof Proposition~\ref{prop:B_relation} provided we can check that the boundary terms in the  partial integration formula vanish, it should not be surprising that the recurrence relation for $\lA[\nu,\kappa]$ has the same form as the recurrence relation for $\lB\{-\kappa,-\nu\}$.
\end{remark}
\begin{proof}[Proof of Proposition~\ref{prop:A_0_relations}]
Let us first assume that $\nu-\kappa \in \{2,3,\ldots\}$.
Writing $\dd (F(x))^{\nu-\kappa - 1} = (\nu-\kappa - 1) (F(x))^{\nu-\kappa-2} (\cos x)^{\alpha} \dd x$ and using integration by parts, we obtain
\begin{align*}
\lA[\nu-2, \kappa]
&=
\frac {\alpha^{\nu-\kappa-1}} { (\nu-\kappa-1)!} \frac 1 {2\pi \ii} \int_{- \ii \infty}^{+\ii \infty} (\cos x)^{-\alpha \nu + \alpha} \dd (F(x))^{\nu-\kappa-1}\\
&=
-\frac {(\nu-1)\alpha^{\nu-\kappa}} {(\nu-\kappa-1)!} \frac 1 {2\pi \ii} \int_{- \ii \infty}^{+\ii \infty} (F(x))^{\nu-\kappa-1} (\cos x)^{-\alpha \nu +\alpha - 1} (\sin x) \dd x.
\end{align*}
The boundary terms in the partial integration formula vanish since $\Re \kappa>0$; see~\eqref{eq:F_asympt}.  Writing $\dd (F(x))^{\nu-\kappa} = (\nu-\kappa) (F(x))^{\nu-\kappa-1} (\cos x)^{\alpha} \dd x$ and again integrating  by parts, we obtain
\begin{align*}
\lA [\nu-2, \kappa]
&=
-\frac {(\nu - 1) \alpha^{\nu-\kappa}} {(\nu-\kappa)!}
\frac 1 {2\pi \ii} \int_{- \ii \infty}^{+\ii \infty} (\cos x)^{-\alpha \nu - 1 } (\sin x) \dd (F(x))^{\nu-\kappa}\\
&=
\frac {(\nu - 1) \alpha^{\nu-\kappa}} {(\nu-\kappa)!}
\frac 1 {2\pi \ii} \int_{- \ii \infty}^{+\ii \infty} (F(x))^{\nu-\kappa}
\left((\cos x)^{-\alpha \nu -1} \sin x\right)' \dd x\\
&=
 \frac {(\nu - 1) \alpha^{\nu-\kappa}} {(\nu-\kappa)!}
\frac 1 {2\pi \ii} \int_{- \ii \infty}^{+\ii \infty} (F(x))^{\nu-\kappa}
\left((\alpha \nu  + 1)(\cos x)^{-\alpha \nu  - 2} \sin^2 x +(\cos x)^{-\alpha \nu} \right) \dd x.
\end{align*}
Again, the boundary terms in partial integration vanish since  $\Re \kappa>0$; see~\eqref{eq:F_asympt}. Using the identity $\sin^2 x = 1-\cos^2 x$  we can write
\begin{align*}
\lA [\nu-2, \kappa]
&=
\frac {(\nu - 1) \alpha^{\nu-\kappa}} {(\nu-\kappa)!}
\frac 1 {2\pi \ii} \int_{- \ii \infty}^{+\ii \infty} (F(x))^{\nu-\kappa}
(\alpha\nu  + 1) (\cos x)^{- \alpha \nu - 2} \dd x\\
&\phantom{=}
-
\frac {(\nu - 1) \alpha^{\nu-\kappa}} {(\nu-\kappa)!}
\frac 1 {2\pi \ii} \int_{- \ii \infty}^{+\ii \infty}(F(x))^{\nu-\kappa}
(\alpha \nu) (\cos x)^{-\alpha \nu}\dd x.
\end{align*}
Recalling the definition of $\lA[\nu,\kappa]$, we arrive at
$$
\lA[\nu-2,\kappa]
=
\left(\nu + \frac 1\alpha\right) (\nu - 1) \lA \left[\nu + \frac 2\alpha, \kappa + \frac 2 \alpha\right]
-
(\nu-1)\nu\, \lA[\nu,\kappa],
$$
which completes the proof if $\nu-\kappa \in \{2,3,\ldots\}$. It remains to consider the cases $\nu = \kappa$ and $\nu = \kappa + 1$. The identities we need to prove take the form
\begin{align}
\kappa \, \lA[\kappa,\kappa]
&=
\left(\kappa+\frac 1 \alpha\right) \,\lA \left[\kappa + \frac 2\alpha, \kappa + \frac 2 \alpha\right],\label{eq:a_kappa_kappa}\\
(\kappa+1) \, \lA[\kappa+1,\kappa]
&=
\left(\kappa+\frac 1 \alpha + 1\right) \lA \left[\kappa + \frac 2\alpha + 1, \kappa + \frac 2 \alpha\right].\label{eq:a_kappa+1_kappa}
\end{align}
To prove~\eqref{eq:a_kappa_kappa} observe that by~\eqref{eq:def_A_0_rep} and~\eqref{eq:int_cos_cosh},
$$
\lA[\kappa,\kappa]
=
\frac \alpha {2\pi} \int_{-\infty}^{+\infty} (\cosh u)^{-\alpha \kappa}\dd u
=
\frac \alpha 2 \cdot \frac{\Gamma\left(\frac{\alpha \kappa}{2}\right)}{\sqrt \pi\,  \Gamma\left(\frac{\alpha \kappa+1}{2}\right)}.
$$
Identity~\eqref{eq:a_kappa+1_kappa} follows from the formula
$$
\lA[\kappa+1,\kappa]  = \frac \alpha 2 \cdot \frac{\Gamma\left(\frac{\alpha+1}{2}\right)}{\Gamma\left(\frac{\alpha}{2}\right)} \cdot \frac{\Gamma\left(\frac{\alpha \kappa + \alpha}{2}\right)}{\Gamma\left(\frac{\alpha \kappa + \alpha +  1}{2}\right)},
$$
which can be obtained as follows. Observe that $\overline{F(\ii u)} = F(-\ii u)$ and $\Re F(\ii u) = F(0) = \frac{\sqrt{\pi}\, \Gamma(\frac{\alpha+1}{2})}{\alpha \, \Gamma(\frac \alpha 2)}$. Hence,
$$
\lA[\kappa+1,\kappa]
=
\frac{\alpha^{2}}{2\pi} \int_{-\infty}^{+\infty} (\cosh u)^{-\alpha \kappa - \alpha} F(\ii u) \dd u
=
\frac{\alpha^2}{2\pi} F(0) \int_{-\infty}^{+\infty} (\cosh u)^{-\alpha \kappa - \alpha} \dd u,
$$
and the claim follows.
\end{proof}

\subsection{Periodicity}\label{subsec:periodicity}
Fix some $r\in\N_0$ and consider the expression
\begin{equation}\label{eq:def_s_r}
s_r(\kappa)
:=
\sum_{\mu = \kappa}^{\kappa + r}
(-1)^{\mu - \kappa}
\lB\{\kappa + r, \mu\}
\left(\mu + \frac 1\alpha\right) \lA\left[\mu+\frac 2\alpha, \kappa+\frac 2\alpha\right].
\end{equation}
By convention, the sum is taken over $\mu\in \{\kappa, \kappa+1, \ldots, \kappa + r\}$. The involved $\lB$- and $\lA$-terms are well-defined if assume that $\Re \kappa > - \frac 1\alpha$. In this range, $s_r(\kappa)$ is an analytic function of $\kappa$.   Our aim is to prove that $s_r(\kappa) = 0$, for all $r\in\N$.    As a first step,  we prove that $s_r$ is a periodic function.
\begin{proposition}\label{prop:periodic}
Let $\alpha > 0$ and $r\in\N_0$. Then, for every $\kappa\in\C$ with $\Re \kappa>\frac 1 \alpha$, we have
$$
s_{r}(\kappa) = s_r \left(\kappa - \frac 2 \alpha\right).
$$
\end{proposition}
\begin{proof}
The only properties of $\lA$ and $\lB$ we shall rely on in this proof are the recurrence relations stated in Propositions~\ref{prop:B_relation} and~\ref{prop:A_0_relations}.
It will be convenient to work with the quantities
\begin{equation}
\rB\{\nu,\kappa\} := \frac{\lB\{\nu,\kappa\}}{\Gamma(\kappa)},
\qquad
\rA[\nu,\kappa] := \lA[\nu,\kappa] \cdot \Gamma(\nu+1).
\end{equation}
With this notation, the recurrence relations of Propositions~\ref{prop:B_relation} and~\ref{prop:A_0_relations} take the form
\begin{align}
&\rB\{\nu,\kappa+2\} + \rB\{\nu,\kappa\}
=
\frac{\left(\kappa - \frac 1\alpha\right) \Gamma\left(\kappa - \frac 2\alpha\right)} {\Gamma(\kappa+1)} \rB \left\{\nu - \frac 2\alpha, \kappa - \frac 2 \alpha\right\},\label{eq:b_relation}\\
&\rA[\nu-2,\kappa] + \rA[\nu,\kappa]
=
\frac{\left(\nu+\frac 1 \alpha\right)\Gamma(\nu)}{\Gamma\left(\nu+\frac 2 \alpha + 1\right)} \rA \left[\nu + \frac 2\alpha, \kappa + \frac 2 \alpha\right]. \label{eq:a_0_relation}
\end{align}
The definition~\eqref{eq:def_s_r} of $s_r(\kappa)$ takes the form
\begin{equation}\label{eq:def_s_r_new}
s_r(\kappa)
=
\sum_{\mu = \kappa}^{\kappa + r}
(-1)^{\mu - \kappa}
\rB\{\kappa + r, \mu\}
\frac{\left(\mu+\frac 1\alpha\right)\Gamma(\mu)}{\Gamma \left(\mu+\frac 2\alpha+1\right)}
 \rA\left[\mu+\frac 2\alpha, \kappa+\frac 2\alpha\right].
\end{equation}
In view of~\eqref{eq:a_0_relation}, we can write~\eqref{eq:def_s_r_new} as
$$
s_r(\kappa)
=
\sum_{\mu = \kappa}^{\kappa + r}
(-1)^{\mu - \kappa}
\rB\{\kappa + r, \mu\}
(\rA[\mu-2,\kappa] + \rA[\mu,\kappa]).
$$
We now regroup the terms by applying Abel's partial summation. Splitting the sum into two sums and introducing in the first sum the new index of
summation $\mu':= \mu - 2$, we obtain
$$
s_r(\kappa)
=
\sum_{\mu' = \kappa-2}^{\kappa + r - 2}
(-1)^{\mu' - \kappa}
\rB\{\kappa + r, \mu'+2\} \rA[\mu',\kappa] +
\sum_{\mu = \kappa}^{\kappa + r}
(-1)^{\mu - \kappa}
\rB\{\kappa + r, \mu\}
\rA[\mu,\kappa].
$$
To the first sum we add two vanishing terms and  subtract two vanishing terms as follows:
$$
s_r(\kappa)
=
\sum_{\mu' = \kappa}^{\kappa + r}
(-1)^{\mu' - \kappa}
\rB\{\kappa + r, \mu'+2\} \rA[\mu',\kappa] +
\sum_{\mu = \kappa}^{\kappa + r}
(-1)^{\mu - \kappa}
\rB\{\kappa + r, \mu\}
\rA[\mu,\kappa].
$$
Replacing $\mu'$ by $\mu$ and  taking the sums together,  we obtain
$$
s_r(\kappa)
=
\sum_{\mu = \kappa}^{\kappa + r}
(-1)^{\mu - \kappa}
(\rB\{\kappa + r, \mu+2\}+\rB\{\kappa + r, \mu\})  \rA[\mu,\kappa].
$$
Applying Relation~\eqref{eq:b_relation} to the $\rB$-terms, we obtain
$$
s_r(\kappa)
=
\sum_{\mu = \kappa}^{\kappa + r}
(-1)^{\mu - \kappa}
\rB \left\{\kappa+r - \frac 2\alpha, \mu - \frac 2 \alpha\right\}
\frac{\left(\mu - \frac 1\alpha\right) \Gamma\left(\mu - \frac 2\alpha\right)} {\Gamma(\mu+1)} \rA[\mu,\kappa].
$$
Writing $\mu^*:= \mu - \frac 2\alpha$ we finally arrive at
$$
s_r(\kappa)
=
\sum_{\mu^* = \kappa-\frac 2\alpha}^{\kappa-\frac 2\alpha + r}
(-1)^{\mu^* - \left(\kappa-\frac 2\alpha\right)}
\rB \left\{\kappa+r - \frac 2\alpha, \mu^*\right\}
\frac{\left(\mu^* + \frac 1\alpha\right) \Gamma(\mu^*)} {\Gamma\left(\mu^*+ \frac 2 \alpha + 1\right)} \rA\left[\mu^*+\frac 2\alpha,\kappa\right].
$$
In view of~\eqref{eq:def_s_r_new}, the right-hand side equals $s_r(\kappa - \frac 2\alpha)$, which completes the proof.
\end{proof}

\begin{remark}
Although we shall not need this, let us mention that the same argument, with trivial modifications, shows that the unsigned version of $s_r(\kappa)$ defined by
\begin{align*}
s_r^{+} (\kappa)
&:=  \sum_{\mu = \kappa}^{\kappa + r}   \lB\{\kappa + r, \mu\}
\left(\mu + \frac 1\alpha\right) \lA\left[\mu+\frac 2\alpha, \kappa+\frac 2\alpha\right]
\end{align*}
is also periodic, namely
$$
s_{r}^+ (\kappa) =  s_r^+ \left(\kappa - \frac 2 \alpha\right).
$$
\end{remark}

\begin{remark}\label{rem:alpha_1_2_special}
If $\alpha = 1$ or $\alpha=2$, Proposition~\ref{prop:periodic} can be used to express $s_r(k)$ (for $k\in\N$) through $s_r(0)$ or $s_r(1)$. In the setting of beta' simplices, this was the way we obtained combinatorial formulae for $\tilde \bJ_{n,k}(\frac{\alpha + n - 1}{2})$ in  Section~\ref{sec:simplest_case_proof} (for $\alpha=2$) and in~\cite{kabluchko_poisson_zero} (for $\alpha=1$, see also Section~\ref{subsec:alpha_1}).  Unfortunately, in the setting of beta simplices, the values $\alpha=1$ and $\alpha=2$ are not admissible as arguments of $\bJ_{n,k}(\frac{\alpha - n + 1}{2})$ (except for small $n$) because of the restriction $\alpha \geq n-3$.
On the other hand, except for the special cases $\alpha=1$ and $\alpha=2$, any attempts to decrease the argument of $s_r(k)$ using Proposition~\ref{prop:periodic} do not seem useful. Instead, we shall let the argument go to $\infty$.
\end{remark}

\subsection{Asymptotics}
Our aim is now to prove that $s_r(\kappa) = 0$  for all $r\in\N$ and all $\kappa\in\C$ with $\Re \kappa > - \frac 1\alpha$.
We observe that Proposition~\ref{prop:periodic} implies that for every $T\in\N$,
\begin{equation}\label{eq:s_r_periodic_iterated}
s_{r}(\kappa)
=
s_r \left(\kappa + \frac 2 \alpha T\right).
\end{equation}
Now, we are going to let $T\to +\infty$ and prove that the limit of the right-hand side is $0$. To this end, we need to investigate the asymptotic properties of $\lB\{\nu, \kappa\}$ and $\lA[\nu, \kappa]$. This is done in the next two  lemmas. As usual, we  write $f_1(T)\sim f_2(T)$ if the quotient $f_1(T)/f_2(T)$ converges to $1$ as the real variable $T$ goes to $+\infty$.

\begin{lemma}\label{lem:B_asymptotic}
Fix some $\kappa_1,\kappa_2\in\C$  such that $\kappa_1-\kappa_2\in\N_0$. Then, as $T\to+\infty$, we have
$$
\lB \left\{\kappa_1 +\frac 2 \alpha T, \kappa_2 + \frac 2 \alpha T\right\}
\sim
\frac{(\alpha F(0))^{\kappa_1-\kappa_2}}{(\kappa_1-\kappa_2)!} \sqrt{\frac {\pi}{T}}.
$$
\end{lemma}
\begin{proof}
Recall from~\eqref{eq:def_B} that
$$
\lB \left\{\kappa_1 +\frac 2 \alpha T, \kappa_2 + \frac 2 \alpha T\right\}
= \frac{\alpha^{\kappa_1-\kappa_2}}{(\kappa_1-\kappa_2)!}\int_{-\pi/2}^{+\pi/2} (\cos x)^{\alpha \kappa_2 + 2T} (F(x))^{\kappa_1-\kappa_2} \dd x.
$$
The standard Laplace asymptotics, see, e.g., \cite[Theorem~B.7 on p.~758]{Flajolet_book}, yields
\begin{align*}
\int_{-\pi/2}^{+\pi/2} (\cos x)^{\alpha \kappa_2 + 2T} (F(x))^{\kappa_1-\kappa_2} \dd x
&=
\int_{-\pi/2}^{+\pi/2} \eee^{2T \log \cos x} \cdot  (\cos x)^{\alpha \kappa_2} (F(x))^{\kappa_1-\kappa_2}\dd x\\
&\sim
\sqrt {\frac {\pi} {T}} (F(0))^{\kappa_1-\kappa_2},
\end{align*}
as $T\to +\infty$,
because the maximum of the function $\varphi(x) := \log \cos x$ on $[-\frac \pi 2, +\frac \pi 2]$ is $\varphi(0) = 0$ and the second derivative is $\varphi''(0) = -1$.
\end{proof}

\begin{lemma}\label{lem:A_0_asymptotic}
Fix some $\kappa_1,\kappa_2\in\C$ such that $\kappa_1-\kappa_2\in\N_0$.  Then, as $T\to+\infty$, we have
$$
\lA\left[\kappa_1 +\frac 2 \alpha T, \kappa_2 + \frac 2 \alpha T\right]
\sim
\frac \alpha 2 \cdot \frac{(\alpha F(0))^{\kappa_1-\kappa_2}}{(\kappa_1-\kappa_2)!}\cdot \frac {1}{\sqrt{\pi T}}.
$$
\end{lemma}
\begin{proof}
Recall from~\eqref{eq:def_A_0_rep} that
$$
\lA\left[\kappa_1 +\frac 2 \alpha T, \kappa_2 + \frac 2 \alpha T\right]
=
\frac{\alpha^{\kappa_1 - \kappa_2 + 1}}{(\kappa_1 - \kappa_2)!} \cdot \frac 1 {2\pi}
\int_{-\infty}^{+\infty} (\cosh u)^{-\alpha \kappa_1 - 2T} (F(\ii u))^{\kappa_1-\kappa_2} \dd u.
$$
The Laplace method, see, e.g., \cite[Theorem~B.7 on p.~758]{Flajolet_book}, yields
\begin{align*}
\int_{-\infty}^{+\infty} (\cosh u)^{-\alpha \kappa_1 - 2T} (F(\ii u))^{\kappa_1-\kappa_2} \dd u
&=
\int_{-\infty}^{+\infty} \eee^{-2T \log \cosh u} \cdot (\cosh u)^{-\alpha \kappa_1}  (F(\ii u))^{\kappa_1-\kappa_2} \dd u\\
&\sim
\sqrt {\frac {\pi} {T}} (F(0))^{\kappa_1-\kappa_2},
\end{align*}
as $T\to +\infty$, because the maximum of the function $\phi(x) := - \log \cosh x$ on $\R$ is $\phi(0) = 0$ and the second derivative is $\phi''(0) = -1$.
\end{proof}

\subsection{Evaluating the sum}
We are now ready to compute the sum $s_r$.
\begin{proposition}\label{prop:s_r_vanishes}
Let $\alpha>0$.  Then, for all $\kappa\in\C$ with $\Re \kappa > - \frac 1\alpha$, we have
$$
s_r(\kappa) =
\begin{cases}
1, &\text{ if } r=0,\\
0, &\text{ if } r\in\N.
\end{cases}
$$
\end{proposition}
\begin{proof}
 Using~\eqref{eq:s_r_periodic_iterated} and then the definition for $s_r$ given in~\eqref{eq:def_s_r}, we can write
\begin{align*}
s_{r}(\kappa)
&=
s_r \left(\kappa + \frac 2 \alpha T\right)\\
&=
\sum_{\mu = \kappa + \frac 2\alpha T}^{\kappa + \frac 2\alpha T + r}
(-1)^{\mu - \kappa - \frac 2\alpha T}
\lB\left\{\kappa + r + \frac 2 \alpha T, \mu\right\}
\left(\mu + \frac 1\alpha\right) \lA\left[\mu + \frac 2\alpha, \kappa+ \frac 2 \alpha T + \frac 2\alpha\right]
\end{align*}
Introducing the new summation index $\ell$ by $\mu = \kappa + \ell + \frac 2 \alpha T$, we get
\begin{multline*}
s_{r}(\kappa)
=
\sum_{\ell = 0}^{r}
(-1)^{\ell}\lB\left\{\kappa + r + \frac 2 \alpha T, \kappa + \ell + \frac 2 \alpha T \right\}\\
\times
\left(\kappa + \ell + \frac 2 \alpha T + \frac 1\alpha\right) \lA\left[\kappa + \ell + \frac 2 \alpha T + \frac 2\alpha, \kappa+ \frac 2 \alpha T +\frac 2\alpha\right]
.
\end{multline*}
We now let $T\to+\infty$ along the integers.  By Lemma~\ref{lem:B_asymptotic}, we have
$$
\lB\left\{\kappa +r + \frac 2 \alpha T , \kappa  +\ell + \frac 2 \alpha T\right\}
\sim
\frac{(\alpha F(0))^{r-\ell}}{(r-\ell)!} \sqrt{\frac {\pi}{T}}.
$$
Furthermore, by Lemma~\ref{lem:A_0_asymptotic}, we have
\begin{align*}
\lA\left[\kappa +  \ell + \frac 2 \alpha T   + \frac 2 \alpha, \kappa+\frac 2 \alpha T  + \frac 2 \alpha\right]
\sim
\frac \alpha 2 \cdot \frac{(\alpha F(0))^{\ell}}{\ell!}\cdot \frac {1}{\sqrt{\pi T}}.
\end{align*}
Taking everything together,  we arrive at
\begin{align*}
s_{r}(\kappa)
&=
\lim_{T\to+\infty}  s_r \left(\kappa + \frac 2 \alpha T\right)\\
&=
\frac \alpha 2 \cdot (\alpha  F(0))^r    \lim_{T\to+\infty}  \frac {1} T \sum_{\ell = 0}^{r} \left(\kappa + \ell + \frac 2 \alpha T  + \frac 1 \alpha \right)  \frac {(-1)^{\ell}} {\ell! (r-\ell)!}\\
&=
(\alpha  F(0))^r  \sum_{\ell = 0}^{r} \frac {(-1)^{\ell}} {\ell! (r-\ell)!}\\
&=
\frac{(\alpha  F(0))^r}{r!}  \sum_{\ell = 0}^{r}  (-1)^{\ell} \binom r \ell.
\end{align*}
For $r\in\N$, the sum on the right-hand side equals $(1-1)^r=0$. For $r=0$, the right-hand side equals $1$.
\end{proof}

\begin{remark}\label{rem:s_r_+}
A similar argument shows that for all $r\in\N_0$, the unsigned version of $s_r$ satisfies
$$
s_r^+ (\kappa) =  \frac{(2\alpha F(0))^r}{r!}.
$$
The only essential difference is that at the very end we have to replace $(1-1)^r$ by $(1+1)^r$.
\end{remark}





As a special case of Proposition~\ref{prop:s_r_vanishes} and Remark~\ref{rem:s_r_+}, we obtain the following non-linear relations between $\lA[n,k]$ and $\lB\{n,k\}$. The first of them reminds the property of Stirling numbers stated in~\eqref{eq:stirling_inverse}.
\begin{proposition}\label{prop:relations_A_B}
Let $\alpha>0$.  For every $n\in\N$ and $k\in \{1,\ldots,n\}$ we have
\begin{align}
&\sum_{m=k}^n (-1)^{m-k} \lB\{n,m\} \left(m+\frac 1\alpha\right) \lA\left[m+\frac 2 \alpha, k+\frac 2 \alpha\right]
=
 \delta_{nk},\label{eq:A_B_rel1}\\
&\sum_{m=k}^n  \lB\{n,m\} \left(m+\frac 1\alpha\right) \lA\left[m+\frac 2 \alpha, k+\frac 2 \alpha\right]
=
\frac{(2\alpha F(0))^{n-k}}{(n-k)!},\label{eq:A_B_rel2}
\end{align}
where $\delta_{nk}$ denotes Kronecker's delta.
\end{proposition}

\subsection{First formula for expected internal angles}
We are now ready to express the quantities $\bJ_{n,k}(\gamma)$ through the quantities $\lA [\nu,\kappa]$ as follows.
\begin{proposition}\label{prop:bJ_formula_A_notilde}
For all integer $n\geq 3$, $k\in \{1,\ldots,n\}$ and all positive $\alpha \geq  n-3$, we have
\begin{equation}\label{eq:prop:bJ_formula_A_notilde}
\bJ_{n,k}\left(\frac{\alpha - n + 1}{2}\right)
=
\alpha^{k-n} \frac{n!}{k!}
\left(\frac {\Gamma(\frac {\alpha+2} 2)} {\sqrt \pi\, \Gamma(\frac{\alpha+1} 2)}\right)^{n-k}
\cdot
\frac{\Gamma({\frac{\alpha  k + 2} 2})}{\Gamma(\frac {\alpha k +3}2)}
\frac{\Gamma({\frac{\alpha  n + 3} 2})}{\Gamma(\frac {\alpha n +2}2)}
\cdot
\frac{\lA\left[n+\frac 2\alpha, k+\frac 2\alpha\right]}{\lA\left[k+\frac 2\alpha, k+\frac 2\alpha\right]}.
\end{equation}
\end{proposition}
\begin{proof}
Take some integer $N\geq 3$. It suffices to show that~\eqref{eq:prop:bJ_formula_A_notilde}  holds for all $n\in \{1,\ldots,N\}$ and $k\in \{1,\ldots,n\}$ provided that $\alpha \geq N-3$ and $\alpha\neq 0$.
Consider the following finite system of linear equations in the unknowns $\xi_{n,k}$, where $n\in\{1,\ldots,N\}$, $k\in \{1,\ldots,n\}$:
\begin{equation}\label{eq:system_alpha_arbitrary_notilde_rep}
\begin{cases}
\sum_{m=k}^n (-1)^{m-k} \bI_{n,m}(\alpha) \xi_{m,k} =  0, &\text{ for all } n\in \{1,\ldots,N\}, \; k\in \{1,\ldots,n-1\},\\
\xi_{n,n} = 1, \text{ for all } n\in \{1,\ldots,N\}.
\end{cases}
\end{equation}
We know from Proposition~\ref{prop:bJ_unique} that the \textit{unique} solution to this system is $\xi_{n,k} = \bJ_{n,k}(\frac{\alpha - n + 1}{2})$. To prove the proposition it therefore suffices to check that the right-hand side of~\eqref{eq:prop:bJ_formula_A_notilde} also defines a solution to~\eqref{eq:system_alpha_arbitrary_notilde_rep}. Recall from~\eqref{eq:I_n_m_b_n_m} that
\begin{align*}
\bI_{n,m}(\alpha)
=\alpha^{-n} n! c_{\frac{\alpha-1}{2}}^{n} \cdot \frac{c_{\frac{\alpha m-1}{2}} \alpha^{m}}{c_{\frac{\alpha-1}{2}}^m m!} \cdot \lB\{n,m\}.
\end{align*}
We now claim that for every function $c(k)$, the following defines a solution to the first line of~\eqref{eq:system_alpha_arbitrary_notilde_rep}:
\begin{equation}\label{eq:def_xi_m_k_notilde}
\xi_{m,k} := c(k)
\cdot \frac{c_{\frac{\alpha-1}{2}}^{m} m!}{c_{\frac{\alpha m-1}{2}} \alpha^{m}} \cdot \left(m+\frac 1\alpha\right) \lA\left[ m+\frac 2\alpha,k+\frac 2\alpha \right].
\end{equation}
Indeed, with this $\xi_{m,k}$ and for all $k\in \{1,\ldots,n-1\}$ we have
\begin{multline*}
\sum_{m=k}^n (-1)^{m-k} \bI_{n,m}(\alpha) \xi_{m,k}
=
\alpha^{-n} n!  c_{\frac{\alpha-1}{2}}^{n}\cdot
c(k)
\cdot
\sum_{m=k}^n (-1)^{m-k} \lB\{n,m\}   \left(m+\frac 1\alpha\right) \lA\left[ m+\frac 2\alpha,k+\frac 2\alpha \right],
\end{multline*}
which vanishes by Proposition~\ref{prop:relations_A_B}. It remains to determine the normalizing constant $c(k)$ such that $\xi_{k,k} = 1$ for all $k\in \{1,\ldots,N\}$. Taking $m=k$ in~\eqref{eq:def_xi_m_k_notilde} we see that $\xi_{k,k} = 1$ is satisfied iff
$$
c(k) = \frac{c_{\frac{\alpha k-1}{2}} \alpha^{k}}{c_{\frac{\alpha-1}{2}}^{k} k!} \cdot \frac{1} {\left(k+\frac 1\alpha\right) \lA\left[ k+\frac 2\alpha,k+\frac 2\alpha \right]}.
$$
Inserting this into~\eqref{eq:def_xi_m_k_notilde}, we obtain a solution to~\eqref{eq:system_alpha_arbitrary_notilde_rep}. Since the solution is unique, and since the numbers $\bJ_{n,k}(\frac{\alpha - n+1}{2})$ also define a solution, we arrive at the formula
\begin{equation}\label{eq:xi_n_k}
\bJ_{n,k}\left(\frac{\alpha - n+1}{2}\right)
=
\alpha^{k-n} \frac{n!}{k!}\cdot  c_{\frac{\alpha-1}{2}}^{n-k} \cdot \frac{c_{\frac{\alpha k-1}{2}}}{c_{\frac{\alpha n-1}{2}}}  \cdot \frac{\left(n+\frac 1\alpha\right) \lA\left[ n+\frac 2\alpha,k+\frac 2\alpha \right]}{\left(k+\frac 1\alpha\right) \lA\left[ k+\frac 2\alpha,k+\frac 2\alpha \right]}.
\end{equation}
Recall that $c_{\beta}$ is given by~\eqref{eq:c_beta}, namely
\begin{equation}\label{eq:c:1_beta_wspom}
c_{\beta}= \frac{ \Gamma\left(\beta + \frac{3}{2}\right) }{\sqrt\pi \,  \Gamma\left(\beta+1\right)}.
\end{equation}
After some simplifications we arrive at~\eqref{eq:prop:bJ_formula_A_notilde}.
\end{proof}
\subsection{Proof of Theorem~\ref{theo:bJ_formula_integral}}
We are now in position to prove Theorem~\ref{theo:bJ_formula_integral} or, more precisely, its equivalent form~\eqref{eq:bJ_nk_equiv2} which we restate as follows.
\begin{proposition}
For all integer $n\geq 3$, $k\in \{1,\ldots,n\}$ and all $\alpha \geq n-3$, we have
\begin{equation}\label{eq:J_nk_integral_recall}
\bJ_{n,k}\left(\frac{\alpha - n + 1}{2}\right)
=
\binom nk \int_{-\infty}^{+\infty} c_{\frac{\alpha n}2} (\cosh u)^{-\alpha n - 2}
\left(\frac 12  + \ii \int_0^u  c_{\frac{\alpha-1}{2}} (\cosh v)^{\alpha}\dd v \right)^{n-k} \dd u.
\end{equation}
\end{proposition}
\begin{proof}
We shall assume that $\alpha\neq 0$ since the case $\alpha=0$ (which occurs only if $n=3$) follows afterwards by letting $\alpha \downarrow 0$ and using continuity. Our starting point is~\eqref{eq:xi_n_k}.
In this formula,
$$
\left(k+\frac 1\alpha\right)\lA\left[ k+\frac 2\alpha,k+\frac 2\alpha \right]
=
\frac {\alpha k +1}{2\pi} \int_{-\infty}^{+\infty} (\cosh x)^{-\alpha k - 2} \dd x
=
\frac{\Gamma\left(\frac{\alpha k + 2}{2}\right)}{\sqrt \pi\, \Gamma\left(\frac{\alpha k + 1}{2}\right)}
=
c_{\frac{\alpha k -1}{2}};
$$
see~\eqref{eq:def_A_0_rep} and~\eqref{eq:int_cos_cosh}.
It follows from~\eqref{eq:xi_n_k} that
\begin{align}
\bJ_{n,k}\left(\frac{\alpha - n+1}{2}\right)
&=
\alpha^{k-n} \frac{n!}{k!}\cdot  c_{\frac{\alpha-1}{2}}^{n-k} \cdot \frac{1}{c_{\frac{\alpha n-1}{2}}}  \cdot \left(n+\frac 1\alpha\right) \lA\left[ n+\frac 2\alpha,k+\frac 2\alpha \right]\notag\\
&=
\alpha^{k-n} \frac{n!}{k!}
\cdot
c_{\frac{\alpha-1}{2}}^{n-k}
\cdot
\frac{\sqrt \pi\, \Gamma({\frac{\alpha  n + 1} 2})}{\Gamma(\frac {\alpha n +2}2)}
\cdot
\left(n+\frac 1\alpha\right)  \, \lA\left[ n+\frac 2\alpha,k+\frac 2\alpha \right]\notag\\
&=
\alpha^{k-n} \frac{n!}{k!}
\cdot
c_{\frac{\alpha-1}{2}}^{n-k}
\cdot
c_{\frac{\alpha n}{2}}
\cdot
\frac {2\pi}\alpha \, \lA\left[ n+\frac 2\alpha,k+\frac 2\alpha \right]
. \label{eq:form_J_n_k_wspom}
\end{align}
Recall from~\eqref{eq:def_A_0_rep} that
\begin{equation*}
\frac {2\pi}\alpha \, \lA\left[n + \frac 2 \alpha,k + \frac 2 \alpha\right]
=
\frac{\alpha^{n - k}}{(n - k)!} \cdot  \int_{-\infty}^{+\infty} \frac{(F(\ii u))^{n-k}}{(\cosh u)^{\alpha n + 2}} \dd u.
\end{equation*}
Hence,
\begin{equation*}
\bJ_{n,k}\left(\frac{\alpha - n+1}{2}\right)
=
\binom nk
c_{\frac{\alpha-1}{2}}^{n-k}
c_{\frac{\alpha n}{2}}
\cdot
\int_{-\infty}^{+\infty} \frac{(F(\ii u))^{n-k}}{(\cosh u)^{\alpha n + 2}} \dd u.
\end{equation*}
It remains to observe that by~\eqref{eq:def_F},
$$
F(\ii u)
= \int_{-\pi/2}^{0} (\cos y)^{\alpha} \dd y + \int_{0}^{\ii u} (\cos y)^{\alpha} \dd y
= \frac 1 {2 c_{\frac{\alpha-1}{2}}} + \ii \int_{0}^{u} (\cosh x)^{\alpha} \dd x.
$$
The proof is complete.
\end{proof}

\subsection{Evaluating the integral: Proof of Theorem~\ref{theo:bJ_formula_residue}}
In view of~\eqref{eq:form_J_n_k_wspom}, in order to prove Theorem~\ref{theo:bJ_formula_residue} we need to compute $\lA[\nu,\kappa]$.
\begin{proposition}\label{prop:residue_evaluate}
Let $\alpha\in\N$ and let $\nu,\kappa\in \alpha^{-1}\N$ be such that $\nu-\kappa\in \N_0$.
If $\alpha \kappa + \nu - \kappa$ is odd,
then
\begin{equation}\label{eq:int_residue}
\lA[\nu,\kappa]
=
\frac{\alpha^{\nu-\kappa+1}}{2\cdot  (\nu-\kappa)!}
\Res\limits_{x=0} \left[\frac{\left(\int_{0}^x (\sin y)^{\alpha} \dd y\right)^{\nu-\kappa}}{(\sin x)^{\alpha \nu}}\right].
\end{equation}
\end{proposition}
\begin{remark}
Recalling~\eqref{eq:form_J_n_k_wspom} and using the proposition with $\nu:=n+\frac 2\alpha$ and $\kappa:= k+\frac 2\alpha$ we obtain
$$
\bJ_{n,k}\left(\frac{\alpha - n + 1}{2}\right)
=
\binom nk
c_{\frac{\alpha-1}{2}}^{n-k} c_{\frac{\alpha n}2} \cdot
\pi  \Res\limits_{x=0} \left[\frac{\left(\int_{0}^x (\sin y)^{\alpha} \dd y\right)^{n-k}}{(\sin x)^{\alpha n + 2}}\right]
$$
provided (i) $\alpha$  is even and $n-k = \nu-\kappa$ is odd (because then $\alpha \kappa + \nu - \kappa = \alpha k + 2 + n-k$ is odd), or (ii) both $\alpha$ and $n$ are odd (because then $\alpha \kappa + \nu - \kappa = \alpha k - k + 2 + n$ is odd). This proves Theorem~\ref{theo:bJ_formula_residue}.
\end{remark}
\begin{proof}[Proof of Proposition~\ref{prop:residue_evaluate}.]
In view of the definition of $\lA[\nu,\kappa]$ given in~\eqref{eq:def_A_0_rep}, we need to prove that
\begin{equation}\label{eq:int_residue_proof}
\int_{-\infty}^{+\infty} \frac{(F(\ii u))^{\nu-\kappa}}{(\cosh u)^{\alpha \nu}} \dd u
=
\pi \Res\limits_{x=0} \left[\frac{\left(\int_{0}^x (\sin y)^{\alpha} \dd y\right)^{\nu-\kappa}}{(\sin x)^{\alpha \nu}}\right].
\end{equation}
Let $R>0$ and consider the closed rectangular contour $\gamma(R)$ passing through the points $-\pi+R\ii$, $-\pi -R\ii$, $- R\ii$, $+R\ii$. By the Cauchy residue formula, we have
$$
I(R) := \frac 1 {2\pi \ii} \oint_{\gamma(R)} \frac{(F(y))^{\nu-\kappa}}{(\cos y)^{\alpha \nu}} \dd y
=
\Res\limits_{y=-\pi/2} \left[\frac{(F(y))^{\nu-\kappa}}{(\cos y)^{\alpha \nu}}\right]
=
\Res\limits_{x=0} \left[\frac{\left(\int_{0}^x (\sin z)^{\alpha} \dd z\right)^{\nu-\kappa}}{(\sin x)^{\alpha \nu}}\right].
$$
On the other hand, since the value of the integral does not depend on $R$, we can let $R\to+\infty$. The contribution of the horizontal parts of the contour vanishes in the limit because in the same way as in~\eqref{eq:F_asympt} one can check that
$$
\lim_{v\to \pm \infty} \frac{(F(u+\ii v))^{\nu-\kappa}}{(\cos (u+\ii v))^{\alpha \nu}} = 0
$$
locally uniformly in $u\in\R$.  We are left with the contributions of the vertical lines:
$$
I := I(R) =
\frac 1 {2\pi \ii} \int_{-\ii \infty}^{+\ii \infty} \frac{(F(y))^{\nu-\kappa}}{(\cos y)^{\alpha \nu}} \dd y
+
\frac 1 {2\pi \ii} \int_{-\pi + \ii \infty}^{-\pi -\ii \infty} \frac{(F(y))^{\nu-\kappa}}{(\cos y)^{\alpha \nu} } \dd y.
$$
To complete the proof of the proposition, we need to argue that under (i) or (ii),
\begin{equation}\label{eq:I_twice_integral}
I= \frac 1 {\pi} \int_{-\infty}^{+\infty}  \frac{(F(\ii u))^{\nu-\kappa}}{(\cosh u)^{\alpha \nu}} \dd u.
\end{equation}
First of all, even without assuming (i) or (ii), the change of variables $y=\ii u$ yields
$$
\frac 1 {2\pi \ii} \int_{-\ii \infty}^{+\ii \infty} \frac{(F(y))^{\nu-\kappa}}{(\cos y)^{\alpha \nu }} \dd y
=
\frac 1 {2\pi} \int_{-\infty}^{+\infty} \frac{(F(\ii u))^{\nu-\kappa}}{(\cosh u)^{\alpha \nu}} \dd u.
$$
Hence, to prove~\eqref{eq:I_twice_integral} it remains to show that
\begin{equation}\label{eq:I_remains_to_show}
\frac 1 {2\pi \ii} \int_{-\pi + \ii \infty}^{-\pi -\ii \infty} \frac{(F(y))^{\nu-\kappa}}{(\cos y)^{\alpha \nu} } \dd y
=
\frac 1 {2\pi} \int_{-\infty}^{+\infty} \frac{(F(\ii u))^{\nu-\kappa}}{(\cosh u)^{\alpha \nu}} \dd u.
\end{equation}
Observe that by the change of variables $w= - x -\pi$ we have
\begin{align*}
F(-\pi - \ii u)
&=
\int_{-\pi/2}^{-\pi - \ii u} (\cos x)^{\alpha} \dd x
=
- \int_{-\pi/2}^{\ii u} (\cos (-w-\pi))^{\alpha} \dd w\\
&=
(-1)^{\alpha + 1} \int_{-\pi/2}^{\ii u} (\cos w)^{\alpha} \dd w
=
(-1)^{\alpha + 1} F(\ii u),
\end{align*}
where we used that $(\cos (-w-\pi))^\alpha = (-1)^{\alpha}(\cos w)^\alpha$.
By the substitution $y = -\pi -\ii u$ it follows that
\begin{align*}
\frac 1 {2\pi \ii} \int_{-\pi + \ii \infty}^{-\pi -\ii \infty} \frac{(F(y))^{\nu-\kappa}}{(\cos y)^{\alpha \nu} } \dd y
&=
-\frac 1 {2\pi} \int_{-\infty}^{+\infty} \frac{(F(-\pi -\ii u))^{\nu-\kappa}}{(\cos (-\pi -\ii u))^{\alpha \nu} } \dd u\\
&=
(-1)^{1+(\alpha+1)(\nu-\kappa) + \alpha \nu} \cdot \frac 1 {2\pi} \int_{-\infty}^{+\infty} \frac{(F(\ii u))^{\nu-\kappa}}{(\cosh u)^{\alpha \nu} } \dd u\\
&=
\frac 1 {2\pi} \int_{-\infty}^{+\infty} \frac{(F(\ii u))^{\nu-\kappa}}{(\cosh u)^{\alpha \nu} } \dd u
\end{align*}
since $\alpha \kappa + \nu - \kappa$ is odd.
This proves~\eqref{eq:I_remains_to_show} and completes the proof of~\eqref{eq:I_twice_integral}.
\end{proof}
\begin{remark}
In the case when $\alpha \kappa + \nu - \kappa$ is even, essentially the same argument shows that $I(R)=0$ and hence
$$
\Res\limits_{x=0} \left[\frac{\left(\int_{0}^x (\sin y)^{\alpha} \dd y\right)^{\nu-\kappa}}{(\sin x)^{\alpha \nu}}\right] = 0.
$$
\end{remark}

\subsection{Evaluating the residue: Proof of Proposition~\ref{prop:J_n_1_half_integer}}
We are going to apply Theorem~\ref{theo:bJ_formula_residue} with $\alpha = n + 2m - 2$ and $k=1$. Note that for every integer  $n\geq 3$ exactly one of the parity assumptions (i) or (ii) of this theorem is satisfied. By Theorem~\ref{theo:bJ_formula_residue}, our task reduces to showing that for every fixed $m\in\N_0$,
$$
R_m(n) := (n+2m-1)^{n-1} [x^{-1}] g_{n,m}(x)
$$
is a rational function of $n$ with rational coefficients, where
$$
g_{n,m}(x) = \frac{\left(\int_{0}^x (\sin y)^{n+2m-2} \dd y\right)^{n-1}}{(\sin x)^{(n +2m - 2) n +2}}.
$$
Since $\sin x \sim  x$ as $x\to 0$, it is easy to check that the first term of the Laurent expansion of $g_{n,m}(x)$ is  as follows:
$$
g_{n,m}(x) \sim x^{(n+2m-1)(n-1) - (n +2m - 2) n  - 2} = x^{-2m-1}, \qquad x\to 0.
$$
Let us stress that the power on the right-hand side does not depend on $n$, which is crucial for the following argument and explains why Proposition~\ref{prop:J_n_1_half_integer} cannot be extended to the quantities $\bJ_{n,k}(m-\frac 12)$ with $k\geq 2$ (for which a similar computation yields a power depending on $n$).

In the rest of the proof, we shall analyze the Laurent series of $g_{n,m}(x)$ at $x=0$. The Taylor series of $\sin y$, raised to the power $n+2m-2$, can be written as
$$
(\sin y)^{n + 2m-2}
=
y^{n + 2m - 2}\left(1 - \frac{y^2}{3!} + \frac{y^4}{5!} - \ldots\right)^{n+2m-2}
=
y^{n + 2m - 2} \sum_{k=0}^{\infty} P_{k}(n) y^{2k},
$$
where $P_0(n)=1, P_{1}(n) = -\frac 16 (n+2m-2),\ldots$ are polynomials in $n$ with rational coefficients. For example, the term $y^4$ can appear on the right-hand side either as a product of the term $\frac 1{5!} y^4$ and many $1$'s (which occurs $n+2m-2$ times), or as a product of two terms $-\frac 1{3!} y^2$ and many $1$'s (which occurs $\frac 12 (n+2m-2)(n+2m-3)$ times). Similar argument applies to any $y^{2k}$ and shows that the $P_k(n)$'s are indeed polynomials in $n$.  The dependence on $m$ is suppressed in our notation and the coefficients of all polynomials and rational functions are rational by default. Integrating, we obtain
$$
\int_{0}^x (\sin y)^{n+2m-2} \dd y
=
\sum_{k=0}^{\infty} P_{k}(n) \frac{x^{n + 2m - 1 + 2k}}{n+2m-1+2k}
=
\frac {x^{n + 2m - 1}} {n+2m-1} \sum_{k=0}^{\infty} S_{k}(n) x^{2k}
,
$$
where $S_k(n)$ are rational functions of $n$  with $S_0(n) = P_0(n) = 1$. Raising this Taylor series to the power $n-1$, we can write
\begin{equation}\label{eq:power_series1}
\left(\int_{0}^x (\sin y)^{n+2m-2} \dd y\right)^{n-1}
=
\frac {x^{(n + 2m - 1)(n-1)}} {(n+2m-1)^{n-1}} \sum_{k=0}^\infty  T_{k}(n) x^{2k},
\end{equation}
where $T_k(n)$ are rational functions of $n$ with $T_0(n) = 1$. Similar argumentation shows that
\begin{equation}\label{eq:power_series2}
(\sin x)^{(n +2m - 2) n +2} = x^{(n +2m - 2) n +2} \sum_{k=0}^\infty U_k(n) x^{2k},
\end{equation}
where $U_k(n)$ are polynomials in $n$ with $U_0(n) = 1$. Dividing the power series in~\eqref{eq:power_series1} and~\eqref{eq:power_series2}, we obtain the following Laurent expansion:
$$
g_{n,m}(x) =   \frac {x^{-2m-1}} {(n+2m-1)^{n-1}} \sum_{k=0}^\infty V_{k,m}(n) x^{2k},
$$
where $V_{k,m}(n)$ are rational functions of $n$ with $V_{0,m}(n) = 1$. The residue of $f$ at $0$ is determined by the term with $k=m$ in the above series. Hence,
$$
R_m(n) := (n+2m-1)^{n-1} [x^{-1}] g_{n,m}(x) = V_{m,m}(n)
$$
is a rational function of $n$. The above argument is in fact an algorithm for computing $R_m(n)$. The expressions for $R_0,R_1,R_2$ were obtained by running this algorithm.
\hfill $\Box$

\begin{remark}
The above proof cannot be adapted to the quantities of the form $\bJ_{n,1}(m)$ with $m\in\{-1,0,\ldots\}$. Indeed, we cannot use Theorem~\ref{theo:bJ_formula_residue} with $\alpha = n + 2m - 1$ and $k=1$ since neither of the parity assumptions (i) or (ii) is satisfied.
\end{remark}

\section{Proofs in the beta' case} \label{sec:proof_beta_tilde}

\subsection{Recurrence relations for external quantities}
Since the proofs in the beta' case are very similar to the proofs in the beta case, we only sketch them. Our aim is to prove Theorem~\ref{theo:bJ_tilde_formula_integral} and then to deduce its corollaries. Fix some  $\alpha >0$ once and for all.  Define
\begin{equation}\label{eq:def_F_tilde}
\tilde F(x) = \int_{-\pi/2}^x (\cos y)^{\alpha-1} \dd y, \qquad x\in \R.
\end{equation}
The same formula defines a univalued analytic function on the complex plane if we agree to make two cuts at $(-\infty, -\frac \pi2]$ and  $[+\frac \pi2,\infty)$.   For $\nu,\kappa\in \C$ such that $\Re \kappa > 0$ and $\nu-\kappa\in \N_0$ we define
\begin{equation}\label{eq:def_B_tilde}
\mB\{\nu, \kappa\}
:=
\frac{\alpha^{\nu-\kappa}}{(\nu-\kappa)!}\int_{-\pi/2}^{+\pi/2} (\cos x)^{\alpha \kappa-1} (\tilde F(x))^{\nu-\kappa} \dd x.
\end{equation}
By convention, we also put $\mB \{\nu,\kappa\} = 0$ if $\nu-\kappa \in \{-1,-2,\ldots\}$.  The recurrence relations satisfied by $\mB \{\nu,\kappa\}$ are as follows.
\begin{proposition}\label{prop:B_relation_tilde}
For $\kappa, \nu \in \C$ such that $\Re \kappa > 0$ and $\nu-\kappa\in \N_0$, we have
\begin{equation}\label{eq:B_relation_tilde}
\mB\{\nu,\kappa+2\} = (\kappa+1) \kappa \, \mB\{\nu,\kappa\}
-
\left(\kappa + \frac 1\alpha\right) (\kappa+1)\, \mB \left\{\nu +\frac 2\alpha, \kappa + \frac 2 \alpha\right\}.
\end{equation}
\end{proposition}
\begin{proof}
The idea is the same as in the proof of Proposition~\ref{prop:B_relation}. Let first $\nu-\kappa \in \{2,3,\ldots\}$.
Writing $\dd (\tilde F(x))^{\nu-\kappa-1} = (\nu-\kappa-1) (\tilde F(x))^{\nu-\kappa-2} (\cos x)^{\alpha-1} \dd x$ and integrating by parts, we obtain
\begin{align*}
\mB\{\nu, \kappa+2\}
&=
\frac {\alpha^{\nu-\kappa-2}} { (\nu-\kappa-1)!} \int_{-\pi/2}^{+\pi/2} (\cos x)^{\alpha \kappa + \alpha} \dd (\tilde F(x))^{\nu-\kappa-1}\\
&=
\frac {(\kappa+1)\alpha^{\nu-\kappa-1}} {(\nu-\kappa-1)!} \int_{-\pi/2}^{+\pi/2} (\tilde F(x))^{\nu-\kappa-1} (\cos x)^{\alpha \kappa + \alpha - 1} (\sin x) \dd x
\end{align*}
Writing $\dd (\tilde F(x))^{\nu-\kappa} = (\nu-\kappa) (\tilde F(x))^{\nu-\kappa-1} (\cos x)^{\alpha-1} \dd x$ and again integrating  by parts, we obtain
\begin{align*}
\mB \{\nu, \kappa+2\}
&=
\frac {(\kappa + 1) \alpha^{\nu-\kappa-1}} {(\nu-\kappa)!}
\int_{-\pi/2}^{+\pi/2} (\cos x)^{\alpha \kappa} (\sin x) \dd (\tilde F(x))^{\nu-\kappa}\\
&=
- \frac {(\kappa + 1) \alpha^{\nu-\kappa-1}} {(\nu-\kappa)!}
\int_{-\pi/2}^{+\pi/2} (\tilde F(x))^{\nu-\kappa} \left((\cos x)^{\alpha \kappa} \sin x\right)' \dd x\\
&=
 \frac {(\kappa + 1) \alpha^{\nu-\kappa-1}} {(\nu-\kappa)!}
\int_{-\pi/2}^{+\pi/2} (\tilde F(x))^{\nu-\kappa} \left(\alpha \kappa\, (\cos x)^{\alpha \kappa  - 1} \sin^2 x -(\cos x)^{\alpha \kappa+1} \right) \dd x
\end{align*}
Using the identity $\sin^2 x = 1-\cos^2 x$  we can write
\begin{align*}
\mB \{\nu, \kappa+2\}
&=
\frac {(\kappa + 1) \alpha^{\nu-\kappa-1}} {(\nu-\kappa)!}
\int_{-\pi/2}^{+\pi/2} (\tilde F(x))^{\nu-\kappa} \alpha \kappa \,(\cos x)^{\alpha \kappa - 1} \dd x\\
&\phantom{=}
-
\frac {(\kappa + 1) \alpha^{\nu-\kappa-1}} {(\nu-\kappa)!}
\int_{-\pi/2}^{+\pi/2} (\tilde F(x))^{\nu-\kappa} (\alpha\kappa+1)(\cos x)^{\alpha \kappa +1}\dd x.
\end{align*}
Recalling the definition of $\mB\{\nu,\kappa\}$, we arrive at
$$
\mB\{\nu,\kappa+2\} = (\kappa+1)\kappa \, \mB\{\nu,\kappa\}
-
\left(\kappa + \frac 1\alpha\right) (\kappa+1)\, \mB \left\{\nu +\frac 2\alpha, \kappa + \frac 2 \alpha\right\},
$$
which completes the proof if $\nu-\kappa\in \{2,3,\ldots\}$. The remaining cases $\nu=\kappa$ and $\nu=\kappa+1$ follow from the identities
\begin{align}
\mB\{\kappa, \kappa\}
&=
\int_{-\pi/2}^{+\pi/2} (\cos x)^{\alpha \kappa-1}\dd x
=
\frac{\sqrt \pi\, \Gamma\left(\frac{\alpha \kappa}{2}\right)}{\Gamma\left(\frac{\alpha \kappa +1}{2}\right)},\\
\mB\{\kappa +1, \kappa\}
&=
\alpha \int_{-\pi/2}^{+\pi/2} (\cos x)^{\alpha \kappa-1} \tilde F(x) \dd x
=
\alpha \int_{-\pi/2}^{+\pi/2} (\cos x)^{\alpha \kappa-1} \tilde F(0) \dd x\\
&=
\frac{\alpha}{2}\cdot \frac{\sqrt \pi\, \Gamma\left(\frac{\alpha}{2}\right)}{\Gamma\left(\frac{\alpha +1}{2}\right)} \cdot \frac{\sqrt \pi\, \Gamma\left(\frac{\alpha \kappa}{2}\right)}{\Gamma\left(\frac{\alpha \kappa +1}{2}\right)},
\end{align}
which are consequences of~\eqref{eq:def_B_tilde}, \eqref{eq:int_cos_cosh} and the fact that the function $(\cos x)^{\alpha \kappa-1} (\tilde F(x)-\tilde F(0))$ is odd.
\end{proof}

\subsection{Recurrence relations for internal quantities}
The dual quantities (which, as we shall see, are related to expected internal angles) are defined as follows. For $\nu,\kappa\in \C$ such that $\nu-\kappa\in \N_0$ we put
\begin{align}
\mA[\nu,\kappa]
&:=
\frac{\alpha^{\nu - \kappa + 1}}{(\nu - \kappa)!} \cdot \frac 1 {2\pi \ii} \int_{- \ii \infty}^{+\ii \infty} (\cos x)^{-\alpha \nu-1} (\tilde F(x))^{\nu-\kappa} \dd x \label{eq:def_A_0_tilde1}\\
&=
\frac{\alpha^{\nu - \kappa + 1}}{(\nu - \kappa)!} \cdot \frac 1 {2\pi} \int_{-\infty}^{+\infty} (\cosh u)^{-\alpha \nu-1} (\tilde F(\ii u))^{\nu-\kappa} \dd u. \label{eq:def_A_0_tilde2}
\end{align}
To ensure that the integral~\eqref{eq:def_A_0_tilde2} is absolutely convergent we have to impose the following additional conditions on  $\nu$ and $\kappa$:
\begin{align}
&\alpha \Re \nu + 1 > 0 \qquad  &&\text{if } \alpha \in (0,1],\label{eq:cond_nu_kappa1}\\
&\alpha \Re \kappa + (\nu-\kappa) + 1 > 0  \qquad  &&\text{if } \alpha >1.\label{eq:cond_nu_kappa2}
\end{align}
To see that the condition in the first line suffices, observe that in the case when $0<\alpha<1$, the limit of $\tilde F(\ii u)$ as $u\to \pm \infty$ is finite, while in the case when $\alpha=1$, the function $\tilde F(\ii u) = \frac \pi 2 + \ii u$ grows subexponentially. To see that the condition in the second line suffices, observe that for $\alpha>1$ the L'Hospital rule yields
\begin{equation}\label{eq:F_asympt_tilde}
\tilde F(\ii u) = \frac{\ii \cdot \sgn u}{(\alpha-1) 2^{\alpha-1}}  \eee^{(\alpha-1) |u|} (1+o(1)), \qquad \text{as } u\to \pm \infty.
\end{equation}
The quantities $\mA[\nu,\kappa]$ satisfy a recurrence relation which is ``dual'' to the relation satisfied by $\mB\{\nu,\kappa\}$.
\begin{proposition}\label{prop:A_0_relations_tilde}
For all $\nu,\kappa\in \C$ such that $\nu-\kappa\in \N_0$ and Conditions~\eqref{eq:cond_nu_kappa1}, \eqref{eq:cond_nu_kappa2} hold,  we have
\begin{align*}
\mA[\nu-2,\kappa]
=
(\nu-1)\nu \, \mA[\nu,\kappa] -\left(\nu-\frac 1 \alpha\right) (\nu-1) \,\mA \left[\nu - \frac 2\alpha, \kappa - \frac 2 \alpha\right].
\end{align*}
\end{proposition}
\begin{proof}
Analogous to the proof of Proposition~\ref{prop:A_0_relations}. The cases $\nu = \kappa$ and $\nu = \kappa+1$ have to be considered separately by using the formulae
\begin{align}
\mA[\kappa,\kappa]
&=
\frac \alpha {2\pi} \int_{-\infty}^{+\infty} (\cosh u)^{-\alpha \kappa-1} \dd u =
\frac \alpha 2 \cdot \frac{\Gamma\left(\frac{\alpha \kappa +1}{2}\right)}{\sqrt \pi\, \Gamma\left(\frac{\alpha \kappa+2}{2}\right)} = \frac 1 \kappa \, \tilde c_{\frac{\alpha \kappa + 1}{2}}, \label{eq:a_diag_tilde}\\
\mA[\kappa+1,\kappa]
&=
\frac{\alpha^{2}}{2\pi} \int_{-\infty}^{+\infty} (\cosh u)^{-\alpha \kappa - \alpha -1} \tilde F(\ii u) \dd u
=
\frac{\alpha^{2}}{2\pi} \tilde F(0) \int_{-\infty}^{+\infty} (\cosh u)^{-\alpha \kappa - \alpha -1}  \dd u
\notag\\
&=
\frac \alpha 2 \cdot \frac{\Gamma(\frac{\alpha+2}{2})}{\Gamma(\frac {\alpha+1} 2)} \cdot \frac{\Gamma(\frac{\alpha \kappa +\alpha + 1}{2})}{\Gamma(\frac {\alpha\kappa+\alpha+2} 2)}.
\end{align}
In the proof of the second formula we used that $\Re \tilde F(\ii u) = \tilde F(0) = \frac{\sqrt{\pi}\, \Gamma(\frac{\alpha+2}{2})}{\alpha \, \Gamma(\frac {\alpha+1} 2)}$, while the imaginary part of the integral cancels.
\end{proof}
\begin{remark}\label{rem:other_quantities_tilde}
Sometimes it is more convenient to work with the quantities
\begin{equation}\label{eq:def_mtA_mtB}
\mtB\{\nu,\kappa\} := \frac{\mB\{\nu,\kappa\}}{\Gamma(\kappa)},
\qquad
\mtA[\nu,\kappa] := \mA[\nu,\kappa] \cdot \Gamma(\nu+1).
\end{equation}
In the special case $\alpha=2$ these definitions are consistent with the ones introduced in Section~\ref{sec:simplest_case_proof}. This will be shown in Section~\ref{subsec:special_cases}.
With this notation, the recurrence relations of Propositions~\ref{prop:B_relation_tilde} and~\ref{prop:A_0_relations_tilde} take the form
\begin{align}
&\mtB\{\nu,\kappa\} - \mtB\{\nu,\kappa+2\}
=
\frac{\left(\kappa  +  \frac 1\alpha\right) \Gamma\left(\kappa  + \frac 2\alpha\right)} {\Gamma(\kappa+1)} \mtB \left\{\nu + \frac 2\alpha, \kappa + \frac 2 \alpha\right\},\label{eq:b_relation_tilde}\\
&
\mtA[\nu,\kappa] - \mtA[\nu-2,\kappa]
=
\frac{\left(\nu - \frac 1 \alpha\right)\Gamma(\nu)}{\Gamma\left(\nu - \frac 2 \alpha + 1\right)} \mtA \left[\nu  -  \frac 2\alpha, \kappa -  \frac 2 \alpha\right]. \label{eq:a_0_relation_tilde}
\end{align}
\end{remark}

\subsection{Evaluating the sum}
Fix some $r\in\N_0$ and consider the expression
\begin{equation}\label{eq:def_s_r_tilde}
\tilde s_r(\kappa)
:=
\sum_{\mu = \kappa}^{\kappa + r}
(-1)^{\mu - \kappa}
\mB\{\kappa + r, \mu\}
\left(\mu - \frac 1\alpha\right) \mA\left[\mu-\frac 2\alpha, \kappa-\frac 2\alpha\right].
\end{equation}
Using~\eqref{eq:cond_nu_kappa1} and~\eqref{eq:cond_nu_kappa2} it is easy to check that $\tilde s_r(\kappa)$ is well defined provided $\Re \kappa > \frac 1\alpha$, for all $\alpha>0$.
\begin{proposition}\label{prop:periodic_tilde}
Let $\alpha\in\N$. Then, for every $r\in\N_0$ and all $\kappa\in\C$ such that $\Re \kappa > \frac 1 \alpha$,
$$
\tilde s_{r}(\kappa) = \tilde s_r \left(\kappa + \frac 2 \alpha\right).
$$
\end{proposition}
\begin{proof}
Analogous to the proof of Proposition~\ref{prop:periodic}. Indeed, the definition of $\tilde s_r(\kappa)$ takes the form
\begin{equation}\label{eq:def_s_r_new_tilde}
\tilde s_r(\kappa)
=
\sum_{\mu = \kappa}^{\kappa + r}
(-1)^{\mu - \kappa}
\mtB\{\kappa + r, \mu\}
\frac{\left(\mu - \frac 1\alpha\right)\Gamma(\mu)}{\Gamma \left(\mu - \frac 2\alpha+1\right)}
 \mtA\left[\mu - \frac 2\alpha, \kappa - \frac 2\alpha\right]
\end{equation}
and the relations satisfied by $\mtA[\nu,\kappa]$ and $\mtB\{\nu,\kappa\}$, see~\eqref{eq:a_0_relation_tilde}, \eqref{eq:b_relation_tilde},  are the same as the relations satisfied by $\ii^{\nu-\kappa}\rA[\nu,\kappa]$ and $\ii^{\nu-\kappa}\rB\{\nu,\kappa\}$, see~\eqref{eq:a_0_relation}, \eqref{eq:b_relation}, provided we replace $\alpha$ by $-\alpha$ throughout. Since the proof of Proposition~\ref{prop:periodic} used only the latter relations, it applies with obvious changes.
\end{proof}
In order to evaluate $\tilde s_r(\kappa)$, we need the following two lemmas on the asymptotic behaviour of $\mA[\nu,\kappa]$ and $\mB\{\nu,\kappa\}$.
\begin{lemma}\label{lem:B_asymptotic_tilde}
Fix some real $\kappa_1,\kappa_2\in\C$  such that $\kappa_1-\kappa_2\in\N_0$. Then, as $T\to+\infty$, we have
$$
\mB \left\{\kappa_1 +\frac 2 \alpha T, \kappa_2 + \frac 2 \alpha T\right\}
\sim
\frac{(\alpha \tilde F(0))^{\kappa_1-\kappa_2}}{(\kappa_1-\kappa_2)!} \sqrt{\frac {\pi}{T}}.
$$
\end{lemma}
\begin{proof}
Almost identical to the proof of Lemma~\ref{lem:B_asymptotic}.
\end{proof}
\begin{lemma}\label{lem:A_asymptotic_tilde}
Fix some $\kappa_1,\kappa_2\in\C$ such that $\kappa_1-\kappa_2\in\N_0$.  Then, as $T\to+\infty$, we have
$$
\mA\left[\kappa_1 +\frac 2 \alpha T, \kappa_2 + \frac 2 \alpha T\right]
\sim
\frac \alpha 2 \cdot \frac{(\alpha \tilde F(0))^{\kappa_1-\kappa_2}}{(\kappa_1-\kappa_2)!}\cdot \frac {1}{\sqrt{\pi T}}.
$$
\end{lemma}
\begin{proof}
Almost identical to the proof of Lemma~\ref{lem:A_0_asymptotic}.
\end{proof}

\begin{proposition}\label{prop:s_r_vanishes_tilde}
Let $\alpha>0$.  Then, for all $\kappa\in\C$ such that $\Re \kappa > \frac 1 \alpha$, we have that
$$
\tilde s_r(\kappa) =
\begin{cases}
1, &\text{ if } r=0,\\
0, &\text{ if } r\in\N.
\end{cases}
$$
\end{proposition}
\begin{proof}
Since the asymptotic expressions given in Lemmas~\ref{lem:B_asymptotic_tilde} and~\ref{lem:A_asymptotic_tilde} coincide with those given in Lemmas~\ref{lem:B_asymptotic} and~\ref{lem:A_0_asymptotic}, the proof of Proposition~\ref{prop:s_r_vanishes} applies with minimal changes.
\end{proof}
\begin{remark}\label{rem:s_r_+_tilde}
It is also possible to show that the unsigned version of $\tilde s_r$ satisfies
\begin{align*}
\tilde s_r^{+} (\kappa)
&:=  \sum_{\mu = \kappa}^{\kappa + r}   \mB\{\kappa + r, \mu\}
\left(\mu - \frac 1\alpha\right) \mA\left[\mu - \frac 2\alpha, \kappa - \frac 2\alpha\right]
=
\frac{(2\alpha \tilde F(0))^r}{r!}.
\end{align*}
\end{remark}

The next result, which is a  special case of Proposition~\ref{prop:s_r_vanishes_tilde} and Remark~\ref{rem:s_r_+_tilde}, records relations similar to those satisfied by the Stirling numbers; see~\eqref{eq:stirling_inverse}.
\begin{proposition}\label{prop:relations_A_B_tilde}
Let $\alpha>0$.  For every $n\in\N$ and $k\in \{1,\ldots,n\}$  such that $\alpha k>1$ we have
\begin{align}
&\sum_{m=k}^n (-1)^{m-k} \mB\{n,m\} \left(m - \frac 1\alpha\right) \mA\left[m - \frac 2 \alpha, k - \frac 2 \alpha\right]
=
 \delta_{nk},\label{eq:A_B_rel1_tilde}\\
&\sum_{m=k}^n  \mB\{n,m\} \left(m - \frac 1\alpha\right) \mA\left[m - \frac 2 \alpha, k - \frac 2 \alpha\right]
=
\frac{(2\alpha \tilde F(0))^{n-k}}{(n-k)!},\label{eq:A_B_rel2_tilde}
\end{align}
where $\delta_{nk}$ denotes Kronecker's delta.
\end{proposition}

\subsection{Formula for expected internal angles}
We are now ready to express the expected internal angles $\tilde \bJ_{n,k}(\gamma)$ through the quantities $\mA[\nu,\kappa]$.
\begin{proposition}\label{prop:bJ_tilde_formula_A_tilde}
Let $\alpha>0$. For all $n\in\N$ and $k\in \{1,\ldots,n\}$ such that $\alpha k>1$ we have
\begin{equation}\label{eq:prop:bJ_tilde_formula_A_tilde}
\tilde \bJ_{n,k}\left(\frac{\alpha  +  n  -  1}{2}\right)
=
\alpha^{k-n} \frac{n!}{k!}
\left(\frac {\Gamma(\frac {\alpha+1} 2)} {\sqrt \pi\, \Gamma(\frac{\alpha} 2)}\right)^{n-k}
\cdot
\frac{\Gamma({\frac{\alpha  k -1} 2})}{\Gamma(\frac {\alpha k}2)}
\frac{\Gamma({\frac{\alpha  n} 2})}{\Gamma(\frac {\alpha n-1}2)}
\cdot
\frac{\mA\left[n - \frac 2\alpha, k - \frac 2\alpha\right]}{\mA\left[k - \frac 2\alpha, k - \frac 2\alpha\right]}
.
\end{equation}
\end{proposition}
\begin{proof}
Consider the following infinite system of linear equations in the unknowns $\tilde \xi_{n,k}$, where $n\in\N$, $k\in \{1,\ldots,n\}$:
\begin{equation}\label{eq:system_alpha_arbitrary_tilde}
\begin{cases}
\sum_{m=k}^n (-1)^{m-k} \tilde \bI_{n,m}(\alpha) \tilde \xi_{m,k} =  0, &\text{ for all } n\in\N, \; k\in \{1,\ldots,n-1\},\\
\tilde \xi_{n,n} = 1, \text{ for all } n\in \N.
\end{cases}
\end{equation}
We know from Proposition~\ref{prop:bJ_tilde_unique} that the \textit{unique} solution to this system is $\tilde \xi_{n,k} = \tilde \bJ_{n,k}(\frac{\alpha +n-1}{2})$. To prove~\eqref{eq:prop:bJ_tilde_formula_A_tilde} it therefore suffices to check that its right-hand side also defines a solution to~\eqref{eq:system_alpha_arbitrary_tilde}. Recall from~\eqref{eq:I_n_k_tilde} and~\eqref{eq:I_tilde_I_bold} that
\begin{align}
\tilde \bI_{n,m}(\alpha)
&=
\binom nm \tilde c_{\frac{\alpha m+1}{2}} \tilde c_{\frac{\alpha+1}{2}}^{n-m} \int_{-\pi/2}^{+\pi/2} (\cos x)^{\alpha m-1} (\tilde  F(x))^{n-m} \dd x\notag\\
&=
\alpha^{-n} n! \tilde c_{\frac{\alpha+1}{2}}^{n} \cdot \frac{\tilde c_{\frac{\alpha m+1}{2}} \alpha^{m}}{\tilde c_{\frac{\alpha+1}{2}}^{m} m!} \cdot \mB\{n,m\}.\label{eq:I_n_m_b_n_m_tilde}
\end{align}
We now claim that for every function $c(k)$, the following defines a solution to the first line of~\eqref{eq:system_alpha_arbitrary_tilde}:
\begin{equation}\label{eq:def_xi_m_k_tilde}
\tilde \xi_{m,k} :=
c(k)\cdot  \frac{\tilde c_{\frac{\alpha+1}{2}}^{m} m!}{\tilde c_{\frac{\alpha m+1}{2}} \alpha^{m}}
\cdot
\left(m - \frac 1\alpha\right) \mA\left[m - \frac 2 \alpha, k - \frac 2 \alpha\right],
\end{equation}
where we need the condition $\alpha k>1$ to ensure that the right-hand side is well defined.  With this $\tilde \xi_{m,k}$ and for all $k\in \{1,\ldots,n-1\}$ we have
$$
\sum_{m=k}^n (-1)^{m-k} \tilde \bI_{n,m}(\alpha) \tilde \xi_{m,k}
=
\alpha^{-n} n! \tilde c_{\frac{\alpha+1}{2}}^{n}
\cdot
c(k)
\cdot
\sum_{m=k}^n (-1)^{m-k} \mB\{n,m\}   \left(m - \frac 1\alpha\right) \mA\left[ m - \frac 2\alpha,k - \frac 2\alpha \right],
$$
which vanishes by Proposition~\ref{prop:relations_A_B_tilde}. It remains to determine the normalization $c(k)$ such that $\tilde \xi_{k,k} = 1$ for all $k\in\N$. Taking $m=k$ in~\eqref{eq:def_xi_m_k_tilde} we see that $\tilde \xi_{k,k} = 1$ holds true iff
$$
c(k) = \frac{\tilde c_{\frac{\alpha k + 1}{2}} \alpha^{k}}{\tilde c_{\frac{\alpha + 1}{2}}^{k} k!} \cdot \frac{1} {\left(k - \frac 1\alpha\right) \mA\left[ k - \frac 2\alpha,k - \frac 2\alpha \right]}.
$$
Inserting this into~\eqref{eq:def_xi_m_k_tilde}, we obtain a solution to~\eqref{eq:system_alpha_arbitrary_tilde}. Since the solution is unique, it must coincide with $\tilde \bJ_{m,k}(\frac{\alpha + m - 1}{2})$, and we arrive at
\begin{equation}\label{eq:xi_n_k_tilde}
\tilde \bJ_{n,k}\left(\frac{\alpha + n - 1}{2}\right)
=
\alpha^{k-n} \frac{n!}{k!}\cdot  \tilde c_{\frac{\alpha + 1}{2}}^{n-k} \cdot \frac{\tilde c_{\frac{\alpha k + 1}{2}}}{\tilde c_{\frac{\alpha n + 1}{2}}}  \cdot \frac{\left(n - \frac 1\alpha\right) \mA\left[ n - \frac 2\alpha,k - \frac 2\alpha \right]}{\left(k - \frac 1\alpha\right) \mA\left[ k - \frac 2\alpha,k - \frac 2\alpha \right]}.
\end{equation}
Recall that $\tilde c_{\beta}$ is given by~\eqref{eq:c_beta_tilde}, namely
\begin{equation}\label{eq:c:1_beta_wspom_tilde}
\tilde c_{\beta}= \frac{ \Gamma(\beta) }{\sqrt\pi \,  \Gamma\left(\beta - \frac 12\right)}.
\end{equation}
After some simplifications we arrive at~\eqref{eq:prop:bJ_tilde_formula_A_tilde}.
The proof is complete.
\end{proof}

\subsection{Proof of Theorem~\ref{theo:bJ_tilde_formula_integral}}
We are going to prove~\eqref{eq:bJ_nk_tilde_equiv2} which is an equivalent way of stating Theorem~\ref{theo:bJ_tilde_formula_integral}.
Let $\alpha>0$. We claim that for all  $n\in\N$ and  $k\in \{1,\ldots,n\}$ such that $\alpha n >1$, we have
\begin{equation}\label{eq:J_nk_integral_recall_tilde_proof}
\tilde \bJ_{n,k}\left(\frac{\alpha + n - 1}{2}\right)
=
\binom nk \int_{-\infty}^{+\infty} \tilde c_{\frac{\alpha n}2} (\cosh u)^{- (\alpha n - 1)} \left(\frac 12  + \ii \int_0^u \tilde c_{\frac{\alpha+1}{2}}(\cosh v)^{\alpha-1}\dd v \right)^{n-k} \dd u.
\end{equation}
\begin{proof}
Let first the stronger condition $\alpha k >1$ be satisfied. By~\eqref{eq:def_A_0_tilde2} and~\eqref{eq:int_cos_cosh}, the term appearing in the denominator of the last fraction in~\eqref{eq:xi_n_k_tilde} is
$$
\left(k-\frac 1\alpha\right)\mA\left[ k - \frac 2\alpha,k - \frac 2\alpha \right]
=
\frac {\alpha k - 1}{2\pi} \int_{-\infty}^{+\infty} (\cosh x)^{-\alpha k + 1} \dd x
=
\frac{\Gamma\left(\frac{\alpha k + 1}{2}\right)}{\sqrt \pi\, \Gamma\left(\frac{\alpha k}{2}\right)}
=
\tilde c_{\frac{\alpha k  + 1}{2}}.
$$
It follows from~\eqref{eq:xi_n_k_tilde} that
\begin{align}
\tilde \bJ_{n,k}\left(\frac{\alpha  + n - 1}{2}\right)
&=
\alpha^{k-n} \frac{n!}{k!}\cdot  \tilde c_{\frac{\alpha + 1}{2}}^{n-k} \cdot \frac{1}{\tilde c_{\frac{\alpha n + 1}{2}}}  \cdot \left(n - \frac 1\alpha\right) \mA\left[ n - \frac 2\alpha,k - \frac 2\alpha \right]\notag \\
&=
\alpha^{k-n} \frac{n!}{k!}
\cdot
\tilde c_{\frac{\alpha + 1}{2}}^{n-k}
\cdot
\frac{\sqrt \pi\, \Gamma({\frac{\alpha  n} 2})}{\Gamma(\frac {\alpha n + 1}2)}
\cdot
\left(n - \frac 1\alpha\right)  \, \mA\left[ n - \frac 2\alpha,k - \frac 2\alpha \right]\notag \\
&=
\alpha^{k-n} \frac{n!}{k!}
\cdot
\tilde c_{\frac{\alpha + 1}{2}}^{n-k}
\cdot
\tilde c_{\frac{\alpha n}{2}}
\cdot
\frac {2\pi}\alpha \, \mA\left[ n - \frac 2\alpha,k - \frac 2\alpha \right] \label{eq:form_J_n_k_wspom_tilde}
.
\end{align}
Recall from~\eqref{eq:def_A_0_tilde2} that
\begin{equation*}
\frac {2\pi}\alpha \, \mA\left[n - \frac 2 \alpha,k - \frac 2 \alpha\right]
=
\frac{\alpha^{n - k}}{(n - k)!} \cdot  \int_{-\infty}^{+\infty} \frac{(\tilde F(\ii u))^{n-k}}{(\cosh u)^{\alpha n - 1}} \dd u.
\end{equation*}
Hence,
\begin{equation*}
\tilde \bJ_{n,k}\left(\frac{\alpha  +  n - 1}{2}\right)
=
\binom nk
\cdot
\tilde c_{\frac{\alpha + 1}{2}}^{n-k}
\cdot
\tilde c_{\frac{\alpha n}{2}}
\cdot
\int_{-\infty}^{+\infty} \frac{(\tilde F(\ii u))^{n-k}}{(\cosh u)^{\alpha n - 1}} \dd u.
\end{equation*}
It remains to observe that by~\eqref{eq:def_F_tilde},
$$
\tilde F(\ii u)
= \int_{-\pi/2}^{0} (\cos y)^{\alpha-1} \dd y + \int_{0}^{\ii u} (\cos y)^{\alpha-1} \dd y
= \frac 1 {2 \tilde c_{\frac{\alpha+1}{2}}} + \ii \int_{0}^{u} (\cosh x)^{\alpha-1} \dd x.
$$
This completes the proof in the case when $\alpha k>1$. To treat the more general case $\alpha n>1$ we use analytic continuation. Fix some $n\in\N$ and $k\in \{1,\ldots,n\}$.  On the one hand, the integral on the right-hand side of~\eqref{eq:J_nk_integral_recall_tilde_proof} is well-defined and represents an analytic function of $\alpha$, considered as a complex variable,  provided that $\Re \alpha > \frac 1n$. On the other hand, for all $\alpha >0$ we can write the definition of expected internal angles as follows:
$$
\tilde \bJ_{n,k}\left(\frac{\alpha + n - 1}{2}\right)
=
\binom nk \int_{(\R^{n-1})^n} \beta ([x_1,\ldots,x_k], [x_1,\ldots,x_n]) \prod_{i=1}^n \left(\frac{\tilde{c}_{n-1,\frac{\alpha + n - 1}{2}}\dd x_i}{\left( 1+\left\| x_i \right\|^2 \right)^{\frac{\alpha + n - 1}{2}}} \right).
$$
The right-hand side is well-defined and represents an analytic function of the complex variable $\alpha$ in the half-plane $\Re \alpha >0$. This can be shown by essentially the same argument as in~\cite[lemma~4.3]{beta_polytopes}. Thus, both sides of~\eqref{eq:J_nk_integral_recall_tilde_proof} can be continued analytically at least to the half-plane $\Re \alpha > \frac 1n$. Since they coincide for real $\alpha> \frac 1k$, the uniqueness of analytic continuation implies that they coincide on their whole domain of definition, thus completing the proof.
\end{proof}

\subsection{Evaluating the integral:  Proof of Theorem~\ref{theo:bJ_tilde_formula_residue}}
Our aim is to compute $\mA[\nu,\kappa]$ by residue calculus whenever this is possible.
\begin{proposition}\label{prop:a_nu_kappa_residue_tilde}
Let $\alpha\in \N$ and let $\nu,\kappa \in \alpha^{-1}\N$ be such that $\nu-\kappa\in \N_0$ and $\alpha \kappa$ is even. Then,
\begin{equation*}
\mA[\nu,\kappa]
=
\frac{\alpha^{\nu-\kappa+1}}{2\cdot  (\nu-\kappa)!}  \Res\limits_{x=0} \left[\frac{\left(\int_{0}^x (\sin y)^{\alpha-1} \dd y\right)^{\nu-\kappa}}{(\sin x)^{\alpha \nu + 1}}\right].
\end{equation*}
\end{proposition}
\begin{remark}
Recalling~\eqref{eq:form_J_n_k_wspom_tilde} and using the proposition with $\nu:=n-\frac 2\alpha$ and $\kappa:= k-\frac 2\alpha$ we obtain
$$
\tilde \bJ_{n,k}\left(\frac{\alpha + n-1}{2}\right)
=
\binom nk
\tilde c_{\frac{\alpha+1}{2}}^{n-k} \tilde c_{\frac{\alpha n}2} \cdot
\pi  \Res\limits_{x=0} \left[\frac{\left(\int_{0}^x (\sin y)^{\alpha-1} \dd y\right)^{n-k}}{(\sin x)^{\alpha n - 1}}\right]
$$
provided $\alpha \kappa = \alpha k - 2$ is even. This proves Theorem~\ref{theo:bJ_tilde_formula_residue}.
\end{remark}
\begin{proof}[Proof of Proposition~\ref{prop:a_nu_kappa_residue_tilde}]
In view of the definition of $\mA[\nu,\kappa]$ given in~\eqref{eq:def_A_0_tilde2}, we need to prove that
\begin{equation}\label{eq:int_residue_tilde}
\int_{-\infty}^{+\infty} \frac{(\tilde F(\ii u))^{\nu-\kappa}}{(\cosh u)^{\alpha \nu + 1}} \dd u
=
\pi \Res\limits_{x=0} \left[\frac{\left(\int_{0}^x (\sin y)^{\alpha-1} \dd y\right)^{\nu-\kappa}}{(\sin x)^{\alpha \nu + 1}}\right].
\end{equation}
Observe that the integrand is \textit{univalued} and meromorphic because $\alpha\in\N$ and $\alpha \nu\in\N$. Let $R>0$ and consider the closed rectangular contour $\gamma(R)$ passing through the points $-\pi+R\ii$, $-\pi -R\ii$, $- R\ii$, $+R\ii$. By the Cauchy residue formula, we have
$$
\tilde I(R) := \frac 1 {2\pi \ii} \oint_{\gamma(R)} \frac{(\tilde F(y))^{\nu-\kappa}}{(\cos y)^{\alpha \nu + 1}} \dd y
=
\Res\limits_{y=-\pi/2} \left[\frac{(\tilde F(y))^{\nu-\kappa}}{(\cos y)^{\alpha \nu + 1}}\right]
=
\Res\limits_{x=0} \left[\frac{\left(\int_{0}^x (\sin z)^{\alpha-1} \dd z\right)^{\nu-\kappa}}{(\sin x)^{\alpha \nu + 1}}\right].
$$
On the other hand, since the value of the integral does not depend on $R$, we can let $R\to+\infty$. The contribution of the horizontal parts of the contour vanishes in the limit because
$$
\lim_{v\to \pm \infty} \frac{(\tilde F(u+\ii v))^{\nu-\kappa}}{(\cos (u+\ii v))^{\alpha \nu + 1}} = 0
$$
locally uniformly in $u\in\R$. This is shown essentially as in the discussion preceding~\eqref{eq:F_asympt_tilde}. For example, for $\alpha>1$ we need to verify that $(\alpha-1)(\nu-\kappa) < \alpha \nu + 1$, which is equivalent to $\alpha \kappa +\nu +1 >\kappa$. This inequality holds true since  $\alpha\in\N$. In the case $0< \alpha \leq 1$ one uses that $\alpha \nu + 1 >0$.   Thus, we are left with the contributions of the vertical lines:
$$
\tilde I := \tilde I(R) =
\frac 1 {2\pi \ii} \int_{-\ii \infty}^{+\ii \infty} \frac{(\tilde F(y))^{\nu-\kappa}}{(\cos y)^{\alpha \nu + 1}} \dd y
+
\frac 1 {2\pi \ii} \int_{-\pi + \ii \infty}^{-\pi -\ii \infty} \frac{(\tilde F(y))^{\nu-\kappa}}{(\cos y)^{\alpha \nu  + 1} } \dd y.
$$
To complete the proof of the proposition, we need to argue that
\begin{equation}\label{eq:I_twice_integral_tilde}
\tilde I= \frac 1 {\pi} \int_{-\infty}^{+\infty}  \frac{(\tilde F(\ii u))^{\nu-\kappa}}{(\cosh u)^{\alpha \nu + 1}} \dd u.
\end{equation}
The substitution $y=\ii u$ yields
$$
\frac 1 {2\pi \ii} \int_{-\ii \infty}^{+\ii \infty} \frac{(\tilde F(y))^{\nu-\kappa}}{(\cos y)^{\alpha \nu + 1}} \dd y
=
\frac 1 {2\pi} \int_{-\infty}^{+\infty} \frac{(\tilde F(\ii u))^{\nu-\kappa}}{(\cosh u)^{\alpha \nu  + 1 }} \dd u.
$$
To prove~\eqref{eq:I_twice_integral_tilde} it remains to show that under the assumption that $\alpha \kappa$ is even,
\begin{equation}\label{eq:I_remains_to_show_tilde}
\frac 1 {2\pi \ii} \int_{-\pi + \ii \infty}^{-\pi -\ii \infty} \frac{(\tilde F(y))^{\nu-\kappa}}{(\cos y)^{\alpha \nu +1} } \dd y
=
\frac 1 {2\pi} \int_{-\infty}^{+\infty} \frac{(\tilde F(\ii u))^{\nu-\kappa}}{(\cosh u)^{\alpha \nu + 1}} \dd u.
\end{equation}
Observe that  $\cos (\pm x\pm \pi) = -\cos x$. By the change of variables $w= - x -\pi$, we have
\begin{align*}
\tilde F(-\pi - \ii u)
&=
\int_{-\pi/2}^{-\pi - \ii u} (\cos x)^{\alpha-1} \dd x
=
- \int_{-\pi/2}^{\ii u} (\cos (-w-\pi))^{\alpha-1} \dd w\\
&=
(-1)^{\alpha} \int_{-\pi/2}^{\ii u} (\cos w)^{\alpha-1} \dd w
=
(-1)^{\alpha} \tilde F(\ii u).
\end{align*}
By the substitution $y = -\pi -\ii u$ it follows that
\begin{align*}
\frac 1 {2\pi \ii} \int_{-\pi + \ii \infty}^{-\pi -\ii \infty} \frac{(\tilde F(y))^{\nu-\kappa}}{(\cos y)^{\alpha \nu + 1} } \dd y
&=
-\frac 1 {2\pi} \int_{-\infty}^{+\infty} \frac{(\tilde F(-\pi -\ii u))^{\nu-\kappa}}{(\cos (-\pi -\ii u))^{\alpha \nu + 1} } \dd u\\
&=
(-1)^{1 + \alpha(\nu-\kappa)- (\alpha \nu + 1)} \cdot \frac {1} {2\pi} \int_{-\infty}^{+\infty} \frac{(\tilde F(\ii u))^{\nu-\kappa}}{(\cosh u)^{\alpha \nu + 1} } \dd u,
\end{align*}
which proves~\eqref{eq:I_remains_to_show_tilde} since $\alpha \kappa$ is even.
\end{proof}

\begin{remark}
In the case when $\alpha \kappa$ is odd, essentially the same argument shows that
$$
\Res\limits_{x=0} \left[\frac{\left(\int_{0}^x (\sin y)^{\alpha-1} \dd y\right)^{\nu-\kappa}}{(\sin x)^{\alpha \nu + 1}}\right] = 0.
$$
\end{remark}

\subsection{Proof of Theorem~\ref{theo:poisson_polyhedra}}
Fix $\alpha>0$. We claim that for all $d\in\N$ and $k\in \{1,\ldots,d\}$  we have
\begin{equation}\label{eq:f_k_wspom}
\E f_{k-1}(\conv \Pi_{d,\alpha})
=
\alpha^{d}  \binom{d}{k}
\cdot
\tilde c_{\frac{\alpha+1}{2}}^{-k}
\cdot
\frac {1} {\pi} \int_{-\infty}^{+\infty} \frac{(\tilde F(\ii u))^{d-k}}{(\cosh u)^{\alpha d + 1} } \dd u.
\end{equation}
This is equivalent to the first claim of Theorem~\ref{theo:poisson_polyhedra}. If $\alpha\in\N$ and $\alpha k$ is even, \eqref{eq:int_residue_tilde} immediately yields the second claim of Theorem~\ref{theo:poisson_polyhedra}:
$$
\E f_{k-1}(\conv \Pi_{d,\alpha})
=
\alpha^{d}  \binom{d}{k}
\left(\frac{\sqrt \pi\, \Gamma(\frac \alpha2)}{\Gamma(\frac{\alpha+1}{2})}\right)^k
\Res\limits_{x=0} \left[\frac{\left(\int_{0}^x (\sin y)^{\alpha-1} \dd y\right)^{d-k}}{(\sin x)^{\alpha d+1}}\right].
$$
\begin{proof}[Proof of~\eqref{eq:f_k_wspom}]
Let us first assume that $\alpha k >1$, which is stronger than $\alpha>0$.
According to~\cite[Theorem~1.21]{beta_polytopes}, the expected number of $(k-1)$-dimensional faces of $\conv \Pi_{d,\alpha}$ is given by
\begin{equation} \label{eq:E_f_k_beta_prime_to_poisson_rep}
\E f_{k-1}(\conv \Pi_{d,\alpha})
=
2 \sum_{\substack{m\in \{k,\ldots,d\}\\ m\equiv d \Mod{2}}}
\tilde \bI_{\infty,m}(\alpha) \tilde \bJ_{m,k}\left(\frac{\alpha + m-1}{2}\right),
\end{equation}
where
\begin{equation}\label{eq:poisson_voronoi_f_k_addition_rep}
\tilde \bI_{\infty,m}(\alpha)
:=
\lim_{n\to\infty} \tilde \bI_{n,m}(\alpha)
=
\frac{\tilde c_{\frac{\alpha m+1}{2}}}{\tilde c_{\frac{\alpha+1}{2}}^m} \cdot \frac{\alpha^{m-1}}{m}
=
\frac{\Gamma({\frac{m\alpha  + 1} 2})}{\Gamma(\frac {m\alpha}2)}
\left(\frac {\Gamma(\frac \alpha 2)} {\Gamma(\frac{\alpha+1} 2)}\right)^{m}
\frac{(\sqrt{\pi}\alpha)^{m-1}}{m}.
\end{equation}
For our purposes, it suffices to define $\tilde \bI_{\infty,m}(\alpha)$ as the right-hand side of~\eqref{eq:poisson_voronoi_f_k_addition_rep}. We shall not use the fact that $\lim_{n\to\infty} \tilde \bI_{n,m}(\alpha)$ equals the right-hand side; see the proof of Theorem~1.21 and Lemma~4.9 in~\cite{beta_polytopes} for the proof.  On the other hand, we know from~\eqref{eq:xi_n_k_tilde} that
$$
\tilde \bJ_{m,k}\left(\frac{\alpha + m-1}{2}\right)
=
\alpha^{k-m} \frac{m!}{k!}\cdot  \tilde c_{\frac{\alpha + 1}{2}}^{m-k} \cdot \frac{\tilde c_{\frac{\alpha k + 1}{2}}}{\tilde c_{\frac{\alpha m + 1}{2}}}  \cdot \frac{\left(m - \frac 1\alpha\right) \mA\left[ m - \frac 2\alpha,k - \frac 2\alpha \right]}{\left(k - \frac 1\alpha\right) \mA\left[ k - \frac 2\alpha,k - \frac 2\alpha \right]},
$$
where we relied on the assumption $\alpha k >1$.
After some cancellations, we obtain
$$
\tilde \bI_{\infty,m}(\alpha) \tilde \bJ_{m,k}\left(\frac{\alpha + m-1}{2}\right)
=
\tilde \bI_{\infty,k}(\alpha)\cdot \frac{\left(m - \frac 1\alpha\right) \Gamma(m) \cdot \mA\left[ m - \frac 2\alpha,k - \frac 2\alpha \right]}{\left(k - \frac 1\alpha\right)\Gamma(k)\cdot  \mA\left[ k - \frac 2\alpha,k - \frac 2\alpha \right]}.
$$
Recall from Remark~\ref{rem:other_quantities_tilde} that the quantities defined by
$
\mtA[\nu,\kappa] := \mA[\nu,\kappa] \cdot \Gamma(\nu+1)
$
satisfy the recurrence relations
\begin{equation*}
\mtA[\nu,\kappa] - \mtA[\nu-2,\kappa]
=
\frac{\left(\nu - \frac 1 \alpha\right)\Gamma(\nu)}{\Gamma\left(\nu - \frac 2 \alpha + 1\right)} \mtA \left[\nu -  \frac 2\alpha, \kappa - \frac 2 \alpha\right].
\end{equation*}
It follows that
$$
\tilde \bI_{\infty,m}(\alpha) \tilde \bJ_{m,k}\left(\frac{\alpha + m-1}{2}\right)
=
\tilde \bI_{\infty,k}(\alpha)\cdot \frac{\mtA[m,k] - \mtA[m-2,k]}{\mtA[k,k]},
$$
where we also used that $\mtA[k-2,k]=0$ by definition. With this at hand, the right-hand side of~\eqref{eq:E_f_k_beta_prime_to_poisson_rep} becomes a telescope sum. Evaluating it, we obtain
$$
\E f_{k-1}(\conv \Pi_{d,\alpha})
= 2 \tilde \bI_{\infty,k}(\alpha)\cdot \sum_{\substack{m\in \{k,\ldots,d\}\\ m\equiv d \Mod{2}}}  \frac{\mtA[m,k] - \mtA[m-2,k]}{\mtA[k,k]}
= 2 \, \tilde \bI_{\infty,k}(\alpha) \cdot \frac{\mtA[d,k]}{\mtA[k,k]}.
$$
Using~\eqref{eq:poisson_voronoi_f_k_addition_rep}, the relation $\mtA[n,k] = n! \, \mA[n,k]$ and the formula $\mA[k,k] = \frac 1 k \, \tilde c_{\frac{\alpha k + 1}{2}}$ that follows from~\eqref{eq:a_diag_tilde}, we arrive at
$$
\E f_{k-1}(\conv \Pi_{d,\alpha}) = 2 \alpha^{k-1} \frac{d!}{k!} \cdot \tilde c_{\frac{\alpha+1}{2}}^{-k} \cdot  \mA[d,k].
$$
Recalling the definition of $\mA[d,k]$ given in~\eqref{eq:def_A_0_tilde2} completes the proof of~\eqref{eq:f_k_wspom} under the assumption $\alpha k >1$. To prove it under the weaker assumption $\alpha > 0$, we use analytic continuation. Fix $d\in\N$ and $k\in \{1,\ldots,d\}$. We claim that the function $\alpha \mapsto \E f_{k-1}(\conv \Pi_{d,\alpha})$  can be analytically continued to the half-plane $\Re \alpha >0$. Indeed, all terms on the right-hand side of~\eqref{eq:E_f_k_beta_prime_to_poisson_rep} can be analytically continued to this half-plane. For the $\tilde \bI_{\infty,m}$-terms this follows from~\eqref{eq:poisson_voronoi_f_k_addition_rep}, whereas for the $\tilde \bJ_{m,k}$-terms we established this in the proof of Theorem~\ref{theo:bJ_tilde_formula_integral}. As in this proof, one also shows that the right-hand side of~\eqref{eq:f_k_wspom} defines an analytic function of $\alpha$ in the half-plane $\Re \alpha > -\frac 1d$. By the uniqueness of analytic continuation, \eqref{eq:f_k_wspom} holds for all $\alpha > 0$.
\end{proof}


\subsection{Evaluating the residue: Proof of Proposition~\ref{prop:J_n_tilde_polynomial}}
The proof is similar to the proof of Proposition~\ref{prop:J_n_1_half_integer}, but is simpler. In view of Theorem~\ref{theo:bJ_tilde_formula_residue} we have to show that
$$
P_{\alpha,k}(n) := \alpha^{n-k} [x^{-1}] h_{n,k}(x)
$$
is a polynomial of $n$ with rational coefficients, where
$$
h_{n,k}(x) = \frac{\left(\int_{0}^x (\sin y)^{\alpha-1} \dd y\right)^{n-k}}{(\sin x)^{\alpha n - 1}}.
$$
Let us analyze the Laurent series of $h_{n,k}(x)$ at $x=0$. We have the Taylor series
$$
\int_{0}^x (\sin y)^{\alpha-1} \dd y = \frac {x^{\alpha}} {\alpha} \sum_{j=0}^\infty a_{j} x^{2j},
$$
where $a_{0},a_2,\ldots$ are rational numbers with $a_0=1$. Raising this series to the $(n-k)$-th power and arguing as in the proof of Proposition~\ref{prop:J_n_1_half_integer}, we obtain
\begin{equation}\label{eq:wspom_taylor1}
\left(\int_{0}^x (\sin y)^{\alpha-1} \dd y\right)^{n-k}
=
\frac {x^{\alpha (n-k)}} {\alpha^{n-k}} \sum_{j=0}^\infty b_{j}(n) x^{2j}
\end{equation}
for some polynomials $b_{0}(n),b_2(n),\ldots$ with rational coefficients and $b_0(n) = 1$. Similarly,
\begin{equation}\label{eq:wspom_taylor2}
(\sin x)^{\alpha n - 1} = x^{\alpha n - 1} \sum_{j=0}^\infty c_j(n) x^{2j}
\end{equation}
for some polynomials $c_0(n),c_2(n),\ldots$ with rational coefficients and $c_0(n) = 1$. Taking the quotient of the power series~\eqref{eq:wspom_taylor1} and~\eqref{eq:wspom_taylor2}, we obtain
$$
h_{n,k}(x) = x^{-\alpha k + 1} \alpha^{-(n-k)} \sum_{j=0}^\infty d_j(n) x^{2j}
$$
for some polynomials $d_0(n),d_2(n),\ldots$ with rational coefficients and $d_0(n) = 1$. For the residue of $h_{n,k}(x)$ we thus obtain
$$
\alpha^{n-k} [x^{-1}] h_{n,k}(x) = d_{\frac{\alpha k - 2}{2}}(n),
$$
which is a polynomial in $n$ with rational coefficients.
\hfill $\Box$

\subsection{Remarks on consistency of notation}\label{subsec:special_cases}
In the special case $\alpha=2$ we have two definitions of the quantities $\mtA[\nu,\kappa]$ and $\mtB\{\nu,\kappa\}$. These were given in Section~\ref{sec:simplest_case_proof}, see~\eqref{eq:def_A_alpha_2} and~\eqref{eq:def_B_alpha_2},  and in the present Section~\ref{sec:proof_beta_tilde}, see~\eqref{eq:def_B_tilde}, \eqref{eq:def_A_0_tilde1}, \eqref{eq:def_mtA_mtB}. Let us argue that these definitions are equivalent. For $\mtB\{\nu,\kappa\}$, it is evident that~\eqref{eq:def_B_alpha_2} is equivalent to~\eqref{eq:def_B_tilde} (where we take $\alpha=2$ and observe that $F(x) = 1 + \sin x$); see also~\eqref{eq:def_mtA_mtB}. To prove that the definitions of $\mtA[\nu,\kappa]$ given in~\eqref{eq:def_A_alpha_2} and~\eqref{eq:def_A_0_tilde1}, \eqref{eq:def_mtA_mtB} are equivalent, it suffices to check that with either of the definitions the matrices $((-1)^m \mtB\{n,m\})_{n,m\in \N}$ and $((-1)^k (m - \frac 12) \mtA[m-1, k-1])_{m,k\in\N}$ are inverses of each other. This has been established in Propositions~\ref{prop:inverse_matr_alpha_2} and~\ref{prop:relations_A_B_tilde}. Note that although Proposition~\ref{prop:relations_A_B_tilde} deals with $\mB$ and $\mA$ rather than with $\mtB$ and $\mtA$, both versions are equivalent when $\alpha=2$, as follows from~\eqref{eq:def_mtA_mtB}.

In the special case $\alpha=1$ the definitions of the quantities $\mtA[\nu,\kappa]$ and $\mtB\{\nu,\kappa\}$ given in~\eqref{eq:def_B_tilde}, \eqref{eq:def_A_0_tilde1}, \eqref{eq:def_mtA_mtB} are almost equivalent to the definitions of $A[n,k]$ and $B\{n,k\}$ given in~\cite{kabluchko_poisson_zero}. We have $\mtB\{n,k\} = B\{n,k\}$, which follows directly by comparing Equations~\eqref{eq:def_B_tilde}, \eqref{eq:def_mtA_mtB} of the present paper with Equations~(1.9), (2.14) of~\cite{kabluchko_poisson_zero}. Also, we have $\mtA[n,k] = \frac 12 A[n,k]$. This is most easily seen if $k$ is even. Indeed,  by Equation~(1.15) in~\cite{kabluchko_poisson_zero} (which we regard as the definition of $A[n,k]$) and Theorem~\ref{theo:poisson_zero_f_vect} of the present paper, we have
$$
A[n,k]
=
\frac{k!}{\pi^k} \E f_{n-k}(\mathcal Z_n)
=
\frac{n!}{(n-k)!} [x^k] \left(\frac x {\sin x}\right)^{n+1}
=
\frac{n!}{(n-k)!} \Res\limits_{x=0} \left[\frac{x^{n-k}}{(\sin x)^{n+1}}\right].
$$
The right-hand side equals $2 \mtA[n,k]$ by Proposition~\ref{prop:a_nu_kappa_residue_tilde}. Without parity restrictions on $k$, the equality $\mtA[n,k] = \frac 12 A[n,k]$ can be established by comparing the following two formulae for $\tilde \bJ_{n,k}(\frac{\alpha  + n - 1}{2})$: Equation~\eqref{eq:form_J_n_k_wspom_tilde} from the present paper and Equation~(3.9) from~\cite{kabluchko_poisson_zero}.

\section{Expected face numbers of beta and beta' polytopes}\label{sec:expected_face_beta_poly}
\subsection{Beta polytopes}
Let $X_1,\ldots, X_n$ be i.i.d.\ points in $\R^d$ with probability distribution $f_{d,\beta}$, where $\beta\geq -1$.  Their convex hull  $P_{n,d}^\beta :=[X_1,\ldots,X_n]$ is called the \textit{beta polytope}. We are interested in the expected $f$-vector of this random polytope.  The case when $\beta\to +\infty$ corresponds to the Gaussian polytope. By a result of Baryshnikov and Vitale~\cite{baryshnikov_vitale}, its expected $f$-vector coincides with the expected $f$-vector of the orthogonal projection of the regular simplex with $n$ vertices onto a random, uniformly distributed $d$-dimensional subspace. For the latter model, the expected $f$-vector  has been determined by Affentranger and Schneider~\cite{AS92} in terms of the internal and external angles of the regular simplex.
For finite values of $\beta$, explicit results have been available only in some special cases~\cite{buchta_mueller,buchta_mueller_tichy,affentranger,beta_polytopes_temesvari}, see also Section~\ref{subsec:intro}.  We are now in position to determine the complete expected $f$-vector.

\begin{theorem}\label{theo:beta_poly_f_vector}
Let $d\geq 3$ and $n\in\N$ be such that $n\geq d+1$. Also, let $\beta\geq -1$ and define $\alpha:= 2\beta+d$. Then, for all $k\in \{1,\ldots,d\}$ we have
$$
\E f_{k-1} (P_{n,d}^\beta)
=
\frac{2\cdot n!}{k!} \left(\frac{\Gamma(\frac{\alpha}{2})}{2\sqrt \pi \, \Gamma(\frac{\alpha+1}{2})}\right)^{n-k}
\sum_{\substack{m\in \{k,\ldots,d\} \\ m\equiv d \Mod{2}}}
\lB\{n,m\} \left(m+\frac 1\alpha\right) \lA\left[m+\frac 2\alpha, k+\frac 2\alpha\right],
$$
where $\lB\{\cdot,\cdot\}$ and $\lA[\cdot,\cdot]$ are given by~\eqref{eq:def_B} and~\eqref{eq:def_A_0}.
\end{theorem}
\begin{proof}
In~\cite[Theorem~1.2]{beta_polytopes} it was shown that
\begin{equation}\label{eq:beta_poly_f_vector}
\E f_{k-1} (P_{n,d}^\beta) = 2 \sum_{\substack{m\in \{k,\ldots,d\} \\ m\equiv d \Mod{2}}} \bI_{n,m}(\alpha) \bJ_{m,k}\left(\frac{\alpha - m + 1}{2}\right).
\end{equation}
For the terms appearing on the right-hand side we have shown in~\eqref{eq:I_n_m_b_n_m} and~\eqref{eq:form_J_n_k_wspom} that
\begin{align*}
\bI_{n,m}(\alpha)
&=
\alpha^{-n} n! c_{\frac{\alpha-1}{2}}^{n} \cdot \frac{c_{\frac{\alpha m-1}{2}} \alpha^{m}}{c_{\frac{\alpha-1}{2}}^m m!} \cdot \lB\{n,m\},
\\
\bJ_{m,k}\left(\frac{\alpha - m + 1}{2}\right)
&=
\alpha^{k-m} \frac{m!}{k!}
\cdot
c_{\frac{\alpha-1}{2}}^{m-k}
\cdot
c_{\frac{\alpha m}{2}}
\cdot
\frac {2\pi}\alpha \, \lA\left[ m+\frac 2\alpha,k+\frac 2\alpha \right].
\end{align*}
Plugging these formulae into~\eqref{eq:beta_poly_f_vector} we obtain the required result after straightforward transformations.
\end{proof}

Taking $\beta=0$ or $\beta=-1$ in Theorem~\ref{theo:beta_poly_f_vector}, we finally obtain a solution to Problem~A stated in the introductory Section~\ref{subsec:intro}. The involved integrals can be evaluated using computer algebra. The next theorem describes the arithmetic properties of the expected $f$-vector.

\begin{theorem}
\label{theo:arithm_beta_f_vect}
Let $d\in\N$, $n\in\N$ and $k\in \{1,\ldots,d\}$ be such that $n\geq d+1$. Also, let $\beta\geq -1$ be integer or half-integer.
\begin{itemize}
\item[(a)] If $2\beta + d$ is odd, then $\E f_{k-1} (P_{n,d}^\beta)$ is a rational number.
\item[(b)] If $2\beta + d$ is even, then  $\E f_{k-1} (P_{n,d}^\beta)$ can be expressed as a linear combination of the numbers $\pi^{-2r}$, where $r=0,1,\ldots, \lfloor \frac{n-k}{2}\rfloor$, with rational coefficients.
\end{itemize}
\end{theorem}
\begin{proof}
If $\alpha:= 2\beta + d$ is odd, then the terms $\bI_{n,m}(\alpha)$ and $\bJ_{m,k}\left(\frac{\alpha - m + 1}{2}\right)$ involved in~\eqref{eq:beta_poly_f_vector} are rational numbers by~\cite[Proposition~5.4]{kabluchko_algorithm} and Theorem~\ref{theo:arithm_J}, respectively. If $\alpha$ is even, the claim follows from~\cite[Proposition~5.6]{kabluchko_algorithm} and Theorem~\ref{theo:arithm_J}.
\end{proof}

\subsection{Beta' polytopes}
Let $\tilde X_1,\ldots, \tilde X_n$ be i.i.d.\ points in $\R^d$ with probability distribution $\tilde f_{d,\beta}$, where $\beta > \frac d2$.   Their convex hull  $\tilde P_{n,d}^\beta :=[X_1,\ldots,X_n]$ is called the \textit{beta' polytope}.  In the next theorem we compute its expected $f$-vector.

\begin{theorem}\label{theo:beta_poly_f_vector_tilde}
Let $d\in\N$ and $n\in\N$ be such that $n\geq d+1$. Let also $\beta > \frac d2$ and put $\alpha:= 2\beta - d > 0$. Then, for all $k\in \{1,\ldots,d\}$ such that $\alpha k>1$ we have
$$
\E f_{k-1} (\tilde P_{n,d}^\beta)
=
\frac{2\cdot n!}{k!} \left(\frac{\Gamma(\frac{\alpha+1}{2})}{2\sqrt \pi \, \Gamma(\frac{\alpha+2}{2})}\right)^{n-k}
\sum_{\substack{m\in \{k,\ldots,d\} \\ m\equiv d \Mod{2}}}
\mB\{n,m\} \left(m - \frac 1\alpha\right) \mA\left[m - \frac 2\alpha, k - \frac 2\alpha\right],
$$
where $\mB\{\cdot,\cdot\}$ and $\mA[\cdot,\cdot]$ are given by~\eqref{eq:def_B_tilde} and~\eqref{eq:def_A_0_tilde1}.
\end{theorem}
\begin{proof}
In~\cite[Theorem~1.14]{beta_polytopes} it was shown that
\begin{equation}\label{eq:beta_poly_f_vector_tilde}
\E f_{k-1} (\tilde P_{n,d}^\beta) = 2 \sum_{\substack{m\in \{k,\ldots,d\} \\ m\equiv d \Mod{2}}} \tilde \bI_{n,m}(\alpha) \tilde \bJ_{m,k}\left(\frac{\alpha + m - 1}{2}\right).
\end{equation}
For the terms appearing on the right-hand side we have shown in~\eqref{eq:I_n_m_b_n_m_tilde} and~\eqref{eq:form_J_n_k_wspom_tilde} that
\begin{align*}
\tilde \bI_{n,m}(\alpha)
&=
\alpha^{-n} n! \tilde c_{\frac{\alpha+1}{2}}^{n} \cdot \frac{\tilde c_{\frac{\alpha m + 1}{2}} \alpha^{m}}{\tilde c_{\frac{\alpha+1}{2}}^m m!} \cdot \mB\{n,m\},
\\
\tilde \bJ_{m,k}\left(\frac{\alpha  + m -  1}{2}\right)
&=
\alpha^{k-m} \frac{m!}{k!}
\cdot
\tilde c_{\frac{\alpha+1}{2}}^{m-k}
\cdot
\tilde c_{\frac{\alpha m}{2}}
\cdot
\frac {2\pi}\alpha \, \mA\left[ m - \frac 2\alpha,k - \frac 2\alpha \right].
\end{align*}
Plugging these formulae into~\eqref{eq:beta_poly_f_vector_tilde} and performing straightforward simplifications yields the required formula.
\end{proof}

The special case of beta' polytopes with $\alpha=1$ (and $\beta= \frac{d+1}{2}$) is related to random convex hulls on the half-sphere, a model first studied in~\cite{barany_etal}. The connection to beta' polytopes has been exploited in~\cite{convex_hull_sphere} and~\cite{kabluchko_poisson_zero}.   Let us also mention that the Poisson polytope $\Pi_{d,\alpha}$  with $\alpha = 2\beta-d$ can be seen as the weak limit, as $n\to\infty$,  of the beta'-polytope $\tilde P_{n,d}^{\beta}$ rescaled by a constant multiple of $n^{-1/\alpha}$; see~\cite[Theorem~3.1]{davis_mulrow_resnick} and~\cite[Theorem~2.1]{convex_hull_sphere}, \cite[Sections~1.4,1.5]{beta_polytopes}. The next theorem clarifies the arithmetic structure of the expected $f$-vector of $\tilde P_{n,d}^{\beta}$.

\begin{theorem}
\label{theo:arithm_beta_f_vect_tilde}
Let $d\in\N$, $n\in\N$ and $k\in \{1,\ldots,d\}$ be such that $n\geq d+1$. Also, let $\beta>\frac d2$ be integer or half-integer.
\begin{itemize}
\item[(a)] If $2\beta - d$ is even, then $\E f_{k-1} (\tilde P_{n,d}^\beta)$ is a rational number.
\item[(b)] If $2\beta - d$ is odd, then  $\E f_{k-1} (\tilde P_{n,d}^\beta)$ can be expressed as a linear combination of the numbers $\pi^{-2r}$, where $r=0,1,\ldots, \lfloor \frac{n-k}{2}\rfloor$, with rational coefficients.
\end{itemize}
\end{theorem}
\begin{proof}
If $\alpha:= 2\beta - d$ is even, then the terms $\tilde \bI_{n,m}(\alpha)$ and $\tilde \bJ_{m,k}\left(\frac{\alpha + m - 1}{2}\right)$ appearing in~\eqref{eq:beta_poly_f_vector_tilde} are rational numbers by~\cite[Theorem~2.10(a)]{kabluchko_algorithm} and Theorem~\ref{theo:arithm_J_tilde}, respectively. If $\alpha$ is even, the claim follows from~\cite[Theorem~2.10(b)]{kabluchko_algorithm} and Theorem~\ref{theo:arithm_J_tilde}.
\end{proof}

\section{Poincar\'e and Dehn-Sommerville relations}\label{sec:dehn_sommerville}
\subsection{Statement of the relations}
Let $T$ be a $d$-dimensional simplex with $n = d+1$ vertices and recall that $\sigma_k(T)$ denotes the sum of internal angles of $T$ at all of its $k$-vertex faces, where $k\in \{1,\ldots, n\}$. The quantities $\sigma_k(T)$ satisfy the following \textit{Poincar\'e relations}:
\begin{equation}\label{eq:poincare_angles}
\sum_{k=m}^n (-1)^k \binom{k}{m} \sigma_{k}(T) = (-1)^n\sigma_{m}(T),
\end{equation}
for all $m\in \{0,\ldots,n\}$; see~\cite{poincare} and~\cite[p.~304]{GruenbaumBook}.
In the special case $m=0$ (where we put $\sigma_0(T) := 0$) this relation is also called the \textit{Euler-Gram relation}.

On the other hand, the $f$-vector of any $d$-dimensional simplicial polytope $P$ satisfies the so-called \textit{Dehn-Sommerville relations} which we shall write in the form
\begin{equation}\label{eq:dehn_sommerville}
\sum_{k=m}^d (-1)^k \binom{k}{m} f_{k-1}(P) = (-1)^d f_{m-1}(P),
\end{equation}
for all $m\in \{0,\ldots,d\}$; see~\cite[p.~146]{GruenbaumBook}. The special case $m=0$ (where we put $f_{-1}(P) = 1$) is the classical \textit{Euler relation}.

Due to the obvious similarity between~\eqref{eq:poincare_angles} and~\eqref{eq:dehn_sommerville}, it is natural to treat them in a uniform way. To this end, let us consider some real numbers $z_0,\ldots, z_n$ satisfying the general relations
\begin{equation}\label{eq:dehn_sommerville_general}
\sum_{k=m}^n (-1)^k \binom{k}{m} z_k = (-1)^n z_m,
\qquad
m\in \{0,\ldots,n\}.
\end{equation}
Obviously, both~\eqref{eq:poincare_angles} and~\eqref{eq:dehn_sommerville} become special cases of~\eqref{eq:dehn_sommerville_general} upon putting $z_k=\sigma_k(T)$ and $z_k= f_{k-1}(P)$, respectively.

\subsection{Even-indexed and odd-indexed terms}
It is well known that relations~\eqref{eq:dehn_sommerville_general} are highly linearly dependent. Various ways of constructing an equivalent system of linearly independent relations are discussed in~\cite[\S~9.2]{GruenbaumBook}. With the aim of proving Propositions~\ref{prop:ugly_residue}, \ref{prop:ugly_residue_tilde}, \ref{prop:poisson_polytope_ugly}, we are interested in expressing the even-indexed quantities $z_k$ through the odd-indexed ones, and vice versa. To state the corresponding formula, let $B_0=1,B_2=\frac 16,B_4=-\frac 1 {30},\ldots$ denote the Bernoulli numbers~\cite[Section~6.5]{graham_knuth_patashnik_book} which appear in the expansions
\begin{equation}\label{eq:bernoulli_def}
\frac z2  \coth \frac{z}{2} = \sum_{n=0}^\infty \frac{B_{2n}}{(2n)!} z^{2n},
\qquad
\frac z2  \tanh \frac{z}{2} = \sum_{n=0}^\infty \frac{(2^{2n}-1)B_{2n}}{(2n)!} z^{2n}, \qquad |z|<\pi.
\end{equation}
\begin{proposition}\label{prop:solve_poincare_rel}
Let  $z_0,\ldots, z_n$ be real numbers satisfying~\eqref{eq:dehn_sommerville_general}.
Then, for all $n\in\N$ and $k\in \{1,\ldots,n-1\}$ we have
\begin{align}
z_{k}
&=
2 \sum_{r=-1,1,3,\ldots} \frac{ B_{r+1}\, (k+r)!}{(r+1)!k!} z_{k+r},
&&\text{ if $n-k$ is even},\label{eq:sol_bernoulli1a}\\
z_{k}
&=
2 \sum_{r=1,3,5,\ldots} (2^{r+1}-1) \frac{ B_{r+1}\, (k+r)!}{(r+1)!k!} z_{k+r},
&&\text{ if $n-k$ is odd}.\label{eq:sol_bernoulli1b}
\end{align}
\end{proposition}
\begin{proof}
Equations~\eqref{eq:sol_bernoulli1b} are due to Peschl~\cite{peschl} and, as he mentions, some ideas go back to Schl\"afli~\cite{schlaefli_vielfache_kont}. A simpler derivation was given by Guinand~\cite{guinand}. We shall give here an independent proof of both, \eqref{eq:sol_bernoulli1a} and~\eqref{eq:sol_bernoulli1b}, since parts of this argument will be needed below.
Introduce new variables $y_0,\ldots,y_n$ by $z_k =\frac {n!}{k!} y_{n-k}$ for all $k\in \{0,\ldots,n\}$ and put $y_N:=0$ for $N>n$. With this notation, Equation~\eqref{eq:dehn_sommerville_general} takes the form
$$
\sum_{k=m}^n (-1)^{n-k} \frac{y_{n-k}}{(k-m)!} = y_{n-m},
\qquad
m\in \{0,\ldots,n\}.
$$
With $\ell:= n-m\in \{0,\ldots,n\}$, we can rewrite the equation as follows:
\begin{align}
&\frac{y_0}{\ell!} + \frac{y_2}{(\ell-2)!}  + \ldots + \frac{y_{\ell-1}}{1!} = \frac{y_1}{(\ell-1)!} + \frac{y_3}{(\ell-3)!} + \ldots + \frac{y_{\ell-2}}{2!} + 2y_\ell,
&&\text{if $\ell$ is odd},\label{eq:proof_bernoulli_relation_y2}\\
&\frac{y_0}{\ell!} + \frac{y_2}{(\ell-2)!}  + \ldots + \frac{y_{\ell-2}}{2!} = \frac{y_1}{(\ell-1)!} + \frac{y_3}{(\ell-3)!} + \ldots + \frac{y_{\ell-1}}{1!},
&&\text{if $\ell$ is even},
\label{eq:proof_bernoulli_relation_y1}
\end{align}
where we cancelled $y_\ell$ in~\eqref{eq:proof_bernoulli_relation_y1}.
Introduce the generating functions
$$
f_{\text{even}}(u) := \sum_{\substack{r=0,2,4,\ldots\\r\leq n}} y_ru^r,
\qquad
f_{\text{odd}}(u) := \sum_{\substack{r=1,3,5,\ldots\\ r\leq n}} y_ru^r.
$$
With this notation, we can write~\eqref{eq:proof_bernoulli_relation_y2} and \eqref{eq:proof_bernoulli_relation_y1} as
\begin{align}
f_{\text{even}}(u)\cdot \sinh u &= f_{\text{odd}}(u) \cdot (\cosh u + 1) + O(u^{n+1}),\label{eq:f_even_odd2}
\\
f_{\text{even}}(u)\cdot (\cosh u -1) &= f_{\text{odd}}(u) \cdot \sinh u + O(u^{n+1}). \label{eq:f_even_odd1}
\end{align}
Dividing both sides of~\eqref{eq:f_even_odd2} and~\eqref{eq:f_even_odd1}  by  $\sinh u = u + o(u)$, we obtain
\begin{equation}\label{eq:f_even_f_odd_relation}
f_{\text{even}}(u) = f_{\text{odd}}(u) \cdot \coth \left(\frac u2\right) + O(u^{n}),
\qquad
f_{\text{odd}}(u) = f_{\text{even}}(u) \cdot \tanh \left(\frac u2\right) + O(u^{n}).
\end{equation}
Recalling the Laurent series of $\coth (\frac u2)$  and $\tanh (\frac u2)$ given in~\eqref{eq:bernoulli_def} and comparing the coefficients of $u^{n-k}$, where $k\in \{1,\ldots,n\}$,  we get
\begin{align*}
y_{n-k} &= 2\sum_{r = -1,1,3,\ldots} \frac{B_{r+1}}{(r+1)!} y_{n-k-r},
&&\text{ if $n-k$ is even},\\
y_{n-k} &= 2\sum_{r = 1,3,\ldots} (2^{r+1}-1)\frac{B_{r+1}}{(r+1)!} y_{n-k-r},
&&\text{ if $n-k$ is odd.}
\end{align*}
Recalling that $y_{n-k} = \frac {k!}{n!} z_k$ yields~\eqref{eq:sol_bernoulli1a} and~\eqref{eq:sol_bernoulli1b}.
\end{proof}

\begin{remark}
Relations involving Bernoulli numbers hold also for the expected volumes of the convex hulls of i.i.d.\ samples from arbitrary distributions. These relations were derived in the works of Affentranger~\cite{affentranger_on_buchta,affentranger_on_buchta_rem}, Buchta~\cite{buchta_distr_indep} and Badertscher~\cite{badertscher}; see also~\cite[Theorem~8.2.6]{SW08}.
\end{remark}

\subsection{Proof of Proposition~\ref{prop:ugly_residue}}
Fix some integer $n\geq 3$ and some even $\alpha \geq n-3$.
Since the Poincar\'e relations are linear,  we can apply them to the beta simplex and take the expectation, which yields
$$
\sum_{k=m}^n (-1)^k \binom{k}{m} \bJ_{n,k}\left(\frac{\alpha - n + 1}{2}\right) = (-1)^n\bJ_{n,m}\left(\frac{\alpha - n + 1}{2}\right),
\qquad
\bJ_{n,0}\left(\frac{\alpha - n + 1}{2}\right):=0,
$$
for all $m\in\{0,\ldots,n\}$. We can therefore apply Proposition~\ref{prop:solve_poincare_rel} and its proof to the quantities $z_k := \bJ_{n,k}(\frac{\alpha - n + 1}{2})$, where $k\in \{1,\ldots,n\}$, and $z_0:=0$.  For all $k\in \{1,\ldots,n\}$ such that $r := n-k$ is odd, the quantities $y_{r}= \frac{(n-r)!}{n!}z_{n-r}$ appearing in the proof of Proposition~\ref{prop:solve_poincare_rel} are given by Theorem~\ref{theo:bJ_formula_residue} (i) as follows:
$$
y_r =
\pi\, c_{\frac{\alpha n}2}
\Res\limits_{x=0} \left[\frac{\left(c_{\frac{\alpha-1}{2}} \int_{0}^x (\sin y)^{\alpha} \dd y\right)^{r}}{r!\, (\sin x)^{\alpha n +2}}\right].
$$
It follows that the function $f_{\text{odd}}(u)$ is given by
$$
f_{\text{odd}}(u) = \pi\, c_{\frac{\alpha n}2}
\sum_{\substack{r=1,3,5,\ldots\\r\leq n}}
\Res\limits_{x=0} \left[\frac{\left( u c_{\frac{\alpha-1}{2}} \int_{0}^x (\sin y)^{\alpha} \dd y\right)^{r}}{r!\,(\sin x)^{\alpha n +2}}\right].
$$
Interchanging the sum and the residue and using the Taylor series of the $\sinh$-function, we can write this as
$$
f_{\text{odd}}(u) = \pi\, c_{\frac{\alpha n}2}
\Res\limits_{x=0} \left[\frac{ \sinh \left(u c_{\frac{\alpha-1}{2}} \int_{0}^x (\sin y)^{\alpha} \dd y\right)}{(\sin x)^{\alpha n +2}}\right]
+O(u^{n+1}).
$$
It follows from the first equation in~\eqref{eq:f_even_f_odd_relation} that
$$
f_{\text{even}}(u) = \pi\, c_{\frac{\alpha n}2} \coth \left(\frac u2\right) \cdot \Res\limits_{x=0} \left[\frac{ \sinh \left(u c_{\frac{\alpha-1}{2}} \int_{0}^x (\sin y)^{\alpha} \dd y\right)}{(\sin x)^{\alpha n +2}}\right] + O(u^{n}).
$$
Recall that the residue is the coefficient of $x^{-1}$. Let $k\in \{1,\ldots,n\}$ be such that $r = n-k$ is even. Comparing the coefficients of $u^{n-k}x^{-1}$, we obtain
$$
\bJ_{n,k}\left(\frac{\alpha - n + 1}{2}\right) = \frac{n!}{k!} y_{n-k}
=
\frac{n!}{k!} \pi\, c_{\frac{\alpha n}2} \cdot [u^{n-k}x^{-1}] \left(\frac{\sinh \left(u c_{\frac{\alpha-1}{2}} \int_{0}^x (\sin y)^{\alpha} \dd y\right)}{\tanh \left(\frac u2\right) \cdot (\sin x)^{\alpha n +2}}\right).
$$
To complete the proof, we observe that  $\sinh v  = \sin (\ii v)$ and $\tanh v = \tan (\ii v)$, hence replacing the hyperbolic functions by the corresponding trigonometric functions yields an additional factor of $(-1)^{(n-k)/2}$.
\hfill $\Box$

The proofs of Propositions~\ref{prop:ugly_residue_tilde} and~\ref{prop:poisson_polytope_ugly} are similar except that now one has to use one of the equalities in~\eqref{eq:f_even_f_odd_relation} depending on whether the codimension $d-k$ or $n-k$ is even or odd.


\section*{Acknowledgement}
Supported by the German Research Foundation under Germany's Excellence Strategy  EXC 2044 -- 390685587, Mathematics M\"unster: Dynamics - Geometry - Structure.

\bibliography{angles_rec}

\begin{thebibliography}{61}
\providecommand{\natexlab}[1]{#1}
\providecommand{\url}[1]{\texttt{#1}}
\expandafter\ifx\csname urlstyle\endcsname\relax
  \providecommand{\doi}[1]{doi: #1}\else
  \providecommand{\doi}{doi: \begingroup \urlstyle{rm}\Url}\fi

\bibitem[Affentranger(1988{\natexlab{a}})]{affentranger_on_buchta}
F.~Affentranger.
\newblock Generalization of a formula of {C}. {B}uchta about the convex hull of
  random points.
\newblock \emph{Elem. Math.}, 43\penalty0 (2):\penalty0 39--45,
  1988{\natexlab{a}}.

\bibitem[Affentranger(1988{\natexlab{b}})]{affentranger_on_buchta_rem}
F.~Affentranger.
\newblock Remarks on the note: ``{G}eneralization of a formula of {C}. {B}uchta
  about the convex hull of random points''.
\newblock \emph{Elem. Math.}, 43\penalty0 (5):\penalty0 151--152,
  1988{\natexlab{b}}.

\bibitem[{Affentranger}(1991)]{affentranger}
F.~{Affentranger}.
\newblock {The convex hull of random points with spherically symmetric
  distributions.}
\newblock \emph{{Rend. Semin. Mat., Torino}}, 49\penalty0 (3):\penalty0
  359--383, 1991.

\bibitem[Affentranger and Schneider(1992)]{AS92}
F.~Affentranger and R.~Schneider.
\newblock Random projections of regular simplices.
\newblock \emph{Discrete Comput. Geom.}, 7\penalty0 (1):\penalty0 219--226,
  1992.
\newblock \doi{10.1007/BF02187839}.

\bibitem[Badertscher(1989)]{badertscher}
E.~Badertscher.
\newblock An explicit formula about the convex hull of random points.
\newblock \emph{Elem. Math.}, 44\penalty0 (4):\penalty0 104--106, 1989.

\bibitem[{B\'ar\'any} et~al.(2017){B\'ar\'any}, {Hug}, {Reitzner}, and
  {Schneider}]{barany_etal}
I.~{B\'ar\'any}, D.~{Hug}, M.~{Reitzner}, and R.~{Schneider}.
\newblock {Random points in halfspheres.}
\newblock \emph{{Random Struct. Algorithms}}, 50\penalty0 (1):\penalty0 3--22,
  2017.
\newblock \doi{10.1002/rsa.20644}.

\bibitem[{Baryshnikov} and {Vitale}(1994)]{baryshnikov_vitale}
Y.~M. {Baryshnikov} and R.~A. {Vitale}.
\newblock {Regular simplices and Gaussian samples.}
\newblock \emph{{Discrete Comput. Geom.}}, 11\penalty0 (2):\penalty0 141--147,
  1994.
\newblock \doi{10.1007/BF02574000}.

\bibitem[Buchta(1984)]{buchta_zufallspolygone}
C.~Buchta.
\newblock Zufallspolygone in konvexen {V}ielecken.
\newblock \emph{J. Reine Angew. Math.}, 347:\penalty0 212--220, 1984.
\newblock \doi{10.1515/crll.1984.347.212}.
\newblock URL \url{https://doi.org/10.1515/crll.1984.347.212}.

\bibitem[Buchta(1985)]{buchta_polyeder}
C.~Buchta.
\newblock Zuf\"{a}llige {P}olyeder---eine \"{U}bersicht.
\newblock In \emph{Zahlentheoretische {A}nalysis}, volume 1114 of \emph{Lecture
  Notes in Math.}, pages 1--13. 1985.
\newblock \doi{10.1007/BFb0101638}.
\newblock URL \url{https://doi.org/10.1007/BFb0101638}.

\bibitem[Buchta(1990)]{buchta_distr_indep}
C.~Buchta.
\newblock Distribution-independent properties of the convex hull of random
  points.
\newblock \emph{J. Theoret. Probab.}, 3\penalty0 (3):\penalty0 387--393, 1990.
\newblock \doi{10.1007/BF01061259}.
\newblock URL \url{https://doi.org/10.1007/BF01061259}.

\bibitem[Buchta and M\"uller(1984)]{buchta_mueller}
C.~Buchta and J.~M\"uller.
\newblock Random polytopes in a ball.
\newblock \emph{J. Appl. Probab.}, 21\penalty0 (4):\penalty0 753--762, 1984.

\bibitem[Buchta and Reitzner(2001)]{buchta_reitzner}
C.~Buchta and M.~Reitzner.
\newblock The convex hull of random points in a tetrahedron: solution of
  {B}laschke's problem and more general results.
\newblock \emph{J. Reine Angew. Math.}, 536:\penalty0 1--29, 2001.
\newblock \doi{10.1515/crll.2001.050}.
\newblock URL \url{https://doi.org/10.1515/crll.2001.050}.

\bibitem[Buchta et~al.(1985)Buchta, M\"uller, and Tichy]{buchta_mueller_tichy}
C.~Buchta, J.~M\"uller, and R.~F. Tichy.
\newblock Stochastical approximation of convex bodies.
\newblock \emph{Math. Ann.}, 271\penalty0 (2):\penalty0 225--235, 1985.

\bibitem[Calka(2019)]{calka_classical_problems}
P.~Calka.
\newblock Some classical problems in random geometry.
\newblock In \emph{Stochastic geometry}, volume 2237 of \emph{Lecture Notes in
  Math.}, pages 1--43. Springer, Cham, 2019.

\bibitem[Chiu et~al.(2013)Chiu, Stoyan, Kendall, and Mecke]{stoyan_etal_book}
S.~N. Chiu, D.~Stoyan, W.~S. Kendall, and J.~Mecke.
\newblock \emph{Stochastic geometry and its applications}.
\newblock Wiley Series in Probability and Statistics. John Wiley \& Sons, Ltd.,
  Chichester, third edition, 2013.
\newblock \doi{10.1002/9781118658222}.
\newblock URL \url{https://doi.org/10.1002/9781118658222}.

\bibitem[Davis et~al.(1987)Davis, Mulrow, and Resnick]{davis_mulrow_resnick}
R.~Davis, E.~Mulrow, and S.~Resnick.
\newblock The convex hull of a random sample in {${\bf R}^2$}.
\newblock \emph{Comm. Statist. Stochastic Models}, 3\penalty0 (1):\penalty0
  1--27, 1987.
\newblock \doi{10.1080/15326348708807044}.
\newblock URL \url{https://doi.org/10.1080/15326348708807044}.

\bibitem[Efron(1965)]{efron}
B.~Efron.
\newblock The convex hull of a random set of points.
\newblock \emph{Biometrika}, 52:\penalty0 331--343, 1965.

\bibitem[Flajolet and Sedgewick(2009)]{Flajolet_book}
P.~Flajolet and R.~Sedgewick.
\newblock \emph{Analytic combinatorics}.
\newblock Cambridge University Press, Cambridge, 2009.

\bibitem[Fomin and Reading(2007)]{fomin_reading}
S.~Fomin and N.~Reading.
\newblock Root systems and generalized associahedra.
\newblock In \emph{Geometric combinatorics}, volume~13 of \emph{IAS/Park City
  Math. Ser.}, pages 63--131. Amer. Math. Soc., Providence, RI, 2007.

\bibitem[Gilbert(1962)]{gilbert}
E.~N. Gilbert.
\newblock Random subdivisions of space into crystals.
\newblock \emph{Ann. Math. Statist.}, 33:\penalty0 958--972, 1962.
\newblock \doi{10.1214/aoms/1177704464}.
\newblock URL \url{https://doi.org/10.1214/aoms/1177704464}.

\bibitem[G{\"otze} et~al.(2019)G{\"otze}, Kabluchko, and
  Zaporozhets]{goetze_kabluchko_zaporozhets}
F.~G{\"otze}, Z.~Kabluchko, and D.~Zaporozhets.
\newblock Grassmann angles and absorption probabilities of {G}aussian convex
  hulls.
\newblock Preprint at http://arxiv.org/abs/1911.04184, 2019.

\bibitem[{Graham} et~al.(1994){Graham}, {Knuth}, and
  {Patashnik}]{graham_knuth_patashnik_book}
R.~L. {Graham}, D.~E. {Knuth}, and O.~{Patashnik}.
\newblock \emph{{Concrete mathematics: a foundation for computer science}}.
\newblock Amsterdam: Addison-Wesley Publishing Group, 2nd ed. edition, 1994.

\bibitem[Gr\"unbaum(2003)]{GruenbaumBook}
B.~Gr\"unbaum.
\newblock \emph{Convex {P}olytopes}, volume 221 of \emph{Graduate Texts in
  Mathematics}.
\newblock Springer-Verlag, New York, second edition, 2003.
\newblock \doi{10.1007/978-1-4613-0019-9}.
\newblock Prepared and with a preface by V. Kaibel, V. Klee and G. M. Ziegler.

\bibitem[Guinand(1959)]{guinand}
A.~P. Guinand.
\newblock A note on the angles in an $n$-dimensional simplex.
\newblock \emph{Proc. of the Glasgow Math. Assoc.}, 4\penalty0 (2):\penalty0
  58--61, 1959.
\newblock \doi{10.1017/S2040618500033888}.

\bibitem[Hadwiger(1979)]{hadwiger}
H.~Hadwiger.
\newblock Gitterpunktanzahl im {S}implex und {W}ills'sche {V}ermutung.
\newblock \emph{Math. Ann.}, 239\penalty0 (3):\penalty0 271--288, 1979.

\bibitem[H\"orrmann et~al.(2015)H\"orrmann, Hug, Reitzner, and
  Th\"ale]{HoermannHugReitznerThaele}
J.~H\"orrmann, D.~Hug, M.~Reitzner, and C.~Th\"ale.
\newblock Poisson polyhedra in high dimensions.
\newblock \emph{Adv. Math.}, 281:\penalty0 1--39, 2015.

\bibitem[Hug(2013)]{hug_rev}
D.~Hug.
\newblock Random polytopes.
\newblock In \emph{Stochastic geometry, spatial statistics and random fields},
  volume 2068 of \emph{Lecture Notes in Math.}, pages 205--238. Springer,
  Heidelberg, 2013.
\newblock \doi{10.1007/978-3-642-33305-7_7}.
\newblock URL \url{https://doi.org/10.1007/978-3-642-33305-7_7}.

\bibitem[Kabluchko(2019{\natexlab{a}})]{kabluchko_algorithm}
Z.~Kabluchko.
\newblock Recursive scheme for angles of random simplices, and applications to
  random polytopes., 2019{\natexlab{a}}.
\newblock Preprint at arXiv: 1907.07534.

\bibitem[Kabluchko(2019{\natexlab{b}})]{kabluchko_angles}
Z.~Kabluchko.
\newblock Angle sums of random simplices in dimensions $3$ and $4$.
\newblock \emph{Proc. AMS, to appear}, 2019{\natexlab{b}}.
\newblock Preprint at arXiv: 1905.01533.

\bibitem[Kabluchko(2019{\natexlab{c}})]{kabluchko_poisson_zero}
Z.~Kabluchko.
\newblock Expected $f$-vector of the {P}oisson zero polytope and random convex
  hulls in the half-sphere, 2019{\natexlab{c}}.
\newblock Preprint at arXiv: 1901.10528.

\bibitem[Kabluchko and Zaporozhets()]{kabluchko_zaporozhets_ongoing}
Z.~Kabluchko and D.~Zaporozhets.
\newblock Work in progress.

\bibitem[Kabluchko and Zaporozhets(2017)]{kabluchko_zaporozhets_absorption}
Z.~Kabluchko and D.~Zaporozhets.
\newblock Absorption probabilities for {G}aussian polytopes, and regular
  spherical simplices.
\newblock \emph{Adv. Appl. Probab., to appear}, 2017.
\newblock Preprint at arXiv: 1704.04968.

\bibitem[Kabluchko and Zaporozhets(2018)]{kabluchko_zaporozhets_gauss_simplex}
Z.~Kabluchko and D.~Zaporozhets.
\newblock Angles of the {G}aussian simplex.
\newblock \emph{Zap. Nauchn. Sem. POMI}, 476:\penalty0 79--91, 2018.
\newblock Preprint at arXiv: 1801.08008.

\bibitem[Kabluchko et~al.(2018)Kabluchko, Th{\"a}le, and
  Zaporozhets]{beta_polytopes}
Z.~Kabluchko, C.~Th{\"a}le, and D.~Zaporozhets.
\newblock Beta polytopes and {P}oisson polyhedra: {$f$}-vectors and angles.
\newblock Preprint at http://arxiv.org/abs/1805.01338, 2018.

\bibitem[Kabluchko et~al.(2019{\natexlab{a}})Kabluchko, Marynych, Temesvari,
  and Th\"{a}le]{convex_hull_sphere}
Z.~Kabluchko, A.~Marynych, D.~Temesvari, and C.~Th\"{a}le.
\newblock Cones generated by random points on half-spheres and convex hulls of
  {P}oisson point processes.
\newblock \emph{Probab. Theory Related Fields}, 175\penalty0 (3-4):\penalty0
  1021--1061, 2019{\natexlab{a}}.
\newblock \doi{10.1007/s00440-019-00907-3}.
\newblock URL \url{https://doi.org/10.1007/s00440-019-00907-3}.

\bibitem[Kabluchko et~al.(2019{\natexlab{b}})Kabluchko, Temesvari, and
  Th\"{a}le]{beta_polytopes_temesvari}
Z.~Kabluchko, D.~Temesvari, and C.~Th\"{a}le.
\newblock Expected intrinsic volumes and facet numbers of random
  beta-polytopes.
\newblock \emph{Math. Nachr.}, 292\penalty0 (1):\penalty0 79--105,
  2019{\natexlab{b}}.
\newblock \doi{10.1002/mana.201700255}.
\newblock URL \url{https://doi.org/10.1002/mana.201700255}.

\bibitem[Kingman(1969)]{kingman_secants}
J.~F.~C. Kingman.
\newblock Random secants of a convex body.
\newblock \emph{J. Appl. Probability}, 6:\penalty0 660--672, 1969.
\newblock \doi{10.1017/s0021900200026693}.
\newblock URL \url{https://doi.org/10.1017/s0021900200026693}.

\bibitem[Knuth(1992)]{knuth}
D.~E. Knuth.
\newblock Two notes on notation.
\newblock \emph{Amer. Math. Monthly}, 99\penalty0 (5):\penalty0 403--422, 1992.
\newblock \doi{10.2307/2325085}.
\newblock URL \url{https://doi.org/10.2307/2325085}.

\bibitem[McMullen(1975)]{mcmullen}
P.~McMullen.
\newblock Non-linear angle-sum relations for polyhedral cones and polytopes.
\newblock \emph{Math. Proc. Cambridge Philos. Soc.}, 78\penalty0 (2):\penalty0
  247--261, 1975.
\newblock \doi{10.1017/S0305004100051665}.
\newblock URL \url{https://doi.org/10.1017/S0305004100051665}.

\bibitem[McMullen(1986)]{mcmullen_polyhedra}
P.~McMullen.
\newblock Angle-sum relations for polyhedral sets.
\newblock \emph{Mathematika}, 33\penalty0 (2):\penalty0 173--188 (1987), 1986.
\newblock \doi{10.1112/S0025579300011165}.
\newblock URL \url{https://doi.org/10.1112/S0025579300011165}.

\bibitem[Meijering(1953)]{meijering}
J.~L. Meijering.
\newblock Inferface area, edge length, and number of vertices in crystal
  aggregates with random nucleation.
\newblock \emph{Philips Res. Rep.}, 8:\penalty0 270--290, 1953.

\bibitem[Miles(1970)]{miles_synopsis}
R.~E. Miles.
\newblock A synopsis of ``{P}oisson flats in {E}uclidean spaces''.
\newblock \emph{Izv. Akad. Nauk Armjan. SSR Ser. Mat.}, 5\penalty0
  (3):\penalty0 263--285, 1970.

\bibitem[Miles(1984)]{miles_sectional}
R.~E. Miles.
\newblock Sectional {V}oronoi tessellations.
\newblock \emph{Rev. Uni{\`o}n Mat. Argentina}, 29:\penalty0 310--327, 1984.

\bibitem[{Miles}(1971)]{miles}
R.E. {Miles}.
\newblock {Isotropic random simplices.}
\newblock \emph{{Adv. Appl. Probab.}}, 3:\penalty0 353--382, 1971.
\newblock \doi{10.2307/1426176}.

\bibitem[M{\o}ller(1989)]{moller}
J.~M{\o}ller.
\newblock Random tessellations in {${\bf R}^d$}.
\newblock \emph{Adv. Appl. Probab.}, 21\penalty0 (1):\penalty0 37--73, 1989.
\newblock \doi{10.2307/1427197}.
\newblock URL \url{https://doi.org/10.2307/1427197}.

\bibitem[M{\o}ller(1994)]{moller_book}
J.~M{\o}ller.
\newblock \emph{Lectures on random {V}orono\u{\i} tessellations}, volume~87 of
  \emph{Lecture Notes in Statistics}.
\newblock Springer-Verlag, New York, 1994.
\newblock \doi{10.1007/978-1-4612-2652-9}.
\newblock URL \url{https://doi.org/10.1007/978-1-4612-2652-9}.

\bibitem[Okabe et~al.(2000)Okabe, Boots, Sugihara, and Chiu]{okabe_etal_book}
A.~Okabe, B.~Boots, K.~Sugihara, and S.~N. Chiu.
\newblock \emph{Spatial tessellations: concepts and applications of {V}oronoi
  diagrams}.
\newblock Wiley Series in Probability and Statistics. John Wiley \& Sons, Ltd.,
  Chichester, second edition, 2000.
\newblock \doi{10.1002/9780470317013}.
\newblock URL \url{https://doi.org/10.1002/9780470317013}.
\newblock With a foreword by D. G. Kendall.

\bibitem[Perles and Shephard(1967)]{perles_shephard}
M.~A. Perles and G.~C. Shephard.
\newblock Angle sums of convex polytopes.
\newblock \emph{Math. Scand.}, 21:\penalty0 199--218 (1969), 1967.
\newblock \doi{10.7146/math.scand.a-10860}.
\newblock URL \url{https://doi.org/10.7146/math.scand.a-10860}.

\bibitem[Peschl(1955)]{peschl}
E.~Peschl.
\newblock Winkelrelationen am {S}implex und die {E}ulersche {C}harakteristik.
\newblock \emph{Bayer. Akad. Wiss. Math.-Nat. Kl. S.-B.}, 1955:\penalty0
  319--345 (1956), 1955.

\bibitem[Poincar{\'e}(1905)]{poincare}
H.~Poincar{\'e}.
\newblock Sur la generalization d'un theoreme elementaire de {G}eometrie.
\newblock \emph{Compt. Rend. Acad. Sci. Paris}, 140:\penalty0 313--117, 1905.

\bibitem[Reitzner(2005)]{ReitznerCombinatorialStructure}
M.~Reitzner.
\newblock The combinatorial structure of random polytopes.
\newblock \emph{Adv. Math.}, 191\penalty0 (1):\penalty0 178--208, 2005.
\newblock \doi{10.1016/j.aim.2004.03.006}.

\bibitem[{R\'enyi} and {Sulanke}(1963)]{renyi_sulanke1}
A.~{R\'enyi} and R.~{Sulanke}.
\newblock {\"Uber die konvexe H\"ulle von $n$ zuf\"allig gew\"ahlten Punkten.}
\newblock \emph{{Z. Wahrscheinlichkeitstheor. Verw. Geb.}}, 2:\penalty0 75--84,
  1963.
\newblock \doi{10.1007/BF00535300}.

\bibitem[Rogers(1961)]{rogers}
C.~A. Rogers.
\newblock An asymptotic expansion for certain {S}chl\"afli functions.
\newblock \emph{J. London Math. Soc.}, 36:\penalty0 78--80, 1961.
\newblock \doi{10.1112/jlms/s1-36.1.78}.

\bibitem[Ruben(1960)]{ruben}
H.~Ruben.
\newblock On the geometrical moments of skew-regular simplices in
  hyperspherical space, with some applications in geometry and mathematical
  statistics.
\newblock \emph{Acta Math.}, 103:\penalty0 1--23, 1960.

\bibitem[{Ruben} and {Miles}(1980)]{ruben_miles}
H.~{Ruben} and R.E. {Miles}.
\newblock {A canonical decomposition of the probability measure of sets of
  isotropic random points in $\mathbb R^n$.}
\newblock \emph{{J. Multivariate Anal.}}, 10:\penalty0 1--18, 1980.
\newblock \doi{10.1016/0047-259X(80)90077-9}.

\bibitem[Schl{\"a}fli(1950)]{schlaefli_vielfache_kont}
L.~Schl{\"a}fli.
\newblock Theorie der vielfachen {K}ontinuit{\"a}t.
\newblock In \emph{Gesammelte Mathematische Abhandlungen}, pages 167--387.
  Springer, 1950.

\bibitem[Schneider(2008)]{schneider_polytopes}
R.~Schneider.
\newblock Recent results on random polytopes.
\newblock \emph{Boll. Unione Mat. Ital. (9)}, 1\penalty0 (1):\penalty0 17--39,
  2008.

\bibitem[Schneider and Weil(2008)]{SW08}
R.~Schneider and W.~Weil.
\newblock \emph{Stochastic and {I}ntegral {G}eometry}.
\newblock Probability and its Applications. Springer--Verlag, Berlin, 2008.

\bibitem[Sloane~(editor)()]{sloane}
N.~J.~A. Sloane~(editor).
\newblock {The {O}n-{L}ine {E}ncyclopedia of {I}nteger {S}equences}.
\newblock https://oeis.org.

\bibitem[{Vershik} and {Sporyshev}(1992)]{vershik_sporyshev}
A.M. {Vershik} and P.V. {Sporyshev}.
\newblock {Asymptotic behavior of the number of faces of random polyhedra and
  the neighborliness problem.}
\newblock \emph{{Sel. Math. Sov.}}, 11\penalty0 (2):\penalty0 181--201, 1992.

\bibitem[Whittaker and Watson(1962)]{whittaker_watson_book}
E.~T. Whittaker and G.~N. Watson.
\newblock \emph{A course of modern analysis. {A}n introduction to the general
  theory of infinite processes and of analytic functions: with an account of
  the principal transcendental functions}.
\newblock Fourth edition. Reprinted. Cambridge University Press, New York,
  1962.

\end{thebibliography}
\bibliographystyle{plainnat}

\end{document}